\documentclass[letterpaper, 11pt, thm-restate]{article}

\usepackage[margin=1in]{geometry}

\usepackage{footmisc}
\usepackage{enumitem}
\usepackage[title]{appendix}
\usepackage{xcolor}
\usepackage{extarrows}
\usepackage{amsthm}
\usepackage{amsmath}
\usepackage{amssymb}
\usepackage{amsfonts}
\usepackage{graphicx}
\usepackage[ruled]{algorithm2e}
\usepackage{thm-restate}
\usepackage[colorlinks = true]{hyperref}
\hypersetup{
    linkcolor=black,
    citecolor=blue,
    urlcolor=blue}

\usepackage{mathtools}

\usepackage{titlesec}
\titleformat{\subsection}[runin]
       {\normalfont\bfseries}
       {\thesubsection}
       {0.5em}
       {}
       [.]

\allowdisplaybreaks

\newcommand{\Z}{\mathbb{Z}}
\newcommand{\N}{\mathbb{N}}
\newcommand{\Q}{\mathbb{Q}}
\newcommand{\R}{\mathbb{R}}

\newcommand{\K}{\mathbb{K}}
\newcommand{\A}{\mathbb{A}}
\newcommand{\F}{\mathbb{F}}
\newcommand{\Qbar}{\overline{\mathbb{Q}}}
\newcommand{\Rp}{\mathbb{R}_{\geq 0}}
\newcommand{\Rpp}{\mathbb{R}_{>0}}

\newcommand{\Qpp}{\mathbb{Q}_{>0}}

\newcommand{\Zpp}{\mathbb{Z}_{>0}}
\newcommand{\GL}{\mathsf{GL}}
\newcommand{\SL}{\mathsf{SL}}
\newcommand{\T}{\mathsf{T}}
\newcommand{\UT}{\mathsf{UT}}
\newcommand{\gen}[1]{\left\langle{#1}\right\rangle}
\newcommand{\Rns}{\left(\R^n\right)^*}
\newcommand{\Rsns}{\left(\R^{s+n}\right)^*}
\newcommand{\bff}{\boldsymbol{f}}

\newcommand{\bg}{\boldsymbol{g}}
\newcommand{\bh}{\boldsymbol{h}}
\newcommand{\balpha}{\boldsymbol{\alpha}}
\newcommand{\mG}{\mathcal{G}}
\newcommand{\mM}{\mathcal{M}}
\newcommand{\mY}{\mathcal{Y}}
\newcommand{\mA}{\mathcal{A}}
\newcommand{\mH}{\mathcal{H}}
\newcommand{\mT}{\mathcal{T}}
\newcommand{\mP}{\mathcal{P}}
\newcommand{\mL}{\mathcal{L}}
\newcommand{\mB}{\mathcal{B}}
\newcommand{\mPC}{\mathcal{PC}}
\newcommand{\hG}{\widehat{\Gamma}}
\newcommand{\tn}{\widetilde{n}}
\newcommand{\ta}{\widetilde{a}}
\newcommand{\tc}{\widetilde{c}}
\newcommand{\tf}{\widetilde{f}}
\newcommand{\tbf}{\widetilde{\boldsymbol{f}}}
\newcommand{\tMZ}{\widetilde{\mM_{\Z}}}

\newcommand{\oX}{\mkern 1.5mu\overline{\mkern-1.5mu X \mkern-1.5mu}\mkern 1.5mu}

\newcommand{\oS}{\overline{S}}

\newcommand{\ApK}{\left(\A^+\right)^K}
\newcommand{\init}{\operatorname{in}}
\newcommand{\ev}{\operatorname{ev}}

\newcommand{\conv}{\operatorname{conv}}
\newcommand{\dist}{\operatorname{dist}}
\newcommand{\diam}{\operatorname{diam}}

\DeclareMathOperator{\rk}{rk}
\newcommand{\la}{\left\langle}
\newcommand{\ra}{\right\rangle}
\newcommand{\say}[1]{``#1"}
\newcommand{\grp}{\mathrm{grp}}
\usepackage{easy-todo}

\newtheorem{theorem}{Theorem}[section]

\newtheorem{proposition}[theorem]{Proposition}
\newtheorem{lemma}[theorem]{Lemma}
\newtheorem{corollary}[theorem]{Corollary}

\theoremstyle{definition}
\newtheorem{definition}[theorem]{Definition}
\newtheorem{fact}[theorem]{Fact}
\newtheorem{remark}[theorem]{Remark}
\newtheorem{example}[theorem]{Example}

\newcounter{ProblemCounter}

\begin{document}

\title{The Identity Problem in virtually solvable matrix groups over algebraic numbers}

\author{
  Corentin Bodart\footnote{Mathematical Institute, University of Oxford, UK, Email: \url{cobodart123@gmail.com} \\ C.B. was partially supported by the Swiss NSF grant 200020-200400.}
  \and
  Ruiwen Dong\footnote{Department of Mathematics, Saarland University, Germany. Magdalen College, University of Oxford, UK. Email: \url{ruiwen.dong@magd.ox.ac.uk}. R.D. was partially supported by the ERC AdG grant 101097307.}
}

\date{}
\maketitle
\thispagestyle{empty}

\begin{abstract}
    The \emph{Tits alternative} states that a finitely generated matrix group either contains a nonabelian free subgroup $F_2$, or it is virtually solvable.
    This paper considers two decision problems in virtually solvable matrix groups: the \emph{Identity Problem} (does a given finitely generated subsemigroup contain the identity matrix?), and the \emph{Group Problem} (is a given finitely generated subsemigroup a group?).
    We show that both problems are decidable in virtually solvable matrix groups over the field of algebraic numbers $\overline{\mathbb{Q}}$.
    Our proof also extends the decidability result for nilpotent groups by Bodart, Ciobanu, Metcalfe and Shaffrir, and the decidability result for metabelian groups by Dong (STOC'24).
    Since the Identity Problem and the Group Problem are known to be undecidable in matrix groups containing $F_2 \times F_2$, our result significantly reduces the decidability gap for both decision problems.
\end{abstract}

\vspace{0.5cm}


\newpage
\setcounter{page}{1}

\section{Introduction}
\subsection{Algorithmic problems in matrix semigroups}
The computational theory of matrix groups and semigroups is one of the oldest and most well-developed parts of computational algebra. 
In the seminal work of Markov from the 1940s~\cite{markov1947certain}, the \emph{Semigroup Membership} problem was shown to be undecidable for integer matrices of dimension six.
This marked one of the first undecidability results obtained outside of mathematical logic and the theory of computing.
The input of Semigroup Membership is a finite set of square matrices $\{A_1, \ldots, A_m\}$ and a matrix $T$, and output is whether or not there exist an integer $p \geq 1$ and a sequence $A_{i_1}, A_{i_2}, \cdots, A_{i_p}$, such that $A_{i_1} A_{i_2} \cdots A_{i_p} = T$.
In the 1950s, Mikhailova~\cite{mikhailova1958occurrence} similarly introduced the \emph{Group Membership} problem.
For both problems, we work in some fixed matrix group $G$.
For effectiveness reasons\footnote{Elements of the field of algebraic numbers $\Qbar$ can be effectively represented and computed~\cite{cohen2013course}.}, we often suppose $G$ to be a subgroup of $\GL(d, \Qbar)$ for some $d \geq 1$ (the group of $d \times d$ invertible matrices over algebraic numbers).
Given a set $X \subseteq G$, denote by $\gen{X}$ the semigroup generated by $X$, and by $\gen{X}_{\grp}$ the group generated by $X$.
Then the two above decision problems can be formulated as follows.
\begin{enumerate}[nosep, label = (\roman*)]
    \item \textit{(Semigroup Membership)} given $A_1, \ldots, A_m, T \in G$, decide whether $T \in \gen{A_1, \ldots, A_m}$.
    \item \textit{(Group Membership)} given $A_1, \ldots, A_m, T \in G$, decide whether $T \in \gen{A_1, \ldots, A_m}_{\grp}$.
    \setcounter{ProblemCounter}{\value{enumi}}
\end{enumerate}
In this paper, we consider two closely related problems introduced by Choffrut and Karhum\"{a}ki in the 2000s~\cite{choffrut2005some}.
Let $I$ denote the identity matrix.
\begin{enumerate}[nosep,label = (\roman*)]
    \setcounter{enumi}{\value{ProblemCounter}}
    \item \textit{(Identity Problem)} given $A_1, \ldots, A_m \in G$, decide whether $I \in \gen{A_1, \ldots, A_m}$.
    \item \textit{(Group Problem)} given $A_1, \ldots, A_m \in G$, decide whether $\gen{A_1, \ldots, A_m}$ is a group.
    \setcounter{ProblemCounter}{\value{enumi}}
\end{enumerate}
These decision problems concern the \emph{structure} of a semigroup rather than its \emph{membership.}
See for example~\cite{bell2010undecidability, bell2017identity, ko2017identity} for earlier developments on the Identity Problem and the Group Problem, and see~\cite{dong2023recent} for a survey on recent progress.
Note that decidability of the Group Problem subsumes decidability of the Identity Problem and the \emph{Inverse Problem} (decide whether $A_1^{-1} \in \gen{A_1, \ldots, A_m}$)~\cite{dong2024semigroup}: these are the essential special cases of Semigroup Membership.

All four algorithmic problems are undecidable when $G$ contains as a subgroup a direct product of two free groups $F_2 \times F_2$~\cite{mikhailova1958occurrence, bell2010undecidability}, for example when $G = \SL(4, \Z)$ (the group of $4 \times 4$ integer matrices with determinant one).
This motivates us to study groups on the other end of the spectrum: namely when $G$ does not contain $F_2$ as a subgroup.
By the celebrated \emph{Tits alternative}~\cite{TITS1972250}, a finitely generated subgroup of $\GL(d, \Qbar)$ either contains $F_2$ as a subgroup, or it is \emph{virtually solvable}.
This paper will focus on the latter case.

\subsection{Virtually solvable matrix groups over $\Qbar$}
We briefly recall the definitions of solvable, virtually solvable, metabelian, and nilpotent groups. 
Given a group $G$ and its subgroup $H$, denote by $[G, H]$ the group generated by the elements $ghg^{-1} h^{-1}, g \in G, h \in H$.
A group $G$ is \emph{solvable} if its \emph{derived series} $G = G^{(0)} \geq G^{(1)} \geq \cdots$ , defined by $G^{(i+1)} = [G^{(i)}, G^{(i)}]$, reaches the trivial group in a finite number of steps. 
A group is \emph{virtually solvable} if it admits a finite index subgroup that is solvable. In particular, solvable groups are also virtually solvable.
Metabelian groups and nilpotent groups are special cases of solvable groups.
A group $G$ is called \emph{metabelian} if $G^{(2)}$ is the trivial group, or equivalently, if $G$ admits an abelian normal subgroup $A$ such that the quotient $G/A$ is abelian.
A group $G$ is \emph{nilpotent} if its \emph{lower central series} $G = G_{0} \geq G_{1} \geq \cdots$, defined by $G_{i+1} = [G, G_i]$, reaches the trivial group in a finite number of steps.

A classic result of Kopytov~\cite{kopytov1968solvability} showed that Group Membership is decidable in virtually solvable matrix groups over $\Qbar$. 
Semigroup Membership is proven to be undecidable for some instances of solvable matrix groups, such as large direct powers of the Heisenberg group $\mathsf{H}_3(\Z)$~\cite{roman2022undecidability}.
Conversely, it has been shown to be decidable for other instances, such as commutative matrix groups over $\Qbar$~\cite{babai1996multiplicative}, the Heisenberg groups $\mathsf{H}_{2n+1}(\Q)$~\cite{colcombet2019reachability}, the Baumslag-Solitar groups $\mathsf{BS}(1, q)$~\cite{DBLP:conf/icalp/CadilhacCZ20}, and the lamplighter groups $(\Z/n\Z) \wr \Z$~\cite{lohrey2015rational}~\footnote{Unlike the other examples, the group $(\Z/n\Z) \wr \Z$ is a matrix group over $(\Z/n\Z)[X, X^{-1}]$ instead of $\Qbar$.}.
Decidability of the Identity Problem and the Group Problem was open for virtually solvable matrix groups over $\Qbar$.
Nevertheless, they were recently proven to be decidable in the case of nilpotent groups~\cite{bodart2024identity, shafrir2024saturation} and in the case of metabelian groups~\cite{dong2024semigroup}.
In this paper, we completely answer this open problem by proving decidability of the Identity Problem and the Group Problem in virtually solvable matrix groups over $\Qbar$.

Note that in the definition of both problems, the input is a set of matrices $\{A_1, \ldots, A_m\}$ in some virtually solvable group $G \leq \GL(d, \Qbar)$.
It is not important whether $G$ is given as a part of the input, as we can always suppose $G$ to be the group generated by $\{A_1, \ldots, A_m\}$.
This does not change the fact that $G$ is virtually solvable, since finitely generated subgroups of virtually solvable groups are still virtually solvable~\cite[Chapter~13]{dructu2018geometric}.

\begin{restatable}{theorem}{thmmain}\label{thm:main}
    The Identity Problem and the Group Problem are decidable in virtually solvable subgroups of $\GL(d, \Qbar)$. 
    That is, given matrices $A_1, \ldots, A_m \in \GL(d, \Qbar)$ that generate a virtually solvable group\footnote{Note that given $A_1, \ldots, A_m \in \GL(d, \Qbar)$, it is decidable whether $\gen{A_1, \ldots, A_m}_{\grp}$ is virtually solvable~\cite{beals1999algorithms}.}, it is decidable whether $I \in \gen{A_1, \ldots, A_m}$, and whether $\gen{A_1, \ldots, A_m}$ is a group.
\end{restatable}


\subsection{Related work and our contributions}
Two significant results on the Identity Problem are its decidability in nilpotent groups, due to Shaffrir~\cite{shafrir2024saturation} (and independently by Bodart, Ciobanu, Metcalfe~\cite{bodart2024identity}), as well as its decidability in metabelian groups, due to Dong~\cite{dong2024semigroup}.
Our paper generalizes these two results, and then uses their generalization as crucial lemmas in our solution for virtually solvable matrix groups over $\Qbar$.
In particular:

\begin{enumerate}[nosep, label = (\arabic*), wide]
    \item We extend a key theorem in~\cite{bodart2024identity, shafrir2024saturation} from \emph{finitely generated} nilpotent groups to \emph{infinitely generated} nilpotent groups of finite Pr\"{u}fer rank (Theorem~\ref{thm:nilpsat}, proof in Section~\ref{sec:nilp}).
    This applies to groups of \emph{unitriangular} matrices over an algebraic number field $\K$.
    \item We generalize the result in~\cite{dong2024semigroup} from finitely generated subsemigroups of metabelian groups to their \emph{rational subsemigroups} (Theorem~\ref{thm:ratmeta}, proof in Section~\ref{sec:meta}). To this end, we combine automata theory with the concepts developed in Dong's paper, and introduce new techniques such as \emph{$\mA$-graphs} and \emph{partial contractions}.
    We connect automata over metabelian groups to graphs over lattices, convex polytopes and algebraic geometry, and prove several theorems under this new context (Theorem~\ref{thm:acctocon}, Lemma~\ref{lem:acctoeq}, Theorem~\ref{thm:dec}).
    \item We then reduce the Group Problem in a virtually solvable matrix group $G$ to deciding whether a rational subsemigroup of a metabelian group is a group (Section~\ref{sec:overview}).
    This reduction is done by constructing a metabelian quotient $T/[N,N]$ of a suitable triangularizable subgroup $T$ of $G$. The normal subgroup $[N,N] \trianglelefteq T$ is the commutator of an \emph{infinitely} generated nilpotent group $N$, hence the extension in (1) is needed. By reducing from $G$ to $T/[N,N]$, we inevitably pass from finitely generated subsemigroups of $G$ to \emph{rational} subsemigroups of $T/[N,N]$. Hence the generalization in (2) is needed.
\end{enumerate}

\section{Prelimilaries}
\subsection{Rational subsemigroups and automata over groups}

Let $S$ be a set of elements in $G$. The semigroup $\gen{S}$ generated by $S$ is defined as the set of non-empty products of elements in $S$, that is, $\gen{S} \coloneqq \{g_1 g_2 \cdots g_p \mid p \geq 1, g_1, g_2, \ldots, g_p \in S\}$.

\begin{restatable}[{\cite[Lemma~2.2]{dong2024semigroup}}]{lemma}{lemsubsume}\label{lem:subsume}
Let $X$ be a finite set of elements in a group $G$.
The semigroup $\gen{X}$ contains the neutral element $I$ if and only if there is a non-empty subset $Y \subseteq X$ such that $\gen{Y}$ is a group.
Hence, if the Group Problem is decidable in $G$, then the Identity Problem is also decidable by testing the Group Problem on all possible subsets of the input.
\end{restatable}

In this paper, we will need the more general notion of \emph{rational semigroups}, which are recognized by \emph{automata over groups}.
See~\cite{lohrey2013rational} for a general reference on rational subsets of groups.
Given a group $G$, an \emph{automaton} over $G$ is defined by a set of $s$ \emph{states} $q_1, \ldots, q_s,$ and a set of $t$ \emph{transitions} $\delta_1, \ldots, \delta_t$.
For each $\ell = 1, \ldots, t,$ the transition $\delta_{\ell}$ has an \emph{origin} state which we denote by $q_{\Omega(\ell)}$, a \emph{destination} state which we denote by $q_{\Delta(\ell)}$, as well as an \emph{evaluation} in $G$ which we denote by $\ev(\delta_{\ell}) \in G$.
We call $q_1$ both the \emph{initial state} and the \emph{accepting state}, that is, we only consider automata whose initial state is the same as the accepting state.

A \emph{path} in $\mA$ is a non-empty sequence of transitions $w = \delta_{i_1} \delta_{i_2} \cdots \delta_{i_m}$, such that $\Delta(i_1) = \Omega(i_2), \Delta(i_2) = \Omega(i_3), \ldots, \Delta(i_{m-1}) = \Omega(i_m)$.
The \emph{evaluation} of such a path is defined as $\ev(w) \coloneqq \ev(\delta_{i_1}) \ev(\delta_{i_2}) \cdots \ev(\delta_{i_m}) \in G$.
A path is called an \emph{accepting run} if additionally $\Omega(i_1) = \Delta(i_m) = 1$.
Let $\ev(\mA) \coloneqq \{\ev(w) \mid w \text{ is an accepting run of } \mA\}$, this is the set of elements of $G$ recognized by accepting runs.
Then $\ev(\mA)$ is a subsemigroup of $G$ because, if $w$ and $w'$ are accepting runs of $\mA$, then the concatenation $ww'$ is also an accepting run, and $\ev(ww') = \ev(w)\ev(w')$.
However, $\ev(\mA)$ does not necessarily contain the neutral element of $G$, since the empty sequence is not considered as an accepting run.
We say that the automaton $\mA$ \emph{recognizes} the semigroup $\ev(\mA)$.
A subset $S \subseteq G$ is called a \emph{rational semigroup} if there is some automaton $\mA$ over $G$ that recognizes $S$.

An automaton $\mA$ is called \emph{trim} if every state is the origin of some transition and every transition appears in some accepting run. Every rational subsemigroup of $G$ is recognized by a trim automaton over $G$ by removing states unreachable from $q_1$ and states from which $q_1$ is unreachable.

If $\mA$ contains only one state $q_1$, then all its transitions are loops, so $\ev(\mA)$ is the semigroup generated by $\ev(\delta_1), \ldots, \ev(\delta_t)$, and we recover the definition of finitely generated semigroups.

\subsection{Triangular matrix groups}
Let $\K$ be any field.
Denote respectively by $\T(d, \K)$ and $\UT(d, \K)$ the group of $d \times d$ upper-triangular invertible matrices over $\K$ and the group of $d \times d$ upper-unitriangular matrices over $\K$:
\[
\T(d, \K) \coloneqq 
\left\{
\begin{pmatrix}
a_1 & * & \cdots & * \\
0 & a_2 & \cdots & * \\
\vdots & \vdots & \ddots & \vdots \\
0 & 0 & \cdots & a_d \\
\end{pmatrix}
,\; a_1 a_2 \cdots a_d \neq 0 \right\}
, \quad
\UT(d, \K) \coloneqq 
\left\{
\begin{pmatrix}
1 & * & \cdots & * \\
0 & 1 & \cdots & * \\
\vdots & \vdots & \ddots & \vdots \\
0 & 0 & \cdots & 1 \\
\end{pmatrix}
\right\},
\]
where $a_i$ and $*$ denote arbitrary entries in $\K$.
The group $\T(d, \K)$ and its subgroups are solvable; the group $\UT(d, \K)$ and its subgroups are nilpotent~\cite{dructu2018geometric}.
If $T$ is any subgroup of $\T(d, \K)$, then $N \coloneqq T \cap \UT(d, \K)$ is a normal subgroup of $T$, and the quotient $T/N$ is abelian.

Let $G$ be a finitely generated subgroup of $\GL(d, \Qbar)$.
By a classic result of Mal'cev~\cite{mal1951some}, $G$ is virtually solvable if and only if it contains a finite index normal subgroup $T$ that is conjugate to a subgroup of $\T(d, \Qbar)$.
This gave rise to algorithms that decide whether a given finitely generated matrix group over $\Qbar$ is virtually solvable:

\begin{theorem}[{\cite[Theorem~1.1]{beals1999algorithms}, \cite[Section~2.8]{ostheimer1999practical}}]\label{thm:effectiveTits}
    There is an algorithm that, given a finite number of generators for a group $G \leq \GL(d, \Qbar)$, decide whether $G$ is virtually solvable.
    Furthermore, when $G$ is virtually solvable, the algorithm computes the generators\footnote{Note that when $G$ is finitely generated, its finite index subgroups are also finitely generated~\cite[Lemma~7.86]{dructu2018geometric}.} for a finite index normal subgroup $T \trianglelefteq G$, as well as a matrix $g \in \GL(d, \Qbar)$, such that $g T g^{-1} \leq \T(n, \Qbar)$.
\end{theorem}

\subsection{Polynomial rings, modules and semidirect products}

Let $R$ be a commutative ring or semiring (such as $\Z$, $\R$, $\N$ or $\Rp$).
Denote by $R[X_1^{\pm}, \ldots, X_n^{\pm}]$ the Laurent polynomial ring or semiring over $R$ with $n$ variables: this is the set of polynomials of variables $X_1, X_1^{-1}, \ldots, X_n, X_n^{-1},$ with coefficients in $R$.
We have $X_i X_i^{-1} = 1$.
When $n$ is fixed, we denote $R[\oX^{\pm}] \coloneqq R[X_1^{\pm}, \ldots, X_n^{\pm}]$.
For a vector $a = (a_1, \ldots, a_n) \in \Z^n$, denote by $\oX^a$ the monomial $X_1^{a_1} X_2^{a_2} \cdots X_n^{a_n}$.

When $R$ is a commutative ring, an $R[\oX^{\pm}]$-module is defined as an abelian group $(M, +)$ along with an operation $\cdot \;\colon R[\oX^{\pm}] \times M \rightarrow M$ satisfying $f \cdot (m+m') = f \cdot m + f \cdot m'$, $(f + g) \cdot m = f \cdot m + g \cdot m$, $fg \cdot m = f \cdot (g \cdot m)$ and $1 \cdot m = m$.
For example, for any $d \in \N$, $R[\oX^{\pm}]^d$ is an $R[\oX^{\pm}]$-module by $f \cdot (g_1, \ldots, g_d) = (fg_1, \ldots, fg_d)$.
Throughout this paper, we use the bold symbol $\bff$ to denote a vector $(f_1, \ldots, f_d) \in R[\oX^{\pm}]^d$.

Given vectors $\bh_1, \ldots, \bh_k \in R[\oX^{\pm}]^d$, we say that they \emph{generate} the $R[\oX^{\pm}]$-module
\[ \sum_{i=1}^k R[\oX^{\pm}] \cdot \bh_i \coloneqq \left\{\sum_{i=1}^k p_i \cdot \bh_i \;\middle|\; p_1, \ldots, p_k \in R[\oX^{\pm}] \right\} \subseteq R[\oX^{\pm}]^d. \]
Given submodules $M,N$ of $R[\oX^{\pm}]^d$ such that $M\supseteq N$, we define the quotient $M/N \coloneqq \{\overline{m} \mid m \in M\}$ where $\overline{m_1} = \overline{m_2}$ if and only if $m_1 - m_2 \in N$.
This quotient is also an $R[\oX^{\pm}]$-module.
We say that an $R[\oX^{\pm}]$-module $\mY$ is \emph{finitely presented} if it can be written as a quotient $M/N$ for two submodules $M, N$ of $R[\oX^{\pm}]^d$ for some $d \in \N$, where both $M$ and $N$ are generated by finitely many elements.
We call a \emph{finite presentation} of $\mY$ the respective generators of such $M, N$.
Given a finitely presented $\Z[X_1^{\pm}, \ldots, X_n^{\pm}]$-module $\mY$, we can define a metabelian group using \emph{semidirect product}:
\begin{equation}\label{eq:defsemi}
\mY \rtimes \Z^n \coloneqq \{(y, a) \mid y \in \mY, a \in \Z^n\};
\end{equation}
multiplication and inversion in this group are defined by
\begin{equation}\label{eq:defsemi2}
(y, a) \cdot (y', a') = \big(y + \oX^a \cdot y', a + a'\big), \quad (y, a)^{-1} = \big(\! - \oX^{-a} \cdot y, -a\big).
\end{equation}
The neutral element of $\mY \rtimes \Z^n$ is $(0, 0)$.
Intuitively, the element $(y, a)$ can be seen as a $2 \times 2$ matrix
$
\begin{pmatrix}
\oX^{a} & y \\
0 & 1 \\
\end{pmatrix}
$, where group multiplication is represented by matrix multiplication.
Note that $\mY \rtimes \Z^n$ naturally contains the subgroups $\Z^n \cong \{(0, a) \mid a \in \Z^n\}$ and $\mY \cong \{(y, 0^n) \mid y \in \mY\}$.

\section{Decidability in virtually solvable matrix groups: proof overview}\label{sec:overview}

Omitted proofs can be found in Appendix~\ref{app:omitted}.
In this section we prove our main result:

{\renewcommand\footnote[1]{}\thmmain*}

By Lemma~\ref{lem:subsume}, it suffices to show decidability of the Group Problem.
Given $A_1, \ldots, A_m \in \GL(d, \Qbar)$ such that $G \coloneqq \gen{A_1, \ldots, A_m}_{\grp}$ is virtually solvable, our goal is to decide whether $\gen{A_1, \ldots, A_m}$ is a group.

By Theorem~\ref{thm:effectiveTits}, $G$ admits a finite index normal subgroup $T$ such that $g T g^{-1} \leq \T(d, \Qbar)$ for an effectively computable element $g \in \GL(d, \Qbar)$.
We can replace $G$ with $g G g^{-1}$ (that is, replace the generators $A_1, \ldots, A_m$ by their conjugates $g A_1 g^{-1}, \ldots, g A_m g^{-1}$), and thus without loss of generality suppose $T \leq \T(d, \Qbar)$.
Furthermore, a finite set of generators for $T$ is also given by Theorem~\ref{thm:effectiveTits}.
After the replacement, let $\K$ denote the field generated by all the entries of $A_1, \ldots, A_m$.
Then, $G$ and $T$ are respectively subgroups of $\GL(d, \K)$ and $\T(d, \K)$.

Deciding whether $\gen{A_1, \ldots, A_m}$ is a group is done through a series of reductions.
First, we reduce it to a deciding whether the rational semigroup $S \coloneqq \gen{A_1, \ldots, A_m} \cap T$ is a group:

\begin{restatable}{lemma}{lemvirtorat}\label{lem:virtorat}
    Let $T \leq \T(d, \K)$ be a finite index normal subgroup of $G = \gen{A_1, \ldots, A_m}_{\grp}$, given by its finite set of generators.
    Then, $\gen{A_1, \ldots, A_m}$ is a group if and only if $S \coloneqq \gen{A_1, \ldots, A_m} \cap T$ is a group.    
    Furthermore, $S$ is a rational subsemigroup of $T$, whose automaton can be effectively computed from $A_1, \ldots, A_m$ and the generators of $T$.
\end{restatable}
    

The following lemma further shows we can without loss of generality suppose $\gen{S}_{\grp} = T$:


\begin{restatable}{lemma}{lemprim}\label{lem:prim}
    Let $S$ be a rational subsemigroup of $T$, then $\gen{S}_{\grp}$ is finitely generated.
    Furthermore, given an automaton over $T$ that recognizes $S$, one can compute an automaton over $\gen{S}_{\grp}$ that recognizes $S$, as well as compute a set of generators for $\gen{S}_{\grp}$.
\end{restatable}

Lemma~\ref{lem:virtorat} and \ref{lem:prim} are proven using classic techniques from automata theory.
Lemma~\ref{lem:virtorat} shows that $\gen{A_1, \ldots, A_m}$ is a group if and only if $S$ is a group. Lemma~\ref{lem:prim} shows that we can suppose $S$ to be given by an automaton over $\gen{S}_{\grp}$ instead of $T$. In other words, we can replace $T$ by its subgroup $\gen{S}_{\grp}$, and suppose without loss of generality $\gen{S}_{\grp} = T$.
We proceed to decide whether $S$ is a group.

Let $N \coloneqq T \cap \UT(d, \K)$, then $N$ is a nilpotent normal subgroup of $T$.\footnote{Even though $T$ is finitely generated, $N$ is in general not finitely generated. For example, take $T = \gen{\begin{pmatrix}
    2 & 0 \\ 0 & 1
\end{pmatrix}, \begin{pmatrix}
    1 & 1 \\ 0 & 1
\end{pmatrix}}_{\grp}$, then $N = \left\{\begin{pmatrix}
    1 & \frac{a}{2^p} \\ 0 & 1
\end{pmatrix} \;\middle|\; a \in \Z, p \in \N \right\}$ is not finitely generated.}
Consider the normal subgroup $[N, N]$ of $N$, we obtain a descending sequence of subgroups
\begin{equation*}
T \trianglerighteq N \trianglerighteq [N, N],
\end{equation*}
where $T/N$ is finitely generated abelian, and $N/[N, N]$ is abelian.

Note that $[N, N]$ is also a normal subgroup of $T$ because, for all $t \in T, n, m \in N$, we have $t (nmn^{-1} m^{-1}) t^{-1} = (tnt^{-1}) (tmt^{-1}) (tnt^{-1})^{-1} (tmt^{-1})^{-1} \in [N, N]$.
Therefore, the quotient $T/[N, N]$ is a finitely generated metabelian group.
Our next step is to show that $S \subseteq T$ is a group if and only if its image $\oS$ under the quotient map $T \rightarrow T/[N, N]$ is a group.
In other words, we will simplify $T$ by ``modulo'' $[N, N]$, and show that this (remarkably) does not change whether $S$ is a group. 

Our first main technical theorem is a weak version of the above simplification:

\begin{restatable}{theorem}{thmnilp}\label{thm:nilpsat}
    Let $\K$ be an algebraic number field and $N$ be a subgroup of $\UT(d, \K)$.
    Let $M$ be a subsemigroup of $N$ and denote by $\overline{M}$ its image under the quotient map $N \rightarrow N/[N, N]$.
    If $\overline{M} = N/[N, N]$ (equivalently, if $M[N,N]=N$), then $M = N$.
    
    
\end{restatable}

Theorem~\ref{thm:nilpsat} is a deep generalization of~\cite[Corollary~1]{shafrir2024saturation} (see also~\cite[Proposition~19]{bodart2024identity}), which proved the case when $N$ is finitely generated.
Theorem~\ref{thm:nilpsat} relaxes this constraint, the key idea being that for an algebraic number field $\K$, subgroups of $\UT(d, \K)$ have finite \emph{Pr\"{u}fer rank}.
The proof of Theorem~\ref{thm:nilpsat} is given in Section~\ref{sec:nilp}.
We now strengthen Theorem~\ref{thm:nilpsat} from $N$ to $T$:

\begin{corollary}\label{cor:solvsat}
    Let $S$ be a subsemigroup of $T$ such that $\gen{S}_{\grp} = T$. Then $S$ is a group (i.e. $S = T$) if and only if its image $\oS$ under the quotient map $T \rightarrow T/[N, N]$ is a group (i.e. $\oS = T/[N, N]$).
\end{corollary}

\begin{proof}
    If $S = T$ then obviously $\overline S=T/[N,N]$. In the other direction, if $\overline S=T/[N,N]$ (which can be rewritten as $S[N,N]=T$), then
    \[ (S\cap N)[N,N]=N. \]
    Using Theorem \ref{thm:nilpsat} we deduce that the semigroup $M=S\cap N$ is actually equal to $N$, and therefore $S\supseteq N\supseteq  [N,N]$. It follows that $S=S[N,N]=T$.
\end{proof}
By Corollary~\ref{cor:solvsat}, 
deciding whether $S \subseteq T$ is a group boils down to deciding whether $\overline{S} \subseteq T/[N, N]$ is a group: this will be our new task.
From an automaton $\mA$ over $T$ that recognizes $S$, we can construct an automaton $\overline{\mA}$ over $T/[N, N]$ that recognizes $\overline{S}$ by projecting the transition evaluations under $T \rightarrow T/[N, N]$.
However, different elements in $T$ can have the same image in $T/[N, N]$.
To decide whether $\overline{S} \subseteq T/[N, N]$ is a group, we need to switch from this ambiguous representation of elements in $T/[N, N]$ to the following more standard way of representing metabelian groups.

\begin{restatable}[{Composition of~\cite[Lemma~2]{kopytov1971solvability}~and~\cite[Lemma~B.3]{dong2024semigroup}}]{lemma}{lemeffective}\label{lem:effective}
    Let $\K$ be an algebraic number field. Suppose we are given a finitely generated subgroup $T$ of $\T(d, \K)$, let $N \coloneqq T \cap \UT(d, \K)$.
    One can compute an embedding $\varphi \colon T/[N, N] \hookrightarrow (\mY \rtimes \Z^n)/H$, where
    \begin{enumerate}[nosep, label ={\normalfont(\roman*)}]
        \item $n \in \N$ and $\mY$ is a finitely presented $\Z[X_1^{\pm}, \ldots, X_n^{\pm}]$-module.
        \item $H$ is a subgroup of $\Z^n \leq \mY \rtimes \Z^n$, and elements of $H$ commute with all elements in $\mY \rtimes \Z^n$.
    \end{enumerate}
    In particular, given any $g \in T$, one can compute $(y, z) \in \mY \rtimes \Z^n$ such that $\varphi(g[N, N]) = (y, z)H$.
\end{restatable}


Since $\varphi$ is an embedding, the semigroup $\overline S\subseteq T/[N,N]$ is a group if and only if its image $\varphi(\overline S)\subseteq (\mY\rtimes\Z^n)/H$ is a group. In turn $\varphi(\overline S)$ is a group if and only if $\varphi(\overline S)H\subseteq \mY\rtimes\Z^n$ is a group. We went from \say{abstract} semigroups inside $T/[N,N]$ to \say{concrete} semigroups inside $\mY\rtimes\Z^n$.

\begin{restatable}{lemma}{lemaddH}\label{lem:addH}
    Let $\varphi(\overline{S})$ be a rational subsemigroup of $(\mY \rtimes \Z^n)/H$ recognized by a given automaton.
    Then one can compute an automaton over $\mY \rtimes \Z^n$ that recognizes the subsemigroup $\varphi(\overline{S})H$.
\end{restatable}

The idea behind Lemma~\ref{lem:addH} is as follows: from an automaton that recognizes $\varphi(\overline{S})$, one can construct an automaton that recognizes $\varphi(\overline{S})H$ by attaching loops at $q_1$, whose evaluations generate $H$ as a semigroup.
Now, deciding whether a rational semigroup $\overline{S} \subseteq T/[N, N]$ is a group boils down to deciding whether the rational semigroup $\varphi(\overline{S}) H \subseteq \mY \rtimes \Z^n$ is a group.

\begin{restatable}{theorem}{thmratmeta}\label{thm:ratmeta}
    Given as input a finitely presented $\Z[X_1^{\pm}, \ldots, X_n^{\pm}]$-module $\mY$ and an automaton $\mA$ over $\mY \rtimes \Z^n$, it is decidable whether $\ev(\mA)$ is a group.
\end{restatable}

Theorem~\ref{thm:ratmeta} is highly non-trivial and is our second main technical contribution.
Its proof is given in Section~\ref{sec:meta}.
Theorem~\ref{thm:ratmeta} is a generalization of~\cite[Theorem~1.1]{dong2024semigroup}, which proves the decidability result for \emph{finitely generated} subsemigroups of $\mY \rtimes \Z^n$.

Summarizing all the above results, we obtain a proof of Theorem~\ref{thm:main}:
\begin{proof}[Proof of Theorem~\ref{thm:main}]
    By Lemma~\ref{lem:subsume}, it suffices to decide the Group Problem.
    Conjugating $A_1, \ldots, A_m$ by the element $g$ computed in Theorem~\ref{thm:effectiveTits}, we can suppose $G \coloneqq \gen{A_1, \ldots, A_m}_{\grp}$ admits a finite index normal subgroup $T \leq \T(d, \K)$ for some algebraic number field $\K$.
    Then, Lemma~\ref{lem:virtorat} shows that $\gen{A_1, \ldots, A_m}$ is a group if and only if the rational semigroup $S = \gen{A_1, \ldots, A_m} \cap T$ is a group, and Lemma~\ref{lem:prim} shows that we can replace $T$ with its subgroup $\gen{S}_{\grp}$.
    Then, by Corollary~\ref{cor:solvsat}, Lemma~\ref{lem:effective} and Lemma~\ref{lem:addH}, 
    the rational semigroup $S$ is a group if and only if $\varphi(\oS)H$ is a group.
    Furthermore, one can construct an automaton $\mA$ over $\mY \rtimes \Z^n$ that recognizes $\varphi(\oS)H$.
    It is then decidable by Theorem~\ref{thm:ratmeta} whether $\varphi(\oS)H$ is a group.
\end{proof}

\section{Rational subsemigroups of metabelian groups}\label{sec:meta}
In this section we prove Theorem~\ref{thm:ratmeta}.
Suppose the automaton $\mA$ over $\mY \rtimes \Z^n$ has states $q_1, \ldots, q_s$ and transitions $\delta_1, \ldots, \delta_t$.
Without loss of generality we can suppose $\mA$ to be trim.
For each transition $\delta_{\ell}, \ell = 1, \ldots, t$, denote its evaluation by $(y_{\ell}, a_{\ell}) \in \mY \rtimes \Z^n$.
Our first step of deciding whether $\ev(\mA)$ is a group is to reduce it to finding an \emph{Identity Traversal} of a \emph{primitive automaton}.

Given an automaton $\mA$ over $\mY \rtimes \Z^n$, we define a new automaton $\mA^{\pm}$ as follows: the states of $\mA^{\pm}$ are the same as $\mA$, and the transitions of $\mA^{\pm}$ are $\delta_1, \ldots, \delta_t, \delta_1^{-}, \ldots, \delta_t^{-}$.
Here, $\delta_{\ell}^{-}$ is a transition from the destination of $\delta_{\ell}$ to the origin of $\delta_{\ell}$, with evaluation $\ev(\delta_{\ell}^{-}) \coloneqq (y_{\ell}, a_{\ell})^{-1}$.
We say that $\mA$ is \emph{primitive}, if the image of $\ev(\mA^{\pm})$ under the projection $\mY \rtimes \Z^n \rightarrow \Z^n$ is $\Z^n$.

\begin{restatable}{lemma}{lemprimitive}\label{lem:primitive}
    Suppose we are given a trim automaton $\mA$ over $\mY \rtimes \Z^n$.
    One can compute $\tn \in \N$, a finitely presented $\Z[\widetilde{X_1}^{\pm}, \ldots, \widetilde{X_{\tn}}^{\pm}]$-module $\widetilde{\mY}$, and a primitive trim automaton $\widetilde{\mA}$ over $\widetilde{\mY} \rtimes \Z^{\tn}$, such that $\ev(\mA)$ is a group if and only if $\ev(\widetilde{\mA})$ is a group.
\end{restatable}

By Lemma~\ref{lem:primitive}, throughout this section we can without loss of generality suppose $\mA$ to be trim and primitive by replacing it with $\widetilde{\mA}$.
We define an \emph{accepting traversal} of $\mA$ to be an accepting run that uses every transition at least once.
We define an \emph{Identity Traversal} of $\mA$ to be an accepting traversal whose evaluation is the neutral element $(0, 0^n) \in \mY \rtimes \Z^n$.

\begin{restatable}{proposition}{proptraverse}\label{prop:grptotraverse}
    The semigroup $\ev(\mA)$ is a group if and only if $\mA$ admits an Identity Traversal.
\end{restatable}

\subsection{From Identity Traversal to Eulerian and face-accessible $\mA$-graphs}\label{subsec:ITtoAG}

Generalizing the notion of $\mG$-graphs from~\cite{dong2024semigroup}, we define the notion of an \emph{$\mA$-graph}.
For each state $q_i$ of $\mA$, we assign to it a lattice $\Lambda_i \cong \Z^n$.
We consider the disjoint union $\Lambda \coloneqq \Lambda_1 \sqcup \cdots \sqcup \Lambda_s$ as a subset of a larger lattice $\Z^{s+n}$ in the following way.
Let $(b_1, 0^n), \ldots, (b_s, 0^n), (0^s, d_1), \ldots, (0^s, d_n)$ be the natural basis of $\Z^{s+n}$.
Here, $b_i \in \Z^s$ is the vector with $1$ on the $i$-th coordinate and $0$ elsewhere, and $d_j \in \Z^n$ is the vector with $1$ on the $j$-th coordinate and $0$ elsewhere.
For $i = 1, \ldots, s$, we identify $\Lambda_i$ with $\{b_i\} \times \Z^n$ by the map $z \mapsto (b_i, z)$.
See Figure~\ref{fig:3D} for an illustration.
Thus, a vertex $v$ in $\Lambda \subset \Z^{s+n}$ is always denoted by a pair $(b_i, z) \in \{b_1, \ldots, b_s\} \times \Z^n$, where the index $i$ signifies that $v$ is in the lattice $\Lambda_i$, and $z$ is the coordinate of $v$ within $\Lambda_i \cong \Z^n$.

\begin{figure}[ht!]
    \centering
    \begin{minipage}[t]{.31\textwidth}
        \centering
        \includegraphics[width=\textwidth,height=1.0\textheight,keepaspectratio, trim={7cm 0cm 3cm 0cm},clip]{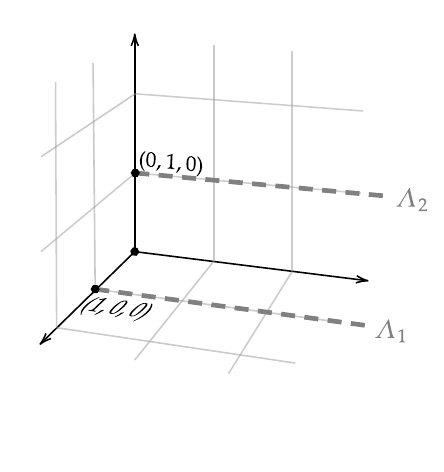}
        \caption{The disjoint union $\Lambda = \Lambda_1 \sqcup \Lambda_2$ as a subset of $\Z^{s+n}$, where $s = 2, n = 1$.}
        \label{fig:3D}
    \end{minipage}
    \hfill
    \begin{minipage}[t]{.32\textwidth}
        \centering
        \includegraphics[width=1\textwidth,height=1.0\textheight,keepaspectratio, trim={6.5cm 0.5cm 5cm 0.5cm},clip]{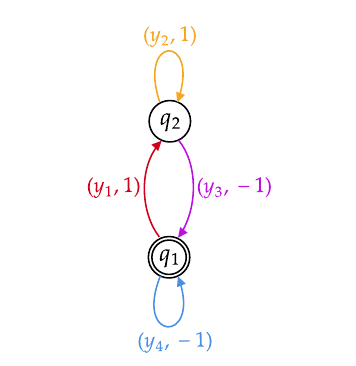}
        \caption{An automaton $\mA$, where 
        $a_1 = 1$, $a_2 = 1$, $a_3 = -1$, $a_4 = -1$.
        }
        \label{fig:A}
    \end{minipage}
    \hfill
    \begin{minipage}[t]{0.30\textwidth}
        \centering
        \includegraphics[width=1\textwidth,height=1.0\textheight,keepaspectratio, trim={7.4cm -1.5cm 3.7cm 0cm},clip]{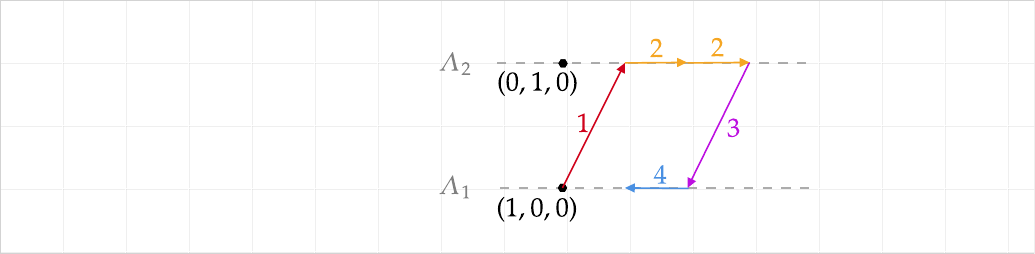}
        \caption{The $\mA$-graph $\Gamma(w)$ associated to the accepting run $w = \delta_1 \delta_2 \delta_2 \delta_3 \delta_4$.}
        \label{fig:Gw}
    \end{minipage}
\end{figure}

\begin{definition}[$\mA$-graphs]
An \emph{$\mA$-graph} is a directed multigraph $\Gamma$, whose set of vertices is a finite subset of $\Lambda$. The edges of $\Gamma$ are each labeled with an index in $\{1, \ldots, t\}$.
Furthermore, each edge with label $\ell$ connects from some vertex $(b_{\Omega(\ell)}, z), z \in \Z^n$ to the vertex $(b_{\Delta(\ell)}, z + a_{\ell})$.
Recall that $\Omega(\ell)$ and $\Delta(\ell)$ are the indices of the origin and destination states of the transition $\delta_{\ell}$, and $a_{\ell} \in \Z^n$ is such that $\ev(\delta_{\ell}) = (y_{\ell}, a_{\ell})$.
In other words, an edge with label $\ell$ connects from $\Lambda_{\Omega(\ell)}$ to $\Lambda_{\Delta(\ell)}$,
the $\Z^n$-coordinates of the target and source of the edge differ by $a_{\ell}$.\footnote{When $s = 1$, i.e.\ when $\ev(\mA)$ is finitely generated, $\mA$-graphs are equivalent to the $\mG$-graphs defined in~\cite{dong2024semigroup}.}
\end{definition}

For an $\mA$-graph $\Gamma$, denote by $V(\Gamma)$ its set of vertices, and by $E(\Gamma)$ its set of edges.
For an edge $e \in E(\Gamma)$, we denote its source vertex by $\sigma(e)$ and its target vertex by $\tau(e)$.

For an accepting run $w$ of $\mA$, we associate to it a unique $\mA$-graph $\Gamma(w)$, defined as follows.
Write $w = \delta_{\ell_1} \cdots \delta_{\ell_m}$. For each $j = 1, \ldots, m$, we add an edge starting at the vertex $(b_{\Omega(\ell_j)}, a_{\ell_1} + \cdots + a_{\ell_{j-1}})$, ending at the vertex $(b_{\Delta(\ell_j)}, a_{\ell_1} + \cdots + a_{\ell_{j}})$, with the label $\ell_j$. For $j = 1$, the edge starts at $(b_1, 0^n)$. See Figure~\ref{fig:A} and \ref{fig:Gw} for an illustration.

\begin{definition}[Element represented by an $\mA$-graph]
For an edge $e$ in an $\mA$-graph $\Gamma$, denote by $\lambda(e)$ its label.
Let $\pi_{\Z^n} \colon \Z^{s+n} \rightarrow \Z^n$ denote the projection $(b, z) \mapsto z$.
The element in $\mY$ \emph{represented} by an edge $e$ is defined as $\oX^{\pi_{\Z^n}(\sigma(e))} \cdot y_{\lambda(e)}$. (For example in Figure~\ref{fig:Gw}, the edge with label $4$ represents the element $X_1^2 \cdot y_4 \in \mY$.)
The element \emph{represented} by an $\mA$-graph is defined as $\sum_{e \in E(\Gamma)} \oX^{\pi_{\Z^n}(\sigma(e))} y_{\lambda(e)}$, the sum of all the elements represented by its edges.
By direct computation, if $w$ is an accepting run of $\mA$ with evaluation $(y, z) \in \mY \rtimes \Z^n$, then $y$ is the element represented by the associated graph $\Gamma(w)$, and $z$ is the sum $\sum_{e \in E(\Gamma)} a_{\lambda(e)}$.
\end{definition}

Given an $\mA$-graph $\Gamma$ and a vector $z \in \Z^n$, we define the translation $\Gamma + (0^s, z)$ to be a copy of $\Gamma$ where each vertex and edge is moved by the vector $(0^s, z)$.
Note that if $\Gamma$ represents $y \in \mY$, then $\Gamma + (0^s, z)$ represents $\oX^{z} \cdot y$.
We call an $\mA$-graph \emph{full-image} if it contains at least one edge of label $\ell$ for each $\ell \in \{1, \ldots, t\}$.
Let $\Gamma$ be a full-image Eulerian $\mA$-graph, then it has a translation $\Gamma + (0^s, z)$ that contains the vertex $(b_1, 0^n)$.
By reading the labels on its Eulerian circuit starting from $(b_1, 0^n)$, we obtain an accepting run $w$ of $\mA$ such that $\Gamma(w) = \Gamma + (0^s, z)$.
We can then show:

\begin{restatable}{lemma}{lemgrp}\label{lem:grptoeul}
There exists an Identity Traversal of $\mA$ if and only if there exists a full-image Eulerian $\mA$-graph that represents $0$.
\end{restatable}

Our next step is to replace the Eulerian property in Lemma~\ref{lem:grptoeul} by a more ``geometric'' property.
For this, we generalize the notion of \emph{face-accessibility} introduced in~\cite{dong2024semigroup}.
Let $C$ be a (closed) convex polytope. A \emph{face} $F$ of $C$ is the intersection of $C$ with any closed halfspace whose boundary is disjoint from the interior of $C$.
A \emph{strict face} is a face of $C$ that is not the empty set or $C$ itself.
For example, if $C$ is of dimension two, then the strict faces of $C$ are its edges and its vertices.

\begin{definition}[face-accessibility of an $\mA$-graph]
    Let $\Gamma$ be an $\mA$-graph.
    Denote by $C \subsetneq \R^{s+n}$ the convex hull of $V(\Gamma)$.
    A strict face $F$ of $C$ is called \emph{accessible} if there is an edge $e \in E(\Gamma)$ such that $\sigma(e) \in F$ and $\tau(e) \in C \setminus F$.
    The $\mA$-graph $\Gamma$ is called \emph{face-accessible} if every strict face of $C$ is accessible.
    See Figures~\ref{fig:notaccessible}, \ref{fig:notaccessible2} and \ref{fig:accessible} for examples.\footnote{When $s=1$, face-accessibility of an $\mA$-graph is equivalent to face-accessibility of a $\mG$-graph defined in~\cite{dong2024semigroup}.}
\end{definition}

\begin{figure}[ht!]
    \centering
    \begin{minipage}[t]{.47\textwidth}
        \centering
        \includegraphics[width=0.55\textwidth,height=1.0\textheight,keepaspectratio, trim={7.4cm 0.05cm 4.7cm 0cm},clip]{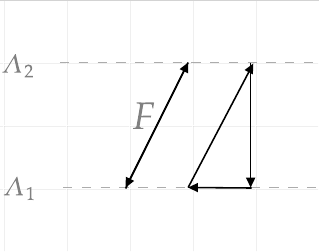}
        \caption{Example of a non face-accessible graph: the strict face $F$ not accessible.}
        \label{fig:notaccessible}
    \end{minipage}
    \hfill
    \begin{minipage}[t]{0.47\textwidth}
        \centering
        \includegraphics[width=0.55\textwidth,height=1.0\textheight,keepaspectratio, trim={7.4cm 0.75cm 3.7cm 0.75cm},clip]{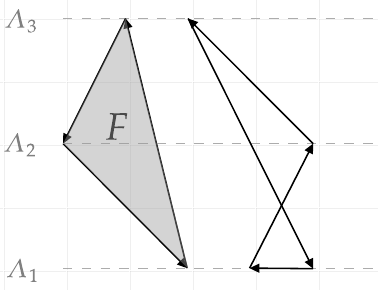}
        \caption{A non face-accessible graph: $F$ is not accessible. Note that $\Lambda_1, \Lambda_2, \Lambda_3$ are not in the same plane, despite appearing so in the figure.}
        \label{fig:notaccessible2}
    \end{minipage}
    
    \vspace{0.2cm}
    \centering
    \begin{minipage}[t]{.47\textwidth}
        \centering
        \includegraphics[width=0.55\textwidth,height=1.0\textheight,keepaspectratio, trim={7.4cm 0.05cm 4.7cm 0cm},clip]{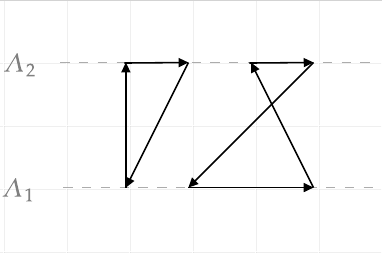}
        \caption{A face-accessible graph $\Gamma$ that is not Eulerian due to disconnectivity.}
        \label{fig:accessible}
    \end{minipage}
    \hfill
    \begin{minipage}[t]{0.47\textwidth}
        \centering
        \includegraphics[width=0.55\textwidth,height=1.0\textheight,keepaspectratio, trim={7.4cm 0.05cm 3.7cm 0cm},clip]{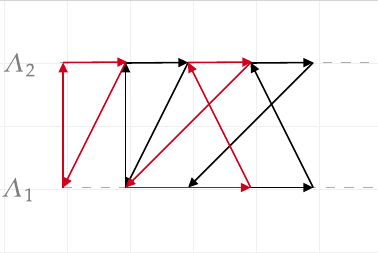}
        \caption{The union $\hG$ of two translations of $\Gamma$ drawn respectively in black and red.}
        \label{fig:trans}
    \end{minipage}
\end{figure}

Recall that a directed graph $\Gamma$ is called \emph{symmetric} if the in-degree is equal to the out-degree at every vertex.
An Eulerian graph is symmetric and face-accessible. While symmetric face-accessible graphs are not necessarily Eulerian, we show that they can be used to construct Eulerian graphs:

\begin{restatable}{theorem}{thmacctoeul}\label{thm:acctocon}
Suppose $\mA$ is trim and primitive.
Let $\Gamma$ be a full-image, symmetric and face-accessible $\mA$-graph.
Then there exist $z_1, \ldots, z_m \in \Z^n$, such that the union of translations $\hG \coloneqq \bigcup_{i = 1}^m \Gamma + (0^s, z_i) $ is an Eulerian graph.
\end{restatable}

See Figure~\ref{fig:trans} for an illustration of Theorem~\ref{thm:acctocon}.
The proof is given in Appendix~\ref{app:proofacctocon}.
When $s=1$, Theorem~\ref{thm:acctocon} degenerates into~\cite[Theorem~3.3]{dong2024semigroup}.
Proving the general statement with $s \geq 2$ is highly non-trivial.
It is crucial that $\Gamma$ is an $\mA$-graph (with vertices in $\Lambda$) instead of an arbitrary graph over $\Z^{s+n}$: the theorem is false for arbitrary graphs over $\Z^{s+n}$. Note that the face-accessibility condition is necessary: one can verify that taking a union of horizontal translations of the graphs in Figure~\ref{fig:notaccessible} or \ref{fig:notaccessible2} cannot produce a connected graph.
The primitiveness of $\mA$ is also necessary for the proof.
Lemma~\ref{lem:grptoeul} and Theorem~\ref{thm:acctocon} lead to:

\begin{restatable}{proposition}{propgtoe}\label{prop:grouptoeuler}
A trim primitive automaton $\mA$ admits an Identity Traversal if and only if there exists a full-image symmetric face-accessible $\mA$-graph that represents $0$.
\end{restatable}

\subsection{From face-accessible $\mA$-graphs to position polynomials via partial contractions}
Similar to~\cite{dong2024semigroup}, we now describe $\mA$-graphs using their \emph{position polynomials}. The difficulty here compared to~\cite{dong2024semigroup} is the characterization of face-accessibility. We overcome this by introducing the new notion of \emph{partial contractions}.

The \emph{position polynomials} of an $\mA$-graph is a tuple of polynomials $\bff = (f_1, \ldots, f_t) \in \N[\oX^{\pm}]^t$, where
\[
f_{\ell} \coloneqq \sum_{e \in E(\Gamma), \lambda(e) = \ell} \oX^{\pi_{\Z^n}(\sigma(e))}, \quad \ell = 1, \ldots, t.
\]
That is, $f_{\ell}$ is the sum of monomials $\oX^{z}$, where $z$ ranges over the $\Z^n$-coordinate of the source vertex of all edges of label $\ell$.
These polynomials have only non-negative coefficients, hence are in $\N[\oX^{\pm}]$.
For example, for the $\mA$-graph $\Gamma$ drawn in Figure~\ref{fig:Gw}, the position polynomials will be the four-tuple $(f_1, f_2, f_3, f_4)$, where $f_1 = 1, f_2 = X_1 + X_1^2, f_3 = X_1^3, f_4 = X_1^2$.

Conversely, given any tuple of polynomials $\bff = (f_1, \ldots, f_t) \in \N[\oX^{\pm}]^t$, one can construct a $\mG$-graph $\Gamma$ such that $\bff$ is exactly its tuple of position polynomials.
Indeed, for each monomial $c \oX^{z}$ appearing in $f_{\ell}$, we can draw $c \geq 1$ edges of label $\ell$ starting at vertex $(b_{\Omega(\ell)}, z)$.

In the wake of Proposition~\ref{prop:grouptoeuler}, we will characterize the following four properties of an $\mA$-graph $\Gamma$ using its position polynomials: (i) whether $\Gamma$ is full-image, (ii) whether $\Gamma$ is symmetric, (iii) whether $\Gamma$ represents $0 \in \mY$, (iv) whether $\Gamma$ is face-accessible.
Properties (i)-(iii) are easy to characterize:

\begin{restatable}{lemma}{lemthreetoeq}\label{lem:threetoeq}
Let $\Gamma$ be an $\mA$-graph with position polynomials $\bff = (f_1, \ldots, f_t) \in \N[\oX^{\pm}]^t$.
\begin{enumerate}[nosep, label ={\normalfont(\roman*)}]
    \item $\Gamma$ is full-image if and only if 
    $
        f_{\ell} \in \N[\oX^{\pm}]^* \coloneqq \N[\oX^{\pm}] \setminus \{0\}, \text{ for } \ell = 1, \ldots, t
    $.
    \item $\Gamma$ is symmetric if and only if for each $i = 1, \ldots, s$, we have
    $
        \sum_{\ell \colon \Omega(\ell) = i} f_{\ell} = \sum_{\ell \colon \Delta(\ell) = i} f_{\ell} \cdot \oX^{a_{\ell}}.
    $
    \item $\Gamma$ represents $0$ if and only if
    $
        \sum_{\ell = 1}^t f_{\ell} \cdot y_{\ell} = 0
    $.
\end{enumerate}
\end{restatable}

To characterize the face-accessibility of $\Gamma$, we will use the \emph{weighted degree} of a ``contracted'' version of position polynomials.
Let $\cdot$ denote the dot product in $\R^n$.
Given a polynomial $f = \sum_{a \in \Z^n} c_a \oX^a \in \R[\oX^{\pm}]^*$ and a vector $v \in \Rns \coloneqq \R^n \setminus \{0\}$, the \emph{weighted degree} of $f$ at direction $v$ is defined as
$
\deg_v(f) \coloneqq \max\{v \cdot  a \mid a \in \Z^n, c_a \neq 0\}
$.
Define additionally $\deg_v(0) = -\infty$.
We now introduce the notion of \emph{partial contractions} of an automaton $\mA$:

\begin{definition}[Partial contraction]
    A \emph{partial contraction} of $\mA$ is a tuple $(S, \mT, \rho)$, where
    \begin{enumerate}[nosep, label = (\roman*)]
        \item $S$ is a non-empty subset of $\{1, \ldots, s\}$.
        \item $\mT$ is a subset of $\{1, \ldots, t\}$, such that the transitions $\{\delta_{\ell} \mid \ell \in \mT\}$ form a spanning tree of the state set $\{q_i \mid i \in S\}$ in the underlying \emph{undirected} graph. In particular, $|\mT| = |S| - 1$.
        \item $\rho$ is an element of $S$; the state $q_{\rho}$ will be seen as the ``root'' of the undirected spanning tree.
    \end{enumerate}
    Then in the automaton $\mA^{\pm}$ (recall its definition before Lemma~\ref{lem:primitive}), for every index $i \in S$, there exists a unique path consisting of transitions in $\{\delta_{\ell} \mid \ell \in \mT\} \cup \{\delta_{\ell}^- \mid \ell \in \mT\}$, that connects from $q_i$ to $q_{\rho}$. See the automaton in Figure~\ref{fig:contract} for an illustration.
\end{definition}

\begin{figure}[h!]
    \centering
    \begin{minipage}[t]{.59\textwidth}
        \centering
        \includegraphics[width=1\textwidth,height=1.0\textheight,keepaspectratio, trim={3.6cm 2cm 1.5cm 0cm},clip]{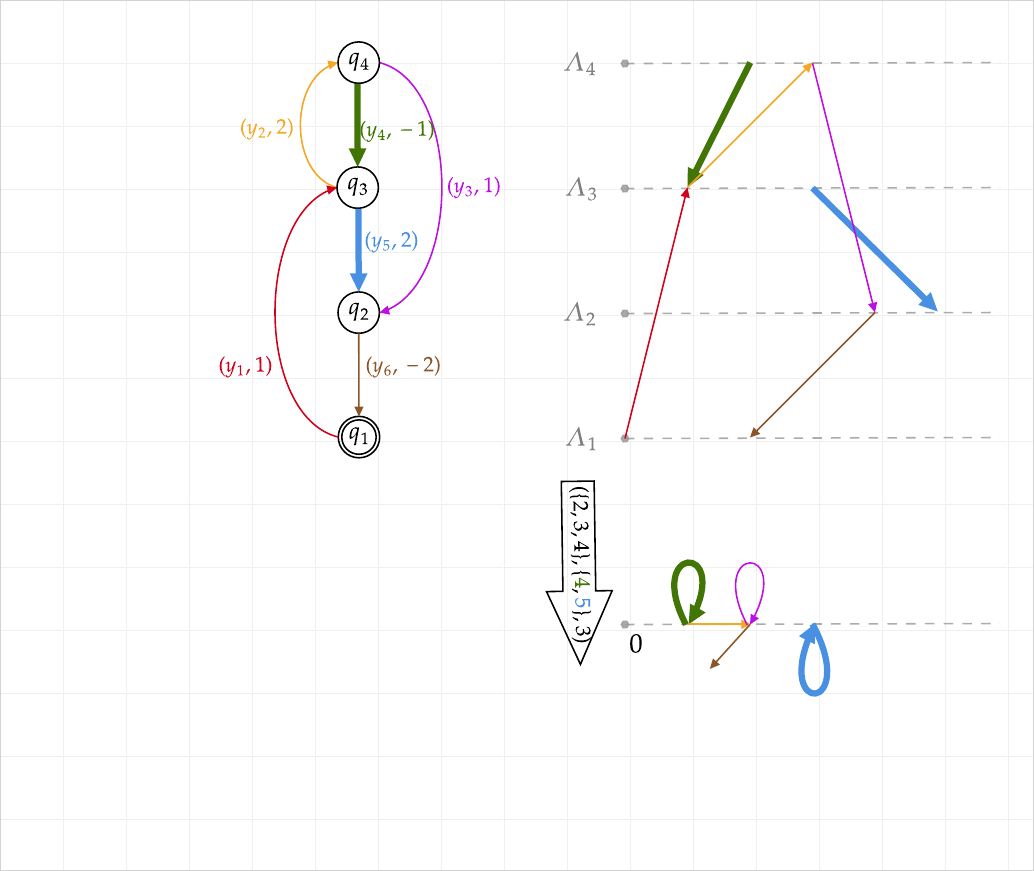}
        \caption{Partial contraction with $S = \{2, 3, 4\}$, $\mT = \{4, 5\}$, $\rho = 3$. The contracted transitions ($\delta_4$ and $\delta_5$) are marked with thick arrows, and $\Omega^{-1}(S) = \{2, 3, 4, 5, 6\}$.
        }
        \label{fig:contract}
    \end{minipage}
    \hfill
    \begin{minipage}[t]{0.36\textwidth}
        \centering
        \includegraphics[width=1\textwidth,height=1.0\textheight,keepaspectratio, trim={5.2cm 0cm 4.75cm 0cm},clip]{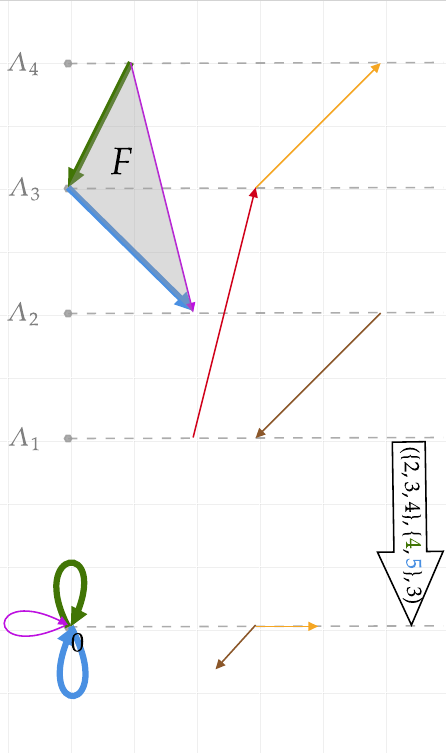}
        \caption{Non-accessible face $F$ becomes three petals after contraction.}
        \label{fig:contract2}
    \end{minipage}
\end{figure}

Let $\bff = (f_1, \ldots, f_t) \in \N[\oX^{\pm}]^t$ be the position polynomials of $\Gamma$.
For a set $S \subseteq \{1, \ldots, s\}$, denote $\Omega^{-1}(S) \coloneqq \{\ell \mid \Omega(\ell) \in S\}$.
Given a partial contraction $(S, \mT, \rho)$ and a tuple of position polynomials $\bff = (f_1, \ldots, f_t) \in \N[\oX^{\pm}]^t$, define the tuple of \emph{contracted position polynomials} 
\[
\bff^{(S, \mT, \rho)} \coloneqq \left(f^{(S, \mT, \rho)}_{\ell} \right)_{\ell \in \Omega^{-1}(S)} \in \N[\oX^{\pm}]^{|\Omega^{-1}(S)|},
\]
where the polynomial $f^{(S, \mT, \rho)}_{\ell} \in \N[\oX^{\pm}], \ell \in \Omega^{-1}(S),$ is defined as follows: there is a unique path $P_{\Omega(\ell)}$ in $\mA^{\pm}$, consisting of transitions in $\{\delta_{\ell} \mid \ell \in \mT\} \cup \{\delta_{\ell}^- \mid \ell \in \mT\}$, that connects from $q_{\Omega(\ell)}$ to $q_{\rho}$. Write its evaluation $\ev(P_{\Omega(\ell)})$ as $(y_{P_{\Omega(\ell)}}, z_{P_{\Omega(\ell)}})$.
We define $f^{(S, \mT, \rho)}_{\ell} \coloneqq f_{\ell} \cdot \oX^{z_{P_{\Omega(\ell)}}}$.

Similarly, if $\Delta(\ell) \in S$, we denote by $P_{\Delta(\ell)}$ the path in $\mA^{\pm}$ from $q_{\Delta(\ell)}$ to $q_{\rho}$, consisting of transitions in $\{\delta_{\ell} \mid \ell \in \mT\} \cup \{\delta_{\ell}^- \mid \ell \in \mT\}$.
The path $P_{\Delta(\ell)}$ evaluates to an element $(y_{P_{\Delta(\ell)}}, z_{P_{\Delta(\ell)}})$ for some $z_{P_{\Delta(\ell)}} \in \Z^n$. 
Define the \emph{contracted edge vectors} in $\Z^n$:
\begin{align*}
a^{(S, \mT, \rho)}_{\ell} \coloneqq 
\begin{cases}
    a_{\ell} + z_{P_{\Delta(\ell)}} - z_{P_{\Omega(\ell)}}, & \quad \ell \in \Omega^{-1}(S) \cap \Delta^{-1}(S), \\
    0^n & \quad\ell \in \Omega^{-1}(S) \setminus \Delta^{-1}(S).
\end{cases}
\end{align*}

\begin{example}[label=exa:cont]\label{expl:compute}
Let us compute the contracted position polynomials for the example in Figure~\ref{fig:contract}.
As shown in the figure, let $f_1 = 1$, $f_2 = X_1$, $f_3 = X_1^3$, $f_4 = X_1^2$, $f_5 = X_1^3$ and $f_6 = X_1^4$.
The partial contraction we use is $S = \{2, 3, 4\}$, $\mT = \{4, 5\}$, $\rho = 3$, so $\Omega^{-1}(S) = \{2, 3, 4, 5, 6\}$.
The path $P_4$ from $q_4$ to $q_3$ evaluates to $(y_4, -1)$, so $z_{P_4} = -1$.
We have $\Omega(2) = \Omega(3) = 4$, therefore $f^{(S, \mT, \rho)}_3 = f_3 \cdot X_1^{-1} = X_1^2$ and $f^{(S, \mT, \rho)}_4 = f_4 \cdot X_1^{-1} = X_1$.
Similarly, the path $P_2$ from $q_2$ to $q_3$ evaluates to $(y_5, 2)^{-1} = (- X_1^{-2} y_5, -2)$, so $z_{P_2} = -2$.
We have $\Omega(6) = 2$, therefore $f^{(S, \mT, \rho)}_6 = f_6 \cdot X_1^{-2} = X_1^2$.
Finally, we have $\Omega(2) = \Omega(5) = 3$, therefore $f^{(S, \mT, \rho)}_2 = f_2 = X_1$ and $f^{(S, \mT, \rho)}_5 = f_5 = X_1^3$.
The contracted position polynomials are therefore $(f^{(S, \mT, \rho)}_2, f^{(S, \mT, \rho)}_3, f^{(S, \mT, \rho)}_4, f^{(S, \mT, \rho)}_5, f^{(S, \mT, \rho)}_6) = (X_1, X_1^{2}, X_1, X_1^{3}, X_1^{2})$.
The contracted edge vectors are $(a^{(S, \mT, \rho)}_2, a^{(S, \mT, \rho)}_3, a^{(S, \mT, \rho)}_4, a^{(S, \mT, \rho)}_5, a^{(S, \mT, \rho)}_6) = (1, 0, 0, 0, 0)$.
\end{example}


The contracted position polynomials can be seen as ``position polynomials'' of a graph $\Gamma^{(S, \mT, \rho)}$ over $\Z^n$, which is obtained from $\Gamma$ by contracting the lattices $\Lambda_i, i \in S$ into a single lattice $\Z^n$, and discarding the other lattices $\Lambda_i, i \notin S$.
See Figure~\ref{fig:contract} for an illustration.
During the contraction, each lattice $\Lambda_i = \{b_i\} \times \Z^n, i \in S,$ is translated by the vector $(-b_i, z_{P_i})$, where $z_{P_i}$ is the sum of the $\Z^n$-coordinates along the unique path (in the subgraph of $\mA^{\pm}$ defined by $\mT$) from $q_i$ to $q_{\rho}$. 
In particular, if two lattices $\Lambda_i, \Lambda_j, i, j \in S,$ are connected by some edge $e$ with label $\ell \in \mT$, then the edge $e$ will become a loop after the contraction (e.g. the green and blue edges in Figure~\ref{fig:contract}).
Similarly, the contracted edge vector $a^{(S, \mT, \rho)}_{\ell}$ corresponds to the edge vector of label $\ell$ after the contraction.
The edges in the contracted graph $\Gamma^{(S, \mT, \rho)}$ have labels in $\Omega^{-1}(S)$.
Edges going to a discarded lattice $\Lambda_i, i \notin S$ (e.g. the blown edge) can be seen as ``dangling'': these are edges with labels in $\Omega^{-1}(S) \cap \Delta^{-1}(S)^{\complement}$, where $\Delta^{-1}(S)^{\complement} \coloneqq \{1, \ldots, t\} \setminus \Delta^{-1}(S)$.
Note that when $S$ is a singleton $\{\rho\}$, the set $\mT$ is empty, and the tuple $\bff^{(\{\rho\}, \emptyset, \rho)}$ is simply $(f_{\ell})_{\ell \in \Omega^{-1}(\rho)}$.

Given a partial contraction $(S, \mT, \rho)$ and a vector $v \in \Rns$, define 
\[
M_{v}((S, \mT, \rho), \bff) \coloneqq \left\{\ell \in \Omega^{-1}(S) \;\middle|\; \deg_{v}\left(f^{(S, \mT, \rho)}_{\ell}\right) = \max_{\ell' \in \Omega^{-1}(S)}\left\{\deg_{v}\left(f^{(S, \mT, \rho)}_{\ell'}\right)\right\} \right\}.
\]
This is the set of labels $\ell \in \Omega^{-1}(S)$ such that $\deg_{v}\left(f^{(S, \mT, \rho)}_{\ell}\right)$ is maximal.
In the contracted graph $\Gamma^{(S, \mT, \rho)}$, the source coordinate of edges with labels in $M_{v}((S, \mT, \rho), \bff)$ have the largest inner product with $v$.
Define additionally
\[
O_{v}(S, \mT, \rho) \coloneqq \left\{\ell \in \Omega^{-1}(S) \;\middle|\; a^{(S, \mT, \rho)}_{\ell} \not\perp v \right\}.
\]
This is the set of labels whose contracted edge vectors $a^{(S, \mT, \rho)}_{\ell}$ are not orthogonal to $v$.

\renewcommand\thmcontinues[1]{continued}
\begin{example}[continues=exa:cont]
We now compute the sets $M_{v}((S, \mT, \rho), \bff)$ and $O_{v}(S, \mT, \rho)$ for Example~\ref{expl:compute}.
Since we are working in dimension $n = 1$, the vector $v$ is a real number in $\R^*$.
If $v < 0$, then $M_{v}((S, \mT, \rho), \bff) = \{2, 4\}$, because $f^{(S, \mT, \rho)}_2$ and $f^{(S, \mT, \rho)}_4$ have the lowest degree (that is, \emph{highest} degree at direction $v = -1$) in the tuple $\bff^{(S, \mT, \rho)}$.
If $v > 0$, then $M_{v}((S, \mT, \rho), \bff) = \{5\}$, because $f^{(S, \mT, \rho)}_5$ has the highest degree in the tuple $\bff^{(S, \mT, \rho)}$.
Correspondingly in the contracted graph $\Gamma^{(S, \mT, \rho)}$, at the left extremity we have the sources of edges with label in $\{2, 4\}$; and at the right extremity we have the sources of edges with label in $\{5\}$. See Figure~\ref{fig:contract}.

As for the set $O_{v}(S, \mT, \rho)$, for all $v \in \R^*$ we have $O_{v}((S, \mT, \rho), \bff) = \{2\}$, because $a^{(S, \mT, \rho)}_2$ is the only non-zero vector among $a^{(S, \mT, \rho)}_{\ell}, \ell \in \Omega^{-1}(S)$.
Correspondingly in the contracted graph $\Gamma^{(S, \mT, \rho)}$, only the edge with label $2$ is not a loop or a dangling edge.
\end{example}

Using the contracted position polynomials and the contracted edge vectors, we can characterize face-accessibility of $\Gamma$.
The following lemma is a formal way of expressing ``$\Gamma$ is face-accessible if and only if all its partial contractions $\Gamma^{(S, \mT, \rho)}$ are face-accessible'' (in particular, faces containing sources of ``dangling'' edges are considered accessible).

\begin{restatable}{lemma}{lemacctoeq}\label{lem:acctoeq}
    Let $\Gamma$ be a symmetric $\mA$-graph with position polynomials $\bff = (f_1, \ldots, f_t) \in \N[\oX^{\pm}]^t$.
    Then $\Gamma$ is face-accessible if and only if for every partial contraction $(S, \mT, \rho)$, we have
    \begin{equation}\label{eq:acccond}
        \left( O_{v}(S, \mT, \rho) \cup \Delta^{-1}(S)^{\complement} \right) \cap M_{v}((S, \mT, \rho), \bff) \neq \emptyset \quad \text{ for every } v \in \Rns.
    \end{equation}
\end{restatable}
\begin{proof}[Sketch of proof]
    See Appendix~\ref{app:omitted} for a full proof.
    We show the contrapositive: the convex hull $C$ of $V(\Gamma)$ has a non-accessible face if and only if $\exists (S, \mT, \rho), \exists v \in \Rns,$ such that 
    \[
    \left( O_{v}(S, \mT, \rho) \cup \Delta^{-1}(S)^{\complement} \right) \cap M_{v}((S, \mT, \rho), \bff) = \emptyset.
    \]

    If $F$ is a non-accessible face of $C$, then consider all the edges whose source is in $F$ (their target will also be in $F$).
    Let $\widetilde{\mT}$ be the set of labels of these edges, and let $\widetilde{S}$ be the index set of origin and destination states of all transitions with label in $\widetilde{\mT}$.
    Consider the subautomaton $\mA'$ of $\mA$ whose states are $q_i, i \in \widetilde{S}$, and whose transitions are $\delta_t, t \in \widetilde{\mT}$ (we do not specify the initial or accepting state of $\mA'$).
    Let $S \subseteq \widetilde{S}$ be such that $\{q_i \mid i \in S\}$ is a weakly connected component of $\mA'$, let $\mT \subseteq \widetilde{\mT}$ be such that $\{\delta_t \mid t \in \mT\}$ is an undirected spanning tree of this weakly  connected component.
    Let $\rho$ be any element of $S$.
    We contract $\Gamma$ according to $(S, \mT, \rho)$ (see Figure~\ref{fig:contract2} for an illustration). Since $F$ is non-accessible, it will still be non-accessible after the contraction.
    After contraction, $F$ will be the extremal face at some direction $v \in \Rns$ (e.g.\ in Figure~\ref{fig:contract2}, $v$ would be $(-1)$, a vector pointing to the left).
    This means that among the position polynomials $f^{(S, \mT, \rho)}_{\ell}, \ell \in \Omega^{-1}(S),$ of $\Gamma^{(S, \mT, \rho)}$, the index of those with the highest $\deg_v(\cdot)$ correspond to the labels of contracted edges that are orthogonal to $v$ (e.g.\ the three loops in the leftmost position in the contracted graph of Figure~\ref{fig:contract2}).
    This is because $F$ is orthogonal to $v$ after contraction.
    In terms of description by $\deg_v(\cdot)$, the indices of these polynomials are exactly $M_{v}((S, \mT, \rho), \bff) \coloneqq \left\{\ell \in \Omega^{-1}(S) \;\middle|\; \deg_{v}\left(f^{(S, \mT, \rho)}_{\ell}\right) = \max_{\ell' \in \Omega^{-1}(S)}\left\{\deg_{v}\left(f^{(S, \mT, \rho)}_{\ell'}\right)\right\} \right\}$.
    Whereas the labels $\ell$ whose corresponding contracted edges that are \emph{not} orthogonal to $v$ are either in $O_{v}(S, \mT, \rho)$ (if the contracted edge is not dangling), or in $\Delta^{-1}(S)^{\complement}$ (if the contracted edge is dangling).
    Therefore $\left( O_{v}(S, \mT, \rho) \cup \Delta^{-1}(S)^{\complement} \right) \cap M_{v}((S, \mT, \rho), \bff) = \emptyset$.

    The other implication direction can be proved similarly.
\end{proof}

The definition of $\bff^{(S, \mT, \rho)}$ can be naturally extended to the case where $\bff \in \Z[\oX^{\pm}]^t$ instead of $\N[\oX^{\pm}]^t$: write $\bff = \bff^+ - \bff^-$ where $\bff^+, \bff^- \in \N[\oX^{\pm}]^t$, and define $\bff^{(S, \mT, \rho)} \coloneqq \left(\bff^+\right)^{(S, \mT, \rho)} - \left(\bff^-\right)^{(S, \mT, \rho)}$.

Define
\begin{equation}\label{eq:MZ}
\mM_{\Z} \coloneqq \left\{\bff = (f_1, \ldots, f_t) \in \Z[\oX^{\pm}]^t \;\middle|\; \sum_{\ell \colon \Omega(\ell) = i} f_{\ell} = \sum_{\ell \colon \Delta(\ell) = i} f_{\ell} \cdot \oX^{a_{\ell}},\; i = 1, \ldots, s; \;\sum_{\ell = 1}^K f_{\ell} \cdot y_{\ell} = 0 \right\}.
\end{equation}
This is the $\Z[\oX^{\pm}]$-module consisting of all $\bff \in \Z[\oX^{\pm}]^t$ satisfying the conditions in Lemma~\ref{lem:threetoeq}~(ii) and (iii).
A finite set of generators for $\mM_{\Z}$ can be effectively computed~\cite{schreyer1980berechnung}, for example using Gr\"{o}bner basis~\cite{eisenbud2013commutative}.

We then want to include the condition in Lemma~\ref{lem:acctoeq}. For this, we will replace the information of $\bff$ by all its contracted position polynomials.
Denote by $\mPC$ the set of all partial contractions of $\mA$, denote $K \coloneqq \sum_{(S, \mT, \rho) \in \mPC} |\Omega^{-1}(S)|$, and define the $\Z[\oX^{\pm}]$-module
\begin{equation}\label{eq:tMZ}
\tMZ \coloneqq \left\{\widetilde{\bff} \coloneqq \left(f^{(S, \mT, \rho)}_{\ell}\right)_{(S, \mT, \rho) \in \mPC, \ell \in \Omega^{-1}(S)} \in \Z[\oX^{\pm}]^K \;\middle|\; \bff \in \mM_{\Z} \right\}.
\end{equation}
Note that each entry $f^{(S, \mT, \rho)}_{\ell}$ in $\widetilde{\bff}$ is obtained by multiplying $f_{\ell}$ by a fixed monomial, independant of $\bff$.
Therefore, if $\{\bff_i \mid i \in I\}$ is a generating set of $\mM_{\Z}$, then $\{\widetilde{\bff}_i \mid i \in I\}$ is a generating set of $\tMZ$.
Note that we can recover $\bff$ as a sub-tuple of $\widetilde{\bff}$: for $\rho \in \{1, \ldots, s\}$, the sub-tuple $\bff^{(\{\rho\}, \emptyset, \rho)}$ of $\widetilde{\bff}$ is simply $(f_{\ell})_{\ell \in \Omega^{-1}(\rho)}$, and therefore $\widetilde{\bff}$ contains as a sub-tuple $(f_{\ell})_{\ell \in \bigcup_{\rho = 1}^s \Omega^{-1}(\rho)} = (f_{\ell})_{\ell \in \{1, \ldots, t\}}$.

Using Lemma~\ref{lem:acctoeq}, we now characterize the face-accessibility of $\Gamma$ by its contracted position polynomials $\widetilde{\bff}$.
For simplicity, rename the indices of $\widetilde{\bff}$ by writing $\widetilde{\bff} = (\widetilde{f}_1, \ldots, \widetilde{f}_K)$, where $\widetilde{f}_i = f^{(S_i, \mT_i, \rho_i)}_{\ell_i}, i = 1, \ldots, K$.
Similarly, define $\ta_i = a^{(S_i, \mT_i, \rho_i)}_{\ell_i}, i = 1, \ldots, K$.
For a partial contraction $(S, \mT, \rho)$, define the sets 
\begin{align*}
& I_{(S, \mT, \rho)} \coloneqq \left\{i \in \{1, \ldots, K\} \;\middle|\; (S_i, \mT_i, \rho_i) = (S, \mT, \rho) \right\}, \\
& J_{(S, \mT, \rho)} \coloneqq \left\{i \in \{1, \ldots, K\} \;\middle|\; (S_i, \mT_i, \rho_i) = (S, \mT, \rho), \ell_i \in \Delta^{-1}(S)^{\complement} \right\}.
\end{align*}
For a set $I \subseteq \{1, \ldots, K\}$, define 
\begin{equation*}
    M_{v}\left(I, \widetilde{\bff}\right) \coloneqq \left\{i \in I \;\middle|\; \deg_{v}\left(\widetilde{f}_i\right) = \max_{i' \in I}\left\{\deg_{v}\left(\widetilde{f}_{i'}\right)\right\} \right\},
\end{equation*}
and define
\[
O_{v} \coloneqq \left\{i \in \{1, \ldots, K\} \;\middle|\; \ta_i \not\perp v \right\}.
\]
One can see that $\bff \in \left(\N[\oX^{\pm}]^*\right)^t$ if and only if $\widetilde{\bff} \in \left(\N[\oX^{\pm}]^*\right)^K$, because $\widetilde{\bff}$ contains $\bff$ as a sub-tuple. Summarizing the definition of $\tMZ$ as well as Lemma~\ref{lem:threetoeq} and \ref{lem:acctoeq}, we obtain:

\begin{restatable}{proposition}{propgraphtoeq}\label{prop:graphtoeq}
    There exists a full-image symmetric face-accessible $\mA$-graph $\Gamma$, if and only if there exists $\tbf \in \tMZ \cap \left(\N[\oX^{\pm}]^*\right)^K$, satisfying
    \begin{equation}\label{eq:gen}
        \left(O_{v} \cup J_{(S, \mT, \rho)}\right) \cap M_{v}\left(I_{(S, \mT, \rho)}, \tbf\right) \neq \emptyset, \quad \text{ for every } v \in \Rns, \; (S, \mT, \rho) \in \mPC.
    \end{equation}
\end{restatable}

\subsection{Decidability for position polynomials}

To finish our proof of decidability, we will prove the following extension of~\cite[Theorem~3.9]{dong2024semigroup} and apply it with Proposition~\ref{prop:graphtoeq}.

\begin{restatable}{theorem}{thmdec}\label{thm:dec}
Denote $\A \coloneqq \R[\oX^{\pm}], \A^+ \coloneqq \Rp[\oX^{\pm}]^*$.
Fix $n \in \N$ and let $\Xi$ be a finite set of indices.
Suppose we are given as input a set of generators $\bg_1, \ldots, \bg_m \in \A^K$ with integer coefficients, the vectors $\ta_1, \ldots, \ta_K \in \Z^n$, as well as subsets $I_{\xi}, J_{\xi} \subseteq \{1, \ldots, K\}$ for each $\xi \in \Xi$.
Denote by $\mM$ be the $\A$-submodule of $\A^K$ generated by $\bg_1, \ldots, \bg_m$.
It is decidable whether there exists $\bff \in \mM \cap \ApK$ satisfying
\begin{equation}\label{eq:deccond}
        \left(O_{v} \cup J_{\xi}\right) \cap M_{v}(I_{\xi}, \bff) \neq \emptyset, \quad \text{ for every } v \in \Rns, \xi \in \Xi.
\end{equation}
Here, if $n = 0$ then $\A$ is understood as $\R$, and Property~\eqref{eq:deccond} is considered trivially true.
\end{restatable}

When the index set $\Xi$ has cardinality one, Theorem~\ref{thm:dec} is exactly~\cite[Theorem~3.9]{dong2024semigroup}.
Our proof of Theorem~\ref{thm:dec} essentially involves adding the quantifier ``for every $\xi \in \Xi$'' in all appropriate places in the proof of~\cite[Theorem~3.9]{dong2024semigroup}, and thus does not present new conceptual difficulties.
The full proof of Theorem~\ref{thm:dec} as well as a comparison with ~\cite[Theorem~3.9]{dong2024semigroup} is provided in Appendix~\ref{app:proofdec}.

We now finish the proof of Theorem~\ref{thm:ratmeta}.
In the wake of Proposition~\ref{prop:graphtoeq}, we need to decide whether there exists $\widetilde{\bff} \in \tMZ \cap \left(\N[\oX^{\pm}]^*\right)^K$ that satisfies condition~\eqref{eq:gen}.
Our goal is to apply Theorem~\ref{thm:dec} with the index set $\Xi = \mPC$.
However, $\tMZ$ is a $\Z[\oX^{\pm}]$-module, and in order to apply Theorem~\ref{thm:dec} we need an $\R[\oX^{\pm}]$-module $\mM$.
Let $\bg_1, \ldots, \bg_m$ be the generators of the $\Z[\oX^{\pm}]$-module $\tMZ$, define
\begin{equation*}
\mM \coloneqq \left\{h_1 \cdot \bg_1 + \cdots + h_m \cdot \bg_m \;\middle|\; h_1, \ldots, h_m \in \R[\oX^{\pm}]\right\}.
\end{equation*}
\begin{restatable}{lemma}{lemM}\label{lem:M}
    There exists an element $\tbf \in \tMZ \cap \left(\N[\oX^{\pm}]^*\right)^K$ satisfying Property~\eqref{eq:gen}, if and only if there exists $\bff \in \mM \cap \left(\Rp[\oX^{\pm}]^*\right)^K$ satisfying Property~\eqref{eq:gen}.
\end{restatable}

Therefore, by Lemma~\ref{lem:M} we can apply Theorem~\ref{thm:dec} with $\Xi = \mPC$ to the $\R[\oX^{\pm}]$-module $\mM$ instead of the $\Z[\oX^{\pm}]$-module $\tMZ$.
Summarizing Proposition~\ref{prop:grptotraverse}, \ref{prop:grouptoeuler}, \ref{prop:graphtoeq}, Lemma~\ref{lem:M} and Theorem~\ref{thm:dec}, we obtain a proof of Theorem~\ref{thm:ratmeta}, the main goal of this section:

\thmratmeta*
\begin{proof}
    Without loss of generality suppose $\mA$ is trim.
    By Lemma~\ref{lem:primitive} we can suppose $\mA$ to be primitive.
    By the series of reductions Proposition~\ref{prop:grptotraverse}, \ref{prop:grouptoeuler} and \ref{prop:graphtoeq}, $\ev(\mA)$ is a group if and only if there exists $\tbf \in \tMZ \cap \left(\N[\oX^{\pm}]^*\right)^K$ that satisfies condition~\eqref{eq:gen}.
    By Lemma~\ref{lem:M}, this is equivalent to the existence of $\bff \in \mM \cap \left(\Rp[\oX^{\pm}]^*\right)^K$ that satisfies condition~\eqref{eq:gen}.
    Note that $\mM_{\Z}$ is defined as the solution set of a system of linear equations over $\Z[\oX^{\pm}]$, therefore its generating set (and even Gr\"{o}bner basis) can be effectively computed~\cite{schreyer1980berechnung, eisenbud2013commutative}.
    We then extend the generating set of $\mM_{\Z}$ to the generating set of $\tMZ$ and $\mM$.
    Applying Theorem~\ref{thm:dec} with $\Xi$ being the set of partial contractions $\mPC$, we conclude that it is decidable whether there exists $\bff \in \mM \cap \left(\Rp[\oX^{\pm}]^*\right)^K$ satisfying condition~\eqref{eq:gen}.
\end{proof}

It is worth noting that -- similar to Lemma~\ref{lem:subsume} -- we can define a procedure that decides whether a rational semigroup $\ev(\mA)$ contains the neutral element from any algorithm deciding whether such semigroups are groups. This may be of independent interest:

\begin{lemma}\label{lem:idrat}
    Let $G$ be a group and $\mA$ a trim automaton over $G$ (with $q_1$ as only starting and accepting state). Then $\ev(\mA)$ contains the neutral element if and only if $\mA$ admits a trim \emph{sub-automaton} $\mA'$, such that $\ev(\mA')$ is a group.
    Here, a sub-automaton of $\mA$ is defined as an automaton whose state set and transition set are subsets of the state set and transition set of $\mA$.
\end{lemma}
\begin{proof}
    Let $e$ denote the neutral element of $G$.
    If $e\in \ev(\mA)$, then $\mA$ admits some accepting run $w$ with $\ev(w) = e$. Let $\mA'$ be the trim sub-automaton of $\mA$ whose states and transitions are exactly those used in $w$. Then $w$ is an Identity Traversal of $\mA'$, so $\ev(\mA')$ is a group by Proposition~\ref{prop:grptotraverse}.
    If $\mA$ admits a sub-automaton $\mA'$ such that $\ev(\mA')$ is a group, take any Identity Traversal $w$ of $\mA'$, then $w$ is an accepting run of $\mA$ with $\ev(w) = e$. So $\ev(\mA)$ contains the neutral element.

    
    Therefore, to decide whether $\ev(\mA)$ contains the neutral element, it suffices to enumerate all trim sub-automata $\mA'$ of $\mA$, and decide whether any of them satisfies the condition \say{$\ev(\mA')$ is a group}.
\end{proof}

Therefore, Theorem~\ref{thm:ratmeta} and Lemma~\ref{lem:idrat} show that it is decidable whether a rational semigroup $\ev(\mA)$ of $\mY \rtimes \Z^n$ contains the neutral element, by enumerating all sub-automaton of the trim automaton $\mA$.
Furthermore, recall that Section~\ref{sec:overview} showed for a given automaton $\mA$ over $\T(d, \K)$, where $\K$ is an algebraic number field, it is decidable whether $\ev(\mA)$ is a group.
Therefore, by Lemma~\ref{lem:idrat}, it is also decidable whether $\ev(\mA) \subseteq \T(d, \K)$ contains the neutral element.

\section{Nilpotent groups of finite Prüfer rank}\label{sec:nilp}

In this section we prove Theorem~\ref{thm:nilpsat}:

\thmnilp*

The idea is to extend a weaker version of the theorem due to Shafrir, and independently, due to Bodart, Ciobanu and Metcalfe:

\begin{theorem}[{\cite[Corollary~1]{shafrir2024saturation}, see also~\cite[Proposition~19]{bodart2024identity}}]\label{thm:nilpfin}
    Let $N$ be a finitely generated nilpotent group and $M$ be a subsemigroup of $N$. If $M[N, N] = N$, then $M = N$. More generally, if $M[N,N]$ is a finite-index subgroup of $N$, then $M$ is a finite-index subgroup of $N$.
\end{theorem}
This result should also be compared to the following theorem:
\begin{theorem}[Folklore, see {\cite[Theorem 2.2.3]{Khukhro+1993}}] \label{saturation}
    Let $G$ be a nilpotent group and $H$ a subgroup. Suppose that $H[G,G]=G[G,G]$, then $H=G$.
\end{theorem}
We will extend Theorem~\ref{thm:nilpfin} from finitely generated nilpotent groups to infinitely generated subgroups of $\UT(d, \K)$. It should be noted that Theorem \ref{thm:nilpfin} does not extend to general nilpotent groups contrary to Theorem \ref{saturation}.
Recall the classic notions of \emph{isolators} and \emph{Pr\"{u}fer rank}.

\begin{definition} 
The \emph{isolator} of a subset $X\subseteq G$ is the subset
\[ I(X) \coloneqq \{g\in G \mid \exists m\in\Z_{>0},\; g^m\in X \}. \]
\end{definition}
\begin{lemma}[{\cite[Theorem 2.5.8]{Khukhro+1993}}]
If $G$ is nilpotent and $H$ a subgroup, then $I(H)$ is a subgroup.
\end{lemma}

\begin{definition}[Prüfer rank] The \emph{Prüfer rank} of a group $G$, denoted $\rk(G)$, is the maximum number of generators needed to generate a finitely generated subgroup $H\le G$.
\end{definition}
For instance, every finitely generated subgroup of $(\Q,+)$ is cyclic, therefore $\rk(\Q)=1$.
\begin{lemma}[{Folklore, see~\cite[p.85, 3.]{lennox2004theory}}]\label{lem:rank} Let $G$ be a group
    \begin{enumerate}[leftmargin=8mm, label={\normalfont(\roman*)}]
        \item If $H\le G$, then $\rk(H)\le \rk(G)$.
        \item If $N\trianglelefteq G$, then $\rk(G)\le \rk(N)+\rk(G/N)$.
    \end{enumerate}
\end{lemma}


\begin{theorem}[{\cite[Theorem 2.5]{finite_base}}] \label{finite_base}
    If $G$ is nilpotent and has finite Prüfer rank, then there exists a finite set $B\subset G$ such that $I(\la B\ra_\grp)=G$.
\end{theorem}
We are now ready to prove Theorem~\ref{thm:nilpsat}.
We recall $[x, y]$ denotes the element $xyx^{-1} y^{-1}$, called the \emph{commutator} of $x$ and $y$.
\begin{proof}[Proof of Theorem~\ref{thm:nilpsat}] We prove the result under a slightly weaker condition: $[N,N]$ has finite Prüfer rank. Let's first observe this condition is indeed satisfied if $N\le\UT(d,\K)$. Indeed,
\[ \rk[N,N] \le \rk(N) \le \rk(\UT(d,\K)) \le \dim_\Q(\K) \cdot {d\choose 2} < \infty \]
using Lemma~\ref{lem:rank} (i) twice, then part (ii) iteratively on the lower central series of $\UT(d,\K)$.

\medbreak

\noindent From now on, we only suppose that $\rk[N,N]<\infty$. Using Theorem \ref{finite_base}, there exists a finite set $B\subset [N,N]$ such that $I(\la B\ra_\grp)=[N,N]$. Each element of $[N,N]$ (hence each element of $B$) can be written as a product of commutators, so we can in turn find a finite set $A\subset N$ such that
    \[ I\big(\!\la \left\{[a,a'] \;\big|\; a,a'\in A \right\}\ra_\grp\big) \supseteq [N,N]  \]
    As $M[N,N]=N$, we find a finite set $X\subset M$ such that, for each $a\in A$ there exists $x,y\in X$ such that $x=a$ and $y=a^{-1}\pmod{[N,N]}$. Fix $g\in M$. We take $h\in M$ such that $h=g^{-1}\pmod{[N,N]}$. Let $\tilde N=\la g,h,A,X\ra_\grp$ and $\tilde M=M\cap\tilde N$. By construction, we have
    \[ \tilde M\cdot I([\tilde N,\tilde N]) \supseteq \tilde M\cdot [N,N] \supseteq \la g,h,X\ra \cdot [N,N] = \tilde N \cdot [N,N] \supseteq \tilde N\]
    hence $\tilde M\cdot I_{\tilde N}([\tilde N,\tilde N])=\tilde N$ (where $I_{\tilde N}(X) \coloneqq \{g\in\tilde N:\exists m>0,\; x^m\in X\}=I(X)\cap\tilde N$). 

    \medbreak

    \noindent As $\tilde N$ is finitely generated, this implies that $\tilde M[\tilde N,\tilde N]$ is a finite-index subgroup of $\tilde N$. Theorem~\ref{thm:nilpfin} gives that $\tilde M$ is a subgroup of $\tilde N$. In particular $g$ admits an inverse in $\tilde M$, hence in $M$. We conclude that $M$ is a subgroup, hence $M=N$ using Theorem \ref{saturation}.
\end{proof}



\begin{remark}\label{rmk:counterexample}
    Theorem~\ref{thm:nilpsat} 
    doesn't hold for more general nilpotent groups.
    For example, it does not hold when $N$ is a subgroup of $\UT(3, \Q(X))$, where $\Q(X)$ denotes the field of rational functions over $\Q$.
    Indeed, let $N$ be the group generated by the following elements:
    \[
    A_i =
    \begin{pmatrix}
        1 & X^i & X^{2i+1} \\
        0 & 1 & 0 \\
        0 & 0 & 1 \\
    \end{pmatrix}
    \!,\,
    B_i =
    \begin{pmatrix}
        1 & -X^i & X^{2i+1} \\
        0 & 1 & 0 \\
        0 & 0 & 1 \\
    \end{pmatrix}
    \!,\,
    C_i =
    \begin{pmatrix}
        1 & 0 & X^{2i+1} \\
        0 & 1 & X^i \\
        0 & 0 & 1 \\
    \end{pmatrix}
    \!,\,
    D_i =
    \begin{pmatrix}
        1 & 0 & X^{2i+1} \\
        0 & 1 & -X^i \\
        0 & 0 & 1 \\
    \end{pmatrix}
    \]
    for $i = 0, 1, 2, \ldots$, that is, $N = \gen{\left\{A_i, B_i, C_i, D_i \mid i \in \N\right\}}_{\grp}$.
    Then
    \[
    N =
    \left\{
    \begin{pmatrix}
        1 & g & f \\
        0 & 1 & h \\
        0 & 0 & 1 \\
    \end{pmatrix}
    \,\middle|\;
    f, g, h \in \Z[X]
    \right\},
    \quad
    [N, N] =
    \left\{
    \begin{pmatrix}
        1 & 0 & f \\
        0 & 1 & 0 \\
        0 & 0 & 1 \\
    \end{pmatrix}
    \,\middle|\;
    f \in \Z[X]
    \right\}.
    \]
    The latter equality is justified by the fact that $[N, N]$ contains the elements
    \[
    A_i C_0 A_i^{-1} C_0^{-1}
    =
    \begin{pmatrix}
        1 & 0 & X^{i} \\
        0 & 1 & 0 \\
        0 & 0 & 1 \\
    \end{pmatrix},\;
    i = 0, 1, 2, \ldots.
    \]
    Let $M$ be the semigroup $\gen{\left\{A_i, B_i, C_i, D_i \mid i \in \N\right\}}$, then $M[N, N] = N$.
    However, we have $M \neq N$ because $I \notin M$.
    Indeed, let
    $
    P = \begin{pmatrix}
        1 & g & f \\
        0 & 1 & h \\
        0 & 0 & 1 \\
    \end{pmatrix}
    $
    any non-empty product of $A_i, B_i, C_i, D_i, i \in \N$, where $m \geq 0$ is the largest index used.
    Then $\deg(g) \leq m, \deg(h) \leq m$. Furthermore, $\deg(f) = 2m+1$, and the coefficient of the term $X^{2m+1}$ in $f$ is positive. 
    So $f \neq 0$, and $P \neq I$.
\end{remark}

\section{Conclusion and outlook}

In this paper we proved decidability of the Identity Problem and the Group Problem in virtually solvable subgroups of $\GL(d, \Qbar)$.
An immediate open question is whether our decidability result still holds when the field $\Qbar$ is replaced by other effectively computable fields such as $\Q(X)$ or $\F_p(X)$ (where $\F_p$ denotes the finite field of cardinality $p$).
A number of interesting solvable groups are not embeddable in $\GL(d, \Qbar)$, such as the \emph{wreath product} $\Z \wr \Z \coloneqq \Z[X^{\pm}] \rtimes \Z$ (which is embeddable in $\GL(2, \Q(X))$) and the lamplighter group $\Z_2 \wr \Z \coloneqq \F_2[X^{\pm}] \rtimes \Z$ (which is embeddable in $\GL(2, \F_2(X))$).
Many such groups still have decidable Identity Problem and Group Problem (direct consequence of Theorem~\ref{thm:ratmeta}, also~\cite{dong2024semigroup}), and the lamplighter group $\Z_2 \wr \Z$ even has decidable Semigroup Membership~\cite{lohrey2015rational}.
Therefore, one might expect the Identity Problem and the Group Problem to be decidable for virtually solvable matrix groups over some well-behaved fields other than $\Qbar$.
On the other hand, a celebrated result of Kharlampovich~\cite{kharlampovich1981finitely} shows that general solvable groups are highly intractable: there exists a 3-step solvable group (a group $G$ such that $G^{(3)}$ is trivial), where the \emph{Word Problem} is undecidable. Moreover, there exists center-by-metabelian groups with decidable Word Problem and undecidable Torsion Problem (hence undecidable Identity Problem) \cite[Proposition 3.2]{Arzhantseva}. Therefore, one might expect the Identity Problem and the Group Problem to be undecidable for solvable matrix groups over more complicated fields.

As demonstrated in Remark~\ref{rmk:counterexample}, one of our key theorems (Theorem~\ref{thm:nilpsat}) no longer holds for the group $\UT(3, \Q(X))$.
Therefore, our proof of decidability for the Group Problem does not apply to solvable subgroups of $\GL(3, \Q(X))$.

\bibliography{solvable}

\appendix
\section{Omitted proofs}\label{app:omitted}

    

\lemvirtorat*
\begin{proof}
    If $\gen{A_1, \ldots, A_m}$ is a group then obviously $\gen{A_1, \ldots, A_m} \cap T$ is a group.
    Suppose that $\gen{A_1, \ldots, A_m} \cap T$ is a group, we show that every element in $\gen{A_1, \ldots, A_m}$ is invertible.
    Take any $s \in \gen{A_1, \ldots, A_m}$, then since $|G/T| < \infty$, we have $s^{|G/T|} \in \gen{A_1, \ldots, A_m} \cap T$.
    Since $\gen{A_1, \ldots, A_m} \cap T$ is a group, we have $s^{- |G/T|} \in \gen{A_1, \ldots, A_m} \cap T$.
    Thus, $s^{-1} = \underbrace{ss \cdots s}_{|G/T|-1} \cdot s^{-|G/T|} \in \gen{A_1, \ldots, A_m}$.
    Therefore, $\gen{A_1, \ldots, A_m}$ is a group.

    Let $B_1 \coloneqq I, B_2 \ldots, B_s \in G$ be the representatives of the left quotient $T \backslash G$.
    These can be effectively computed using the following saturation procedure.
    Start with the set $\mB \coloneqq \{I, A_1, \ldots, A_m\}$. For $i = 1, \ldots, m$, we remove $A_i$ from $\mB$ if there exists $j > i$ such that $A_i A_j^{-1} \in T$.
    Note that membership in $T$ is decidable since $T$ is solvable~\cite{kopytov1968solvability}.
    We then repeat the following process: check for all pairs of elements $B_i, B_j \in \mB$, whether there exists some element $B_k \in \mB$ such that $B_i B_j B_k^{-1} \in T$.
    If there exists a pair of element $B_i, B_j \in \mB$, such that $B_i B_j B_k^{-1} \notin T$ for all $B_k \in S$, then we append the matrix $B_i B_j$ to the set $\mB$.
    Otherwise we stop the process.
    Note that at each point of the process, the set $\mB$ contains representatives for different equivalent classes in $T \backslash G$.
    By the finiteness of $T \backslash G$, the process must terminate, by which point $\mB$ contains the complete set of representatives for $T \backslash G$.

    We then construct an automaton $\mA$ over $T$ with states $q_1, \ldots, q_s$ using the method from~\cite[Lemma~3.3]{kambites2007rational}.
    For each $A_i, B_j, i \in \{1, \ldots, m\}, j \in \{1, \ldots, s\}$, there exists $\varphi(j, i) \in \{1, \ldots, s\}$ such that $B_j A_i \in T B_{\varphi(j, i)}$. Compute $T_{ji} \in T$ be such that $B_j A_i = T_{ji} B_{\varphi(j, i)}$.
    For each $i \in \{1, \ldots, m\}, j \in \{1, \ldots, s\}$, we add a transition $\delta_{mj - m + i}$ in $\mA$ from the state $q_j$ to $q_{\varphi(j, i)}$, with evaluation $T_{ji}$.
    We claim that $\ev(\mA) = \gen{A_1, \ldots, A_m} \cap T$.

    Indeed, take any product $A_{i_1} A_{i_2} \cdots A_{i_p} \in \gen{A_1, \ldots, A_m} \cap T$.
    Let $T B_{j_1}, T B_{j_2}, T B_{j_3}, \ldots, T B_{j_{p+1}},$ respectively denote the equivalent classes of $I, A_{i_1}, A_{i_1} A_{i_2}, \ldots, A_{i_1} A_{i_2} \cdots A_{i_{p}}$ in $T \backslash G$.
    In particular, for $l = 1, \ldots, p$, we have $T B_{j_l} A_{i_l} = T B_{j_{l+1}}$, so $\varphi(j_l, i_l) = j_{l+1}$.
    Consider the path $\delta_{m j_1 - m + i_1} \delta_{m j_2 - m + i_2} \cdots \delta_{m j_p - m + i_p}$ in $\mA$.
    It is indeed a path because for each $l = 1, \ldots, p$, the transition $\delta_{m j_l - m + i_l}$ originates at the state $q_{j_l}$ and ends at state $q_{\varphi(j_i, i_l)} = q_{j_{l+1}}$.
    Furthermore, the path originates at $q_{1}$ because $T B_{j_1} = T$ so $j_1 = 1$.
    It ends at $q_{1}$ because $A_{i_1} A_{i_2} \cdots A_{i_{p}} \in T$, so $T B_{j_{p+1}} = T$ and $j_{p+1} = 1$.
    Therefore, $A_{i_1} A_{i_2} \cdots A_{i_p} = \ev(\delta_{m j_1 - m + i_1} \delta_{m j_2 - m + i_2} \cdots \delta_{m j_p - m + i_p}) \in \ev(\mA)$.
    Thus, $\gen{A_1, \ldots, A_m} \cap T \subseteq \ev(\mA)$.

    To prove $\ev(\mA) \subseteq \gen{A_1, \ldots, A_m} \cap T$, take any accepting run $\delta_{m j_1 - m + i_1} \delta_{m j_2 - m + i_2} \cdots \delta_{m j_p - m + i_p}$ of $\mA$.
    Then $j_1 = \varphi(j_p, i_p) = 1$, and $\varphi(j_l, i_l) = j_{l+1}$ for $l = 1, \ldots, p-1$.
    Therefore, 
    \begin{multline*}
    \ev(\delta_{m j_1 - m + i_1} \delta_{m j_2 - m + i_2} \cdots \delta_{m j_p - m + i_p}) \\
    = T_{j_1 i_1} T_{j_2 i_2} \cdots T_{j_p i_{p}} = B_{j_1} A_{i_1} B_{\varphi(j_1, i_1)}^{-1} \cdot B_{j_2} A_{i_2} B_{\varphi(j_2, i_2)}^{-1} \cdot \cdots \cdot B_{j_p} A_{i_p} B_{\varphi(j_p, i_p)}^{-1} \\
    = A_{i_1} A_{i_2} \cdots A_{i_p} \in \gen{A_1, \ldots, A_m}.
    \end{multline*}
    Thus, $\ev(\mA) \subseteq \gen{A_1, \ldots, A_m} \cap T$.    
    We therefore conclude that $\ev(\mA) = \gen{A_1, \ldots, A_m} \cap T$.
\end{proof}

\lemprim*
\begin{proof}
    The proof uses standard techniques for automata over groups, see for example~\cite{gilman2016groups} for similar results.
    Let $\mA$ be an automaton over $T$ such that $\ev(\mA) = S$.
    Without loss of generality suppose $\mA$ to be trim.
    For each state $q_i, i=2, \ldots, s$, of $\mA$, let $w_i$ denote a path from $q_i$ to $q_1$ and write $B_i \coloneqq \ev(w_i)$.
    Define additionally $B_1 \coloneqq I$.
    Let $\mA'$ be the automaton constructed as follows.
    The states of $\mA'$ are the same as $\mA$. For each transition $\delta_{\ell}$ in $\mA$, there is a transition $\delta'_{\ell}$ in $\mA'$ with the same origin and destination states, such that $\ev(\delta'_{\ell}) = B_{\Omega(\ell)}^{-1} \ev(\delta_{\ell}) B_{\Delta(\ell)}$.
    We will show that $\mA'$ is an automaton over $\gen{\ev(\mA)}_{\grp}$ and $\ev(\mA') = \ev(\mA)$.

    First, we show $\ev(\mA') = \ev(\mA)$.
    Note that $\delta'_{\ell_1} \cdots \delta'_{\ell_m}$ is an accepting run of $\ev(\mA')$ if and only if $\delta_{\ell_1} \cdots \delta_{\ell_m}$ is an accepting run of $\ev(\mA)$.
    Furthermore, when they are accepting runs, we have
    \begin{multline*}
        \ev(\delta'_{\ell_1} \cdots \delta'_{\ell_m}) = \left(B_{\Omega(\ell_1)}^{-1} \ev(\delta_{\ell_1}) B_{\Delta(\ell_1)}\right) \cdot \left(B_{\Omega(\ell_2)}^{-1} \ev(\delta_{\ell_2}) B_{\Delta(\ell_2)} \right) \cdot \cdots \cdot \left(B_{\Omega(\ell_m)}^{-1} \ev(\delta_{\ell_m}) B_{\Delta(\ell_m)}\right) \\
        = \ev(\delta_{\ell_1}) \cdots \ev(\delta_{\ell_m}) = \ev(\delta_{\ell_1} \cdots \delta_{\ell_m}).
    \end{multline*}
    Therefore $\ev(\mA') = \ev(\mA)$.
    
    Next, we show that $\gen{\ev(\delta'_1), \ldots, \ev(\delta'_t))}_{\grp} = \gen{\ev(\mA)}_{\grp}$.
    We claim that $B_{\Omega(\ell)}^{-1} \ev(\delta_{\ell}) B_{\Delta(\ell)} \in \gen{\ev(\mA)}_{\grp}$ for all $\ell$.
    Since $\mA$ is trim, for each $i = 1, \ldots, s,$ there exists a path $v_{i}$ from $q_1$ to $q_{i}$.
    Then $\ev(v_{\Omega(\ell)}) \ev(\delta_{\ell}) B_{\Delta(\ell)} = \ev(v_{\Omega(\ell)} \delta_{\ell} w_{\Delta(\ell)}) \in \ev(\mA)$.
    Similarly, $\ev(v_{\Omega(\ell)}) B_{\Omega(\ell)} = \ev(v_{\Omega(\ell)} w_{\Omega(\ell)}) \in \ev(\mA)$.
    Therefore, 
    \[
    \ev(\delta'_{\ell}) = B_{\Omega(\ell)}^{-1} \ev(\delta_{\ell}) B_{\Delta(\ell)} = \left(\ev(v_{\Omega(\ell)}) B_{\Omega(\ell)}\right)^{-1} \left(\ev(v_{\Omega(\ell)}) \ev(\delta_{\ell}) B_{\Delta(\ell)}\right) \in \gen{\ev(\mA)}_{\grp}.
    \]

    Since every element in $\ev(\mA')$ is a product of $\ev(\delta'_1), \ldots, \ev(\delta'_t)$, we have
    \[
    \gen{\ev(\mA')}_{\grp} \leq \gen{\ev(\delta'_1), \ldots, \ev(\delta'_t))}_{\grp} \leq \gen{\ev(\mA)}_{\grp} = \gen{\ev(\mA')}_{\grp}.
    \]
    So $\gen{\ev(\mA)}_{\grp} = \gen{\ev(\delta'_1), \ldots, \ev(\delta'_t))}_{\grp}$ is finitely generated, and the transitions of $\mA'$ actually evaluates in $\gen{\ev(\mA)}_{\grp}$.
\end{proof}

\lemeffective*

\begin{proof}
    By~\cite[Lemma~2]{kopytov1971solvability}, we can compute a finite presentation of $T/[N, N]$ in the variety of metabelian groups (see for example~\cite[Chapter~9]{lennox2004theory} for the definition of finite presentation in varieties).
    Using this finite presentation, \cite[Lemma~B.3]{dong2024semigroup} shows that $T/[N, N]$ can be effectively embedded in a quotient $(\mY \rtimes \Z^n)/H$ satisfying the conditions (i) and (ii).
\end{proof}

\lemaddH*
\begin{proof}    
    Let $\mA$ be an automaton over $(\mY \rtimes \Z^n)/H$ recognizing $\varphi(\overline{S})$.
    Denote by $\delta_1, \ldots, \delta_t$ its transitions.
    Let $\mA'$ be the automaton over $\mY \rtimes \Z^n$ obtained from $\mA$ by replacing the evaluations $\ev(\delta_1) = (y_1, a_1)H, \ldots, \ev(\delta_t) = (y_t, a_t)H,$ respectively by $(y_1, a_1), \ldots, (y_t, a_t)$.
    Denote by $(0, h_1), \ldots, (0, h_m)$ the generators of $H$ as a semigroup.
    We then append $m$ transitions to $\mA'$, whose origins and destinations are the accepting state $q_1$, and whose evaluations are respectively $(0, h_1), \ldots, (0, h_m)$.
    We thus obtain an automaton over $\mY \rtimes \Z^n$ that recognizes $\varphi(S)H$, because $(0, h_1), \ldots, (0, h_m)$ commute with $(y_1, a_1), \ldots, (y_t, a_t)$.
\end{proof}

\lemprimitive*
\begin{proof}
    Denote by $\pi$ the projection $\mY \rtimes \Z^n \rightarrow \Z^n$.
    Define $L \coloneqq \pi(\ev(\mA^{\pm}))$.
    Fix $i \in \{2, \ldots, s\}$. Let $w_i$ be a path in $\mA^{\pm}$ from $q_1$ to $q_i$, and define $z_i \coloneqq \pi(\ev(w_i))$.
    Then every path $w$ from $q_1$ to $q_i$ satisfies $\pi(\ev(w)) \in z_i + L$, because one can concatenate $w$ with the ``inverse'' $w_i^{-}$ to obtain an accepting run of $\mA^{\pm}$, so $\pi(\ev(w)) - z_i = \pi(\ev(w w_i^{-})) \in L$.
    Therefore, we have $a_{\ell} \in z_{\Delta(\ell)} - z_{\Omega(\ell)} + L$, for all $\ell$.

    Next, consider the automaton $\mA'$ obtained from $\mA$ as follows: the states of $\mA'$ are the same as $\mA$, and the transitions of $\mA'$ are $\delta'_1, \ldots, \delta'_t$, where $\delta'_{\ell}, \ell = 1, \ldots, t,$ has the same origin and target as $\delta_{\ell}$, but $\ev(\delta'_{\ell}) = (0, - z_{\Omega(\ell)}) \cdot \ev(\delta_{\ell}) \cdot (0, z_{\Delta(\ell)})$.
    We claim that $\ev(\mA') = \ev(\mA)$, by the same argument as in the proof of Lemma~\ref{lem:prim}. Indeed, take any accepting run $w = \delta_{\ell_1} \delta_{\ell_2} \cdots \delta_{\ell_p}$ of $\mA$, we have $1 = \Omega(\ell_1), \Delta(\ell_1) = \Omega(\ell_2), \ldots, \Delta(\ell_{p-1}) = \Omega(\ell_p), \Delta(\ell_p) = 1$.
    Then the run $w' = \delta'_{\ell_1} \delta'_{\ell_2} \cdots \delta'_{\ell_p}$ of $\mA'$ is accepting, and 
    \begin{multline*}
        \ev(w') = \\
        (0, - z_{\Omega(\ell_1)}) \cdot \ev(\delta_{\ell_1}) \cdot (0, z_{\Delta(\ell_1)}) \cdot (0, - z_{\Omega(\ell_2)}) \cdot \ev(\delta_{\ell_2}) \cdot (0, z_{\Delta(\ell_2)}) \cdots (0, - z_{\Omega(\ell_p)}) \cdot \ev(\delta_{\ell_p}) \cdot (0, z_{\Delta(\ell_p)}) \\
        = \ev(\delta_{\ell_1}) \ev(\delta_{\ell_2}) \cdots \ev(\delta_{\ell_p}) = \ev(w).
    \end{multline*}
    Similarly, if $w' = \delta'_{\ell_1} \delta'_{\ell_2} \cdots \delta'_{\ell_p}$ is an accepting run of $\mA'$, then $w = \delta_{\ell_1} \delta_{\ell_2} \cdots \delta_{\ell_p}$ is an accepting run of $\mA$ such that $\ev(w) = \ev(w')$.
    Therefore $\ev(\mA') = \ev(\mA)$.
    The same argument shows $\ev({\mA'}^{\pm}) = \ev(\mA^{\pm})$, and hence $\pi(\ev({\mA'}^{\pm})) = \pi(\ev(\mA^{\pm})) = L$.
    Note that for $\ell = 1, \ldots, t$, we have $\pi(\ev(\delta'_\ell)) = -z_{\Omega(\ell)} + a_{\ell} + z_{\Delta(\ell)} \in L$.
    Therefore, we can without loss of generality replace $\mA$ with $\mA'$, and suppose $a_{\ell} = \pi(\ev(\delta_\ell)) \in L$ for all $\ell$.

    Finally, let $\beta_1, \ldots, \beta_{\widetilde{n}} \in \Z^n$ denote a basis of lattice $L$, and define the new variables $\widetilde{X_i} \coloneqq \oX^{\beta_i}, i = 1, \ldots, \widetilde{n}$.
    Let $\widetilde{\mY}$ be the $\Z[\widetilde{X_1}^{\pm}, \ldots, \widetilde{X_{\tn}}^{\pm}]$-module generated by $y_1, \ldots, y_t$, then a finite presentation of $\widetilde{\mY}$ can be effectively computed~\cite[Theorem~2.14]{baumslag1981computable}.
    Thus, each element $(y_{\ell}, a_{\ell})$ can now be represented as an element in $\widetilde{\mY} \rtimes \Z^{\tn}$.
    The automaton $\mA$ is thus considered as an automaton $\widetilde{\mA}$ over $\widetilde{\mY} \rtimes \Z^{\tn}$.
    Since $\mA$ and $\widetilde{\mA}$ recognize the same elements under different presentations, $\ev(\mA)$ is a group if and only if $\ev(\widetilde{\mA})$ is a group.
    Furthermore, by the definition $L$, we have $\pi(\ev(\widetilde{\mA}^{\pm})) = \Z^{\tn}$; so $\widetilde{\mA}$ is primitive.
\end{proof}

\proptraverse*
\begin{proof}
    Suppose $\ev(\mA)$ is a group. Let $w $ be any accepting traversal, then $\ev(w) \in \ev(\mA)$. Since $\ev(\mA)$ is a group, we have $\ev(w)^{-1} \in \ev(\mA)$. Let $v$ be an accepting run that evaluates to $\ev(w)^{-1}$. Then the concatenation $wv$ is an accepting traversal and $\ev(wv) = \ev(w) \ev(v) = (0, 0^n)$.

    Suppose there exists an Identity Traversal $w$.
    Each transition $\delta_{\ell}, \ell = 1, \ldots, t$, appears at least once in $w$. For each $\ell = 1, \ldots, t$, write $w = u_{\ell} \delta_{\ell} v_{\ell}$, and define $w_{\ell}$ to be the concatenation $v_{\ell} u_{\ell}$: it is a path that starts at the destination of $\delta_{\ell}$ and ends at the origin of $\delta_{\ell}$, such that $\ev(w_{\ell}) = \ev(\delta_{\ell})^{-1}$.
    In order to show that $\ev(\mA)$ is a group, let $w = \delta_{\ell_1} \delta_{\ell_2} \cdots \delta_{\ell_p}$ be any accepting run and we prove $\ev(w)^{-1} \in M$.
    Consider the concatenation $w' \coloneqq w_{\ell_p} \cdots w_{\ell_2} w_{\ell_1}$, then $w'$ is an accepting run and $\ev(w') = \ev(w_{\ell_p}) \cdots \ev(w_{\ell_2}) \ev(w_{\ell_1}) = \ev(\delta_{\ell_m})^{-1} \cdots \ev(\delta_{\ell_2})^{-1} \ev(\delta_{\ell_1})^{-1} = \ev(w)^{-1}$.
    Therefore $\ev(w)^{-1} = \ev(w') \in \ev(\mA)$.
\end{proof}

\lemgrp*
\begin{proof}
    Let $w$ be an Identity Traversal, then $\ev(w) = (0, 0^n)$. This shows that the graph $\Gamma(w)$ represents the element $0 \in \mY$. Furthermore, $\Gamma(w)$ is a path starting at the vertex $(b_1, 0^n)$ and ends in $(b_1, 0^n + 0^n)$, hence it is Eulerian.
    Since $w$ contains transitions of every label, $\Gamma(w)$ is full-image.
    
    Let $\Gamma$ be a full-image Eulerian $\mA$-graph that represents $0$. By translating $\Gamma$ we can suppose $\Gamma$ to contain the vertex $(b_1, 0^n)$, this does not change the fact that $\Gamma$ represents $0$.
    Let $P$ be an Eulerian circuit of $\Gamma$ that starts and ends in $(b_1, 0^n)$. We trace the labels $\ell_1, \ldots, \ell_p,$ of edges in $P$ to obtain a run $w(P) = \delta_{\ell_1} \delta_{\ell_2} \cdots \delta_{\ell_p}$ in $\mA$. The run $w(P)$ is accepting because the $P$ starts and ends in the lattice $\Lambda_1$. Let $(y, z) \in \mY \rtimes \Z^n$ denote the evaluation $\ev(w(P))$. We have $y = 0$ because $\Gamma$ represents $0 \in \mY$. We have $z = \sum_{e \in E(\Gamma)} a_{\lambda(e)} = 0^n$ because the $P$ is an Eulerian circuit.
    Therefore $P$ is an accepting run that evaluates to the neutral element. Furthermore, $P$ uses every transition at least once because $\Gamma$ is full-image, so $P$ is an Identity Traversal.
\end{proof}

\propgtoe*
\begin{proof}
    If the automaton $\mA$ admits an Identity Traversal, then Lemma~\ref{lem:grptoeul} shows there exists a full-image Eulerian $\mA$-graph representing $0$: this Eulerian graph is symmetric and face-accessible.

    If there is a full-image symmetric face-accessible Eulerian $\mA$-graph $\Gamma$ representing $0$, then by Theorem~\ref{thm:acctocon}, some a union of translations $\hG \coloneqq \bigcup_{i = 1}^m \Gamma + (0^s, z_i)$ is an Eulerian graph.
    The graph $\hG$ represents the element $\sum_{i = 1}^m \oX^{z_i} \cdot 0 = 0$, and it is full-image because it contains the full-image graph $\Gamma + (0^s, z_1)$ as a subgraph.
    Therefore, by Lemma~\ref{lem:grptoeul}, the automaton $\mA$ admits an Identity Traversal.
\end{proof}

\lemthreetoeq*
\begin{proof}
    (i) $\Gamma$ is full-image if and only if each label appears at least once, meaning $f_{\ell} \neq 0$ for all $\ell$.

    (ii) For $i \in \{1, \ldots, s\}$, we have 
    \begin{equation*}
        \sum_{\ell \colon \Omega(\ell) = i} f_{\ell} = \sum_{\ell \colon \Omega(\ell) = i,} \sum_{e \in E(\Gamma), \lambda(e) = \ell} \oX^{\pi_{\Z^n}(\sigma(e))}
        = \sum_{e \in E(\Gamma), \sigma(e) \in \Lambda_i} \oX^{\pi_{\Z^n}(\sigma(e))},
    \end{equation*}
    and
    \begin{multline*}
        \sum_{\ell \colon \Delta(\ell) = i} f_{\ell} \cdot \oX^{a_{\ell}} = \sum_{\ell \colon \Delta(\ell) = i,} \sum_{e \in E(\Gamma), \lambda(e) = \ell} \oX^{\pi_{\Z^n}(\sigma(e))} \cdot \oX^{a_{\ell}} \\
        = \sum_{\ell \colon \Delta(\ell) = i,} \sum_{e \in E(\Gamma), \lambda(e) = \ell} \oX^{\pi_{\Z^n}(\tau(e))}
        = \sum_{e \in E(\Gamma), \tau(e) \in \Lambda_i} \oX^{\pi_{\Z^n}(\tau(e))}.
    \end{multline*}
    These two sums are equal if and only if the in-degree equals the out-degree at every vertex in $\Lambda_i$.
    Therefore, $\Gamma$ is symmetric if and only if $\sum_{\ell \colon \Omega(\ell) = i} f_{\ell} = \sum_{\ell \colon \Delta(\ell) = i} f_{\ell} \cdot \oX^{a_{\ell}}$ holds for all $i \in \{1, \ldots, s\}$.

    (iii) $\Gamma$ represents the element 
    \[
    \sum_{e \in E(\Gamma)} \oX^{\pi_{\Z^n}(\sigma(e))} \cdot y_{\lambda(e)} = \sum_{\ell = 1}^t \sum_{e \in E(\Gamma), \lambda(e) = \ell} \oX^{\pi_{\Z^n}(\sigma(e))} \cdot y_{\ell} = \sum_{\ell = 1}^t f_{\ell} \cdot y_{\ell},
    \]
    which is $0$ if and only if $\sum_{\ell = 1}^t f_{\ell} \cdot y_{\ell} = 0$.
\end{proof}

\lemacctoeq*
\begin{proof}
    We show the contrapositive: the convex hull $C$ of $V(\Gamma)$ has a non-accessible face if and only if $\exists (S, \mT, \rho), \exists v \in \Rns,$ such that 
    \[
    \left( O_{v}(S, \mT, \rho) \cup \Delta^{-1}(S)^{\complement} \right) \cap M_{v}((S, \mT, \rho), \bff) = \emptyset.
    \]

    \begin{enumerate}[nosep, label = \arabic*.]
        \item 
    Suppose $C$ has a non-accessible face $F$. Then every edge of $\Gamma$ that starts inside $F$ and ends inside $F$.
    We will show that there exists a partial contraction $(S, \mT, \rho)$ and some vector $v \in \Rns$, such that
    $\left( O_{v}(S, \mT, \rho) \cup \Delta^{-1}(S)^{\complement} \right) \cap M_{v}((S, \mT, \rho), \bff) = \emptyset$.
    
    Let $(b, v) \in \Rsns$ be such that $F$ is the extremal face of $C$ at direction $(b, v)$.
    That is, $F = \{x \in C \mid \forall x' \in C, (b, v) \cdot x \geq (b, v) \cdot x'\}$.
    Let $\widetilde{S}$ denote the set of indices $i$ such that $V(\Gamma) \cap F \cap \Lambda_i \neq \emptyset$.
    Let $E_{(b, v)}$ denote the set of edges in $\Gamma$ whose source and target are both in $F$, and let $\widetilde{\mT}_{(b, v)}$ be the set of labels appearing in $E_{(b, v)}$. Consider the subautomaton $\widetilde{\mA}$ of $\mA$ whose set of states is $\{q_i \mid i \in \widetilde{S}\}$ and whose set of transitions is $\{\delta_{\ell} \mid \ell \in \widetilde{\mT}_{(b, v)}\}$.
    Choose any $S \subseteq \widetilde{S}$ such that $\{q_i \mid i \in S\}$ is a connected component of $\widetilde{\mA}$, and choose $\mT \subseteq \widetilde{\mT}_{(b, v)}$ such that $\{\delta_{\ell} \mid \ell \in \mT\}$ is an undirected spanning tree of this connected component.
    Let $\rho$ be any element of $S$, we will show that
    \[
    \left( O_{v}(S, \mT, \rho) \cup \Delta^{-1}(S)^{\complement} \right) \cap M_{v}((S, \mT, \rho), \bff) = \emptyset.
    \]
    Take any $\ell \in M_{v}((S, \mT, \rho), \bff) \subseteq \Omega^{-1}(S)$, we will show that $\ell \not\in \Delta^{-1}(S)^{\complement}$ and $\ell \not\in O_{v}(S, \mT, \rho)$. 
    By the definition of $M_{v}((S, \mT, \rho), \bff)$, there is a monomial $c \oX^{z}$ appearing in $f^{(S, \mT, \rho)}_{\ell}$, such that $v \cdot z$ is maximal among all monomials appearing in $f_{(S, \mT, \rho), \ell'}, \ell' \in \Omega^{-1}(S)$.
    This yields an edge $e$ with label $\ell$, starting at $(b_{\Omega(\ell)}, z) \in F \cap \Lambda_{\Omega(\ell)}$ with $\Omega(\ell) \in S$.

    \begin{enumerate}[nosep, label = (\roman*)]
        \item First we show $\ell \not\in \Delta^{-1}(S)^{\complement}$, which is equivalent to $\Delta(\ell) \in S$.
    
    Since $\Omega(\ell) \in S$, all edges starting in $V(\Gamma) \cap (F \cap \Lambda_{\Omega(\ell)})$ end in $F$.
    Since $\{q_i \mid i \in S\}$ is a connected component of $\widetilde{\mA}$, all all edges starting in $V(\Gamma) \cap (F \cap \Lambda_{\Omega(\ell)})$ actually ends in $\bigcup_{i \in S} (F \cap \Lambda_{i})$. Therefore $\Delta(\ell) \in S$, so $\ell \not\in \Delta^{-1}(S)^{\complement}$.

    \item Next we show $\ell \not\in O_{v}(S, \mT, \rho)$, which is equivalent to $a^{(S, \mT, \rho)}_{\ell} \perp v$. 
    If $\Delta(\ell) \notin S$ then $a^{(S, \mT, \rho)}_{\ell} = 0^n \perp v$, so consider the case where $\Delta(\ell) \in S$.
    By the definition of $\mT \subseteq \widetilde{\mT}_{(b, v)}$, every label $\ell' \in \mT$ satisfy $(b_{\Delta(\ell')} - b_{\Omega(\ell')}, a_{\ell'}) = \tau(e) - \sigma(e) \perp (b, v)$, where $e$ is an edge contained within $F$ with label $\ell'$.
    Let $P_{\Omega(\ell)}$ be the path consisting of transitions in $\{\delta_{l} \mid l \in \mT\} \cup \{\delta_{l}^- \mid l \in \mT\}$ that connects from $q_{\Omega(\ell)}$ to $q_{\rho}$, and write $\ev({P_{\Omega(\ell)}})$ as $(y_{P_{\Omega(\ell)}}, z_{P_{\Omega(\ell)}})$.
    Then
    \begin{equation}\label{eq:acc1}
        (b_{\rho} - b_{\Omega(\ell)}, z_{P_{\Omega(\ell)}}) = \sum_{\delta_{l}^{\pm} \in {P_{\Omega(\ell)}}} \pm (b_{\Delta(l)} - b_{\Omega(l)}, a_{l}) \perp (b, v).
    \end{equation}
    Similarly, let $P_{\Delta(\ell)}$ be the path consisting of transitions in $\{\delta_{l} \mid l \in \mT\} \cup \{\delta_{l}^- \mid l \in \mT\}$ that connects from $q_{\Delta(\ell)}$ to $q_{\rho}$, and write $\ev(P_{\Delta(\ell)})$ as $(y_{P_{\Delta(\ell)}}, z_{P_{\Delta(\ell)}})$. Then we have
    \begin{equation}\label{eq:acc2}
        (b_{\rho} - b_{\Delta(\ell)}, z_{P_{\Delta(\ell)}}) = \sum_{\delta_{l}^{\pm} \in {P_{\Delta(\ell)}}} \pm (b_{\Delta(l)} - b_{\Omega(l)}, a_{l}) \perp (b, v).
    \end{equation}
    Furthermore, since both the source and target of the edge $e$ (with label $\ell$) are in $F$, we have
    \begin{equation}\label{eq:acc3}
        (b_{\Delta(\ell)} - b_{\Omega(\ell)}, a_{\ell}) \perp (b, v).
    \end{equation}
    Taking \eqref{eq:acc2}$+$\eqref{eq:acc3}$-$\eqref{eq:acc1} yields $(0^s, z_{P_{\Delta(\ell)}} + a_{\ell} - z_{P_{\Omega(\ell)}}) \perp (b, v)$. 
    Therefore $a^{(S, \mT, \rho)}_{\ell} = (a_{\ell} + z_{P_{\Delta(\ell)}} - z_{P_{\Omega(\ell)}}) \perp v$.
    \end{enumerate}

    \item Suppose there exists a partial contraction $(S, \mT, \rho)$ and some vector $v \in \Rns$, such that
    \[
    \left( O_{v}(S, \mT, \rho) \cup \Delta^{-1}(S)^{\complement} \right) \cap M_{v}((S, \mT, \rho), \bff) = \emptyset.
    \]
    We will show that some strict face of $C$ is not accessible.
    For a set $X$ we denote by $\conv(X)$ its convex hull.

    For $i = 1, \ldots, s$, define $C_i \coloneqq \conv(V(\Gamma) \cap \Lambda_i)$, and let $F_i$ be the extremal face of $C_i$ at direction $(0^s, v)$.
    Let $b = (\beta_1, \ldots, \beta_s) \in \R^s$ be such that 
    \begin{enumerate}[nosep, label = (\roman*)]
        \item $(b, v) \perp (b_{\Delta(\ell)} - b_{\Omega(\ell)}, a_{\ell})$ for all $\ell \in \mT$.
        \item the extremal face $F$ of $\conv(V(\Gamma))$ at direction $(b, v)$ satisfies $V(\Gamma) \cap F = \bigcup_{i \in S} (V(\Gamma) \cap F_i)$.
        In other words, $F$ is so that its intersection with $\conv(V(\Gamma))$ contains and only contains vertices in the lattices $\Lambda_i, i \in S$.
    \end{enumerate}
    Such $b$ can always be found. Indeed, since transitions with label $\mT$ form a tree with states of index $S$, we have $|\mT| = |S| - 1$.
    So condition (i) can be satisfied by choosing the coordinates $\beta_i, i \in S$.
    Then, condition (ii) can be satisfied by choosing the coordinates $\beta_i, i \notin S$ to be sufficiently small compared to $\beta_i, i \in S$.
    
    We will show that $F$ is not accessible.
    Let $e$ be an edge with label $\ell$ starting in $(b_{\Omega(\ell)}, z) \in V(\Gamma) \cap F$.
    \begin{enumerate}[nosep, label = (\roman*)]
        \item First, we show $\ell \in M_{v}((S, \mT, \rho), \bff)$.
        Since $\sigma(e) \in \Lambda_i$ with $i \in S$, we have $\ell \in \Omega^{-1}(S)$.
        Take any label $\ell' \in \Omega^{-1}(S)$ and let $c' \oX^{z'}$ be the monomial with largest $\deg_v$ in $f_{\ell'}$. Then there is an edge $e'$ starting in $(b_{\Omega(\ell')}, z')$.
        Since $V(\Gamma) \cap F = \bigcup_{i \in S} (V(\Gamma) \cap F_i) \supseteq V(\Gamma) \cap F_{\Omega(\ell')}$, there is an edge $e''$ with $\sigma(e'') = (b_{\Omega(\ell')}, z'') \in F_{\Omega(\ell')}$.
        Therefore $(\beta, v) \cdot (b_{\Omega(\ell')}, z'') \geq (\beta, v) \cdot (b_{\Omega(\ell')}, z')$, yielding
        \begin{equation}
            v \cdot z'' \geq v \cdot z'.
        \end{equation}
        Let $P_{\Omega(\ell)}$ be the path consisting of transitions in $\{\delta_{l} \mid l \in \mT\} \cup \{\delta_{l}^- \mid l \in \mT\}$ that connects from $q_{\Omega(\ell)}$ to $q_{\rho}$, and write $\ev({P_{\Omega(\ell)}})$ as $(y_{P_{\Omega(\ell)}}, z_{P_{\Omega(\ell)}})$.
        Then
        \begin{equation}\label{eq:acc5}
            (b_{\rho} - b_{\Omega(\ell)}, z_{P_{\Omega(\ell)}}) = \sum_{\delta_{l}^{\pm} \in {P_{\Omega(\ell)}}} \pm (b_{\Delta(l)} - b_{\Omega(l)}, a_{l}) \perp (b, v).
        \end{equation}
        Similarly, let $P_{\Omega(\ell')}$ be the path consisting of transitions in $\{\delta_{l} \mid l \in \mT\} \cup \{\delta_{l}^- \mid l \in \mT\}$ that connects from $q_{\Omega(\ell')}$ to $q_{\rho}$, and write $\ev({P_{\Omega(\ell')}})$ as $(y_{P_{\Omega(\ell')}}, z_{P_{\Omega(\ell')}})$.
        Then
        \begin{equation}\label{eq:acc6}
            (b_{\rho} - b_{\Omega(\ell')}, z_{P_{\Omega(\ell')}}) = \sum_{\delta_{l}^{\pm} \in {P_{\Omega(\ell')}}} \pm (b_{\Delta(l)} - b_{\Omega(l)}, a_{l}) \perp (b, v).
        \end{equation}
        Furthermore, since both $(b_{\Omega(\ell)}, z)$ and $(b_{\Omega(\ell')}, z'')$ are in $F$, we have
        \begin{equation}\label{eq:acc7}
            (b_{\Omega(\ell')} - b_{\Omega(\ell)}, z'' - z) \perp (b, v).
        \end{equation}
        Computing \eqref{eq:acc6}$+$\eqref{eq:acc7}$-$\eqref{eq:acc5} yields $(0^s, z_{P_{\Omega(\ell')}} + z'' - z - z_{P_{\Omega(\ell)}}) \perp (b, v)$.
        Therefore,
        \begin{equation*}
            \deg_v(f^{(S, \mT, \rho)}_{\ell}) = v \cdot (z + z_{P_{\Omega(\ell)}}) = v \cdot (z'' + z_{P_{\Omega(\ell')}}) \geq v \cdot (z' + z_{P_{\Omega(\ell')}}) = \deg_v(f^{(S, \mT, \rho)}_{\ell'}).
        \end{equation*}
        So we indeed have $\ell \in M_{v}((S, \mT, \rho), \bff)$.

        \item Next, we show that the target $\tau(e)$ of $e$ must be in $F$.
        Indeed, since $\ell \in M_{v}((S, \mT, \rho), \bff)$ and $\left( O_{v}(S, \mT, \rho) \cup \Delta^{-1}(S)^{\complement} \right) \cap M_{v}((S, \mT, \rho), \bff) = \emptyset$, we have $\ell \notin O_{v}(S, \mT, \rho)$ and $\ell \notin \Delta^{-1}(S)^{\complement}$.
        The condition $\ell \notin \Delta^{-1}(S)^{\complement}$ shows $\Delta(\ell) \in S$, so $\tau(e) \in \bigcup_{i \in S} (V(\Gamma) \cap \Lambda_i)$.
        The $\ell \notin O_{v}(S, \mT, \rho)$ shows $a_{\ell} + z_{P_{\Delta(\ell)}} - z_{P_{\Omega(\ell)}} = a^{(S, \mT, \rho)}_{\ell} \perp v$.
        Therefore,
        \begin{multline*}
            \tau(e) - \sigma(e) = (b_{\Delta(\ell)} - b_{\Omega(\ell)}, a_{\ell}) = (b_{\Delta(\ell)} - b_{\Omega(\ell)}, z_{P_\Omega(\ell)} - z_{P_\Delta(\ell)} + a^{(S, \mT, \rho)}_{\ell}) \\
            = (b_{\rho} - b_{\Omega(\ell)}, z_{P_{\Omega(\ell)}}) - (b_{\rho} - b_{\Delta(\ell)}, z_{P_{\Delta(\ell)}}) + (0^s, a^{(S, \mT, \rho)}_{\ell}) \perp (b, v)
        \end{multline*}
        by Equations~\eqref{eq:acc5} and \eqref{eq:acc6}.
        Since $\sigma(e) \in F$ and $F$ is the extremal face at direction $(b, v)$, we conclude that $\tau(e) \in F$.
    \end{enumerate}
    We have thus shown that every edge $e$ starting in $V(\Gamma) \cap F$ ends in $F$. Therefore $F$ is not accessible.
    \end{enumerate}
\end{proof}

\propgraphtoeq*
\begin{proof}
    Suppose there exists a full-image symmetric face-accessible $\mA$-graph $\Gamma$.
    Take the tuple of position polynomials $\bff$ of $\Gamma$.
    Then by Lemma~\ref{lem:threetoeq}, we have $\bff \in \left(\N[\oX^{\pm}]^*\right)^t$ and $\bff \in \mM_{\Z}$.
    Therefore $\tbf \in \tMZ \cap \left(\N[\oX^{\pm}]^*\right)^K$.
    Since $\Gamma$ is face-accessible, the tuple $\bff$ satisfies the condition~\eqref{eq:acccond} in Lemma~\ref{lem:acctoeq} for every partial contraction $(S, \mT, \rho) \in \mPC$.
    This is equivalent to $\tbf$ satisfying condition~\eqref{eq:gen}.

    Suppose there exists $\tbf \in \tMZ \cap \left(\N[\oX^{\pm}]^*\right)^K$ that satisfies condition~\eqref{eq:gen}.
    We recover $\bff \in \mM_{\Z} \cap \left(\N[\oX^{\pm}]^*\right)^t$ as a sub-tuple of $\tbf$.
    There exists an $\mA$-graph $\Gamma$ whose position polynomials are $\bff$.
    Since $\bff \in \mM_{\Z} \cap \left(\N[\oX^{\pm}]^*\right)^t$, it satisfies the conditions (i), (ii) and (iii) in Lemma~\ref{lem:threetoeq}.
    Therefore $\Gamma$ is full-image symmetric and represents $0$.
    Finally, since $\tbf$ satisfies condition~\eqref{eq:gen}, the tuple $\bff$ satisfies the condition~\eqref{eq:acccond} in Lemma~\ref{lem:acctoeq} for every partial contraction $(S, \mT, \rho) \in \mPC$.
    Therefore, $\Gamma$ is face-accessible.
\end{proof}

\lemM*
\begin{proof}
    An element $\tbf \in \tMZ \cap \left(\N[\oX^{\pm}]^*\right)^K$ satisfying Property~\eqref{eq:gen} is obviously an element in $\mM \cap \left(\Rp[\oX^{\pm}]^*\right)^K$.
    Therefore it suffices to prove the ``if'' implication.

    Suppose we have an element $\bff \in \mM \cap \left(\Rp[\oX^{\pm}]^*\right)^K$ satisfying Property~\eqref{eq:gen}.
    Write $\bff = (f_1, \ldots, f_K)$ where for $i = 1, \ldots, K$,
    \[
    f_i = \sum_{b \in B_i} c_{i, b} \oX^{b}.
    \]
    Here, the \emph{support} $B_i$ is a non-empty finite subset of $\Z^n$, and $c_{i, b} \in \Rpp$ for all $b \in B_i$.
    Since Property~\eqref{eq:gen} depends only on the supports $B_1, \ldots, B_K$, it suffices to show that there exists $\tbf = (\tf_1, \ldots, \tf_K) \in \tMZ \cap \left(\N[\oX^{\pm}]^*\right)^K$ where
    \[
    \tf_i = \sum_{b \in B_i} \tc_{i, b} \oX^{b}
    \]
    and $\tc_{i, b} \in \Zpp$ for all $b \in B_i$.
    
    Since $\bff \in \mM$, we have $\bff = \sum_{j = 1}^m h_j \cdot \bg_j$ for some $h_1, \ldots, h_m \in \R[\oX^{\pm}]$.
    For each $j \in \{1, \ldots, m\}$, write $h_j = \sum_{b \in H_j} h_{j, b} \oX^{b}$, where $H_j$ is a finite subset of $\Z^n$.
    Then the equation $\bff = \sum_{j = 1}^m h_j \cdot \bg_j$ can be rewritten as a finite system of linear equations over $\R$, where the left hand sides are $0$ or $c_{i, b}, b \in B_i, i = 1, \ldots, K$, and the right hand sides are $\Z$-linear combinations of the variables $h_{j, b}, j \in \{1, \ldots, m\}, b \in H_j$ (because the coefficients of $\bg_j$ are integers for all $j$).

    Since this system of linear equations is homogeneous and the coefficients are all in $\Z$, it has a solution $h_{j, b} \in \R, j \in \{1, \ldots, m\}, b \in H_j$ and $c_{i, b} \in \Rpp, b \in B_i, i = 1, \ldots, K,$ if and only if it has a solution with $h_{j, b} \in \Q, c_{i, b} \in \Qpp$ for all $i, j, b$.
    By multiplying all $h_{j, b}, c_{i, b}$ with their common denominator, we obtain a solution $\widetilde{h}_{j, b} \in \Z, \tc_{i, b} \in \Zpp$ for all $i, j, b$.
    Then, $\tf_i \coloneqq \sum_{b \in B_i} \tc_{i, b} \oX^{b}, i = 1, \ldots, K$ and $\widetilde{h}_j = \sum_{b \in H_j} \widetilde{h}_{j, b} \oX^{b}, j = 1, \ldots, m,$ satisfy $\tbf = \sum_{j = 1}^m \widetilde{h}_j \cdot \bg_j$.
    Hence, $\tbf = (\tf_1, \ldots, \tf_K) \in \tMZ \cap \left(\N[\oX^{\pm}]^*\right)^K$.
    The element $\tbf$ satisfies Property~\eqref{eq:gen} since the condition depends only on the supports $B_1, \ldots, B_K$.
\end{proof}

\section{Proof of Theorem~\ref{thm:acctocon}}\label{app:proofacctocon}

In this appendix we prove Theorem~\ref{thm:acctocon}. Figure~\ref{fig:connect} illustrates the main steps of the proof.

{\renewcommand\footnote[1]{}\thmacctoeul*}

\begin{figure}[ht!]
    \centering
        \includegraphics[width=\textwidth,height=0.9\textheight,keepaspectratio, trim={1.5cm 3.5cm 3cm 1.5cm},clip]{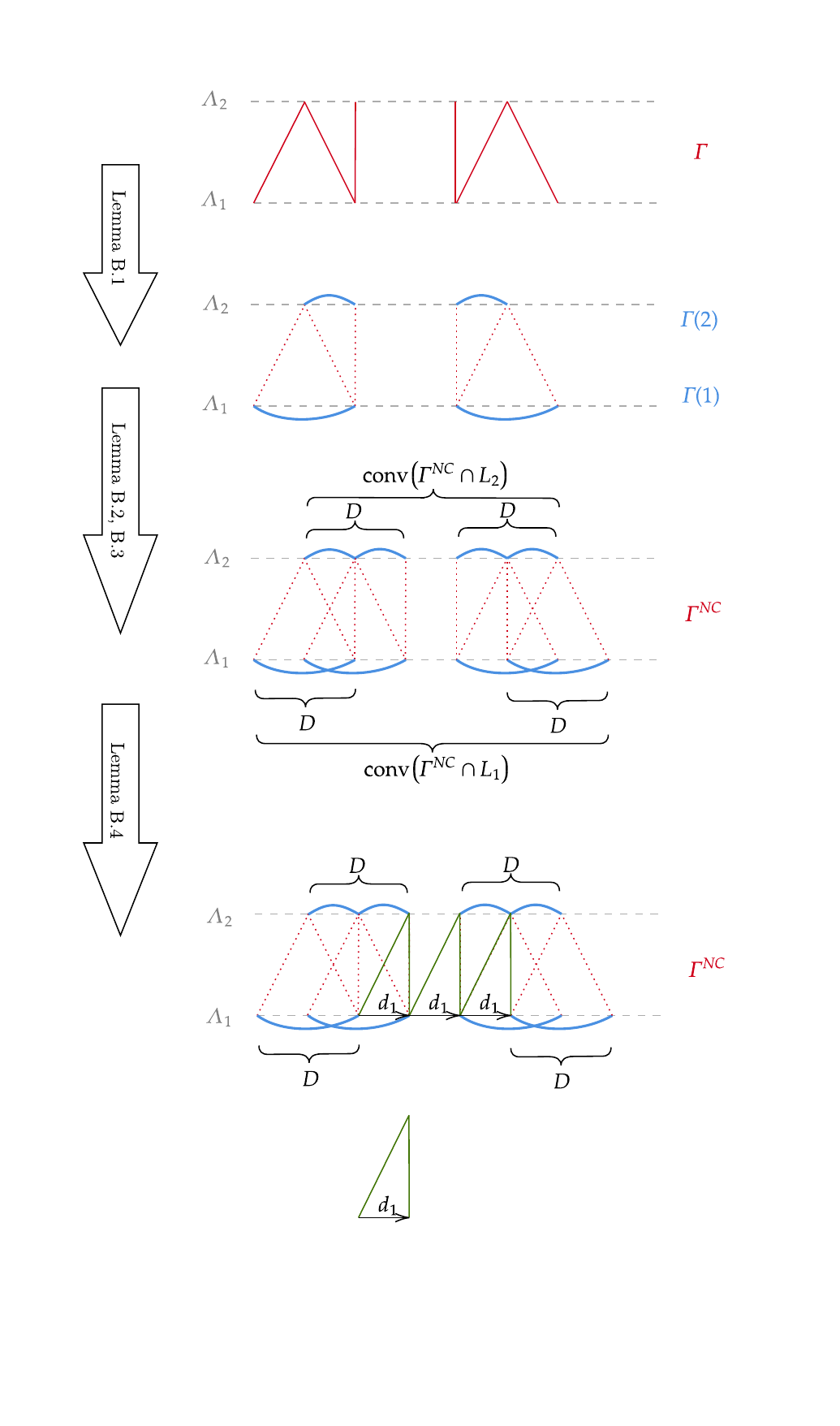}
        \caption{Proof of Theorem~\ref{thm:acctocon}}
        \label{fig:connect}
\end{figure}

Our main strategy is to reduce to the case $s = 1$, which has already been proved in~\cite[Theorem~3.3]{dong2024semigroup}.
Let $\Gamma$ be a full-image, symmetric and face-accessible $\mA$-graph.
Recall that a directed graph is Eulerian if and only if it is symmetric and connected (for symmetric graphs, strong and weak connectivity are equivalent).
Since $\Gamma$ is full-image and symmetric, any union of translations $\hG = \bigcup_{i = 1}^m \Gamma + (0^s, z_i)$ is also full-image and symmetric when $m \geq 1$.
Therefore it suffices to find $z_1, \ldots, z_m \in \Z^n$ such that $\hG = \bigcup_{i = 1}^m \Gamma + (0^s, z_i)$ is connected. 

For an edge $e$ from $\sigma(e)$ to $\tau(e)$, we denote by $e^-$ its \emph{inverse}, that is, an edge from $\tau(e)$ to $\sigma(e)$.
Since strong and weak connectivity are equivalent for symmetric graphs, for each edge $e$ in $\Gamma$ we add its inverse $e^-$ into $\Gamma$, this will not change the connectivity of $\Gamma$ or the eventually constructed $\hG$.
Hence, we can suppose without loss of generality that for each $e \in E(\Gamma)$ we have added $e^-$ in $E(\Gamma)$, and $\Gamma$ is actually an $\mA^{\pm}$-graph.

For a sequence sequence of polytopes $P_1, P_2, \ldots, P_m \subseteq \R^n$, we define their \emph{Minkowski sum}
\[
P_1 + P_2 + \cdots + P_m \coloneqq \{p_1 + p_2 + \cdots + p_m \mid p_i \in P_i, i = 1, \ldots, m\}.
\]
For simplicity we denote by $mP$ the sum $\underbrace{P + P + \cdots + P}_{m \text{ times}}$. Since $P$ is convex, $mP$ is also equal to the set $\{m \cdot p \mid p \in P\}$. 

For each $i \in \{1, \ldots, s\}$, consider the set of vertices $V(\Gamma) \cap \Lambda_i$: these are the vertices of $\Gamma$ appearing in the lattice $\Lambda_i$.
The convex hull of $V(\Gamma) \cap \Lambda_i$ is a polytope of the form $\{b_i\} \times C_i$, where $C_i$ is a polytope in $\R^n$.

Consider the Minkowski sum $C \coloneqq C_1 + \cdots + C_s$. Then $C$ is of dimension $n$ since $\mA$ is primitive and $\Gamma$ is full-image.
Indeed, if $C$ is of dimension less than $n$, then $C_1, \ldots, C_s$ are all contained in the some hyperplane $H \subsetneq \R^n$.
The vertices in $\Lambda_1$ reachable by any concatenation of edges will be contained in $H$, which contradicts the definition of primitiveness of $\mA$.

Consider the $\mA$-graph 
\[
\Gamma^{C} \coloneqq \bigcup_{z \in C \cap \Z^n}  \Gamma + (0^s, z).
\]
Then for each $i \in \{1, \ldots, s\}$, the convex hull $\conv(\Gamma^{C} \cap \Lambda_i)$ is equal to $\{b_i\} \times (C + C_i)$.

\begin{lemma}\label{lem:acctoacc}
    Fix $i \in \{1, \ldots, s\}$. Let $F$ be a strict face of $C + C_i \subsetneq \R^n$. Then there exists a path in $\Gamma^{C}$ starting in $\{b_i\} \times F$ and ending in $\{b_i\} \times ((C + C_i) \setminus F)$.
\end{lemma}
\begin{proof}
    See Figure~\ref{fig:concat} and \ref{fig:concat2} for an illustration of this proof.    
    Without loss of generality suppose $i = 1$. 
    Let $v \in \Rns$ be such that $F$ is the extremal face of $C + C_1$ at direction $v$.
    That is, $F = \{x \in C + C_1 \mid \forall x' \in C + C_1, v \cdot x \geq v \cdot x'\}$.
    Let $F_1, F_2, \ldots, F_s$ be respectively the extremal face of $C_1, C_2, \ldots, C_s$ at direction $v$, then $F_1 + F_2 + \cdots + F_s + F_1 \subseteq F$.
    Indeed, take $x_1 \in F_1, \ldots, x_s \in F_s, x_{s+1} \in F_1$ and any $x' \in C + C_1 = C_1 + C_2 + \cdots + C_s + C_1$.
    Write $x' = x'_1 + \cdots + x'_s + x'_{s+1}$ with $x'_1 \in C_1, \ldots, x'_s \in C_1, x'_{s+1} \in C_1$, then we have $v \cdot x_1 \geq v \cdot x'_1, \ldots, v \cdot x_{s+1} \geq v \cdot x'_{s+1}$.
    Therefore $v \cdot (x_1 + \cdots + x_{s+1}) \geq v \cdot x'$ for every $x' \in C + C_1$, which yields $x_1 + \cdots + x_{s+1} \in F$.
    This shows $F_1 + F_2 + \cdots + F_s + F_1 \subseteq F$.
    
    We claim that we can find a sequence of edges $e_1, \ldots, e_M$ in $E(\Gamma)$ satisfying the following conditions. (See Figure~\ref{fig:concat} for an illustration.)
    \begin{enumerate}[nosep, label = (\roman*)]
        \item $\sigma(e_1) \in \{b_1\} \times F_1$.
        \item For each $j = 2, \ldots, M$, there exists an index $i_j \in \{2, \ldots, s\}$ such that the source vertex $\sigma(e_j)$ is in $\{b_{i_j}\} \times F_{i_j}$.
        \item The indices $i_1 \coloneqq 1, i_2, i_3, \ldots, i_M$ are distinct.
        \item For $j = 2, \ldots, M$, the target vertex $\tau(e_{j-1})$ is in the lattice $\Lambda_{i_j}$.
        \item there exists $j \in \{1, \ldots, M\}$ such that $\tau(e_{j-1})$ is in $\{b_{i_j}\} \times \left(C_{i_j} \setminus F_{i_j} \right)$.
        \item $\tau(e_M) \in \Lambda_{i_m}$ for some $1 \leq m \leq M$.
    \end{enumerate}

    \begin{figure}[ht!]
        \centering
        \begin{minipage}[t]{.47\textwidth}
            \centering
            \includegraphics[width=\textwidth,height=1.0\textheight,keepaspectratio, trim={5.8cm 0.8cm 2cm 0cm},clip]{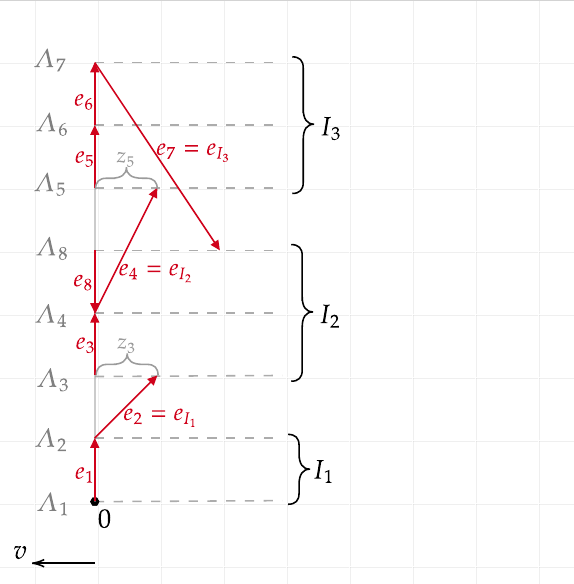}
            \caption{edges $e_i$ in the graph $\Gamma$}
            \label{fig:concat}
        \end{minipage}
        \hfill
        \begin{minipage}[t]{0.47\textwidth}
            \centering
            \includegraphics[width=\textwidth,height=1.0\textheight,keepaspectratio, trim={5.8cm 0.8cm 2cm 0cm},clip]{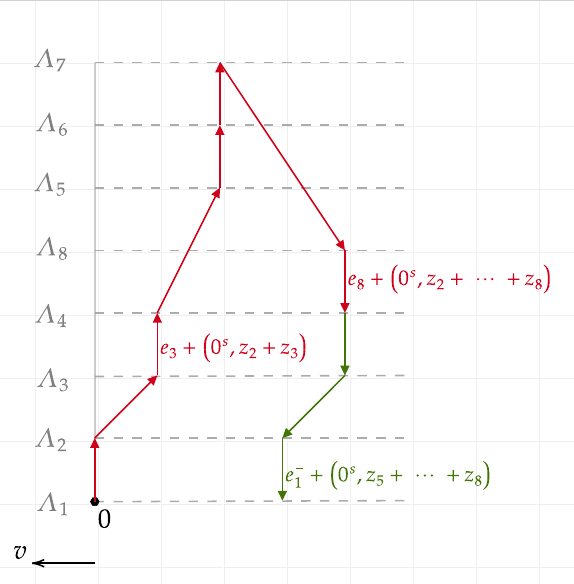}
            \caption{the concatenated path $P$ in $\Gamma^{C}$}
            \label{fig:concat2}
        \end{minipage}
    \end{figure}
    
    For $i, j \in \{1, \ldots, s\}$, we write $i \leftrightarrow j$ if there is an edge in $E(\Gamma)$ between $\{b_i\} \times F_i$ and $\{b_j\} \times F_j$. Let $\sim$ denote the transitive closure of the relation $\leftrightarrow$.
    That is, $i \sim j$ if and only if there exists a sequence of indices $l_1 = i, l_2, l_3, \ldots, l_p = j$ such that for $k = 1, \ldots, p-1$, some edge $e_k \in E(\Gamma)$ starts in $\{b_{l_k}\} \times F_{l_k}$ and ends in $\{b_{l_{k+1}}\} \times F_{l_{k+1}}$.
    
    We perform the following procedure. Let $I_1$ be the equivalent class of $\{1, \ldots, s\}/ \sim$ containing $1$. Take a closed halfspace $\mH$ of $\R^{s+n}$ such that $\mH \cap \conv(\Gamma) \cap \Lambda = \bigcup_{i \in I_1} \{b_i\} \times F_i$.
    Since the strict face $\mH \cap \conv(\Gamma)$ is accessible, there exists an edge $e_{I_1}$ in $\Gamma$ from $\bigcup_{i \in I_1} \left(\{b_i\} \times F_i \right)$ to $(\conv(\Gamma) \cap \Lambda) \setminus \bigcup_{i \in I_1} \left(\{b_i\} \times F_i \right)$. Since $\sim$ is the transitive closure of $\leftrightarrow$, the target of $e_{I_1}$ must be in $\{b_j\} \times (C_j \setminus F_j)$ for some $j$. Otherwise, $\tau(e_{I_1}) \in \{b_j\} \times F_j$ for some $j$, so $j \in I_1$ by the definition of $I_1$, contradicting $\tau(e_{I_1}) \notin \bigcup_{i \in I_1} \left(\{b_i\} \times F_i \right)$.

    If $j \in I_1$ then we stop the procedure, otherwise we denote by $I_2$ the $\sim$-equivalence class of $j$ and repeat the above procedure. We stop repeating the procedure when we have found equivalence classes $I_1, I_2, \ldots, I_T$ and an edge $e_{I_T}$, such that the target of $e_{I_T}$ is in $\{b_j\} \times (C_j \setminus F_j)$ for some $j \in I_t$ with $1 \leq t \leq T$.
    
    For $k = 1, \ldots, T-1,$ there exists an edge $e_{I_k} \in E(\Gamma)$ from $\bigcup_{i \in I_k} F_i$ to $\bigcup_{i \in I_{k+1}} (C_i \setminus F_i)$.
    Also, there exists an edge $e_{I_T} \in E(\Gamma)$ from $\bigcup_{i \in I_T} F_i$ to $\bigcup_{i \in I_{t}} (C_i \setminus F_i)$.
    The vertices $\tau(e_{I_{k-1}})$ and $\sigma(e_{I_{k}})$ are both in $\bigcup_{i \in I_k} \left( \{b_i\} \times F_i \right)$, so there exists a distinct sequence of indices $i_{k, 1}, \ldots, i_{k, {l_k}} \in I_k$, such that $\tau(e_{I_{k-1}}) \in \{b_{i_{k, 1}}\} \times F_{i_{k, 1}}, \sigma(e_{I_{k}}) \in \{b_{i_{k, l_k}}\} \times F_{i_{k, {l_k}}}$, and for $j = 1, \ldots, l_k - 1,$ there exists an edge $e_{i_{k,j}}$ from $\{b_{i_{k, j}}\} \times F_{i_{k,j}}$ to $\{b_{i_{k, j+1}}\} \times F_{i_{k, {j+1}}}$.

    Thus we get a sequence of edges 
    \begin{equation}\label{eq:sequence}
    e_{i_{1, 1}}, \ldots, e_{i_{1, l_1}}, e_{I_1}, e_{i_{2, 1}}, \ldots, e_{i_{2, l_2}}, e_{I_2}, \;\; \ldots \;\;, e_{i_{T, 1}}, \ldots, e_{i_{T, l_T}}, e_{I_T},
    \end{equation}
    from the above procedure.
    We extend this sequence in the following way.
    Recall that $\tau(e_{I_T}) \in \{b_j\} \times (C_j \setminus F_j)$ for some $j \in I_t$ with $1 \leq t \leq T$.
    If $j$ already appears among the indices $\{i_{t, 1}, \ldots, i_{t, l_t}\}$, then we do not extend the sequence~\eqref{eq:sequence}.
    Otherwise, since $j, i_{t, l_t} \in I_t$, there exists a sequence of indices $i_{T+1, 1}, \ldots, i_{T+1, l_{T+1}} \in I_t \setminus \{i_{t, 1}, \ldots, i_{t, {l_t}}\}$, such that $\tau(e_{I_{T}}) \in F_{i_{T+1, 1}}$, and for $j = 1, \ldots, l_{T+1} - 1,$ there exists an edge $e_{i_{T+1,j}}$ from $F_{i_{T+1,j}}$ to $F_{i_{T+1, {j+1}}}$.
    Furthermore, there exists an edge $e_{i_{T+1,l_{T+1}}}$ from $F_{i_{T+1,l_{T}}}$ to $F_{j}$ with $j \in \{i_{t, 1}, \ldots, i_{t, l_t}\}$.
    We then extend the sequence~\eqref{eq:sequence} with $e_{i_{T+1,1}}, \ldots, e_{i_{T+1,l_{T+1}}}$ and obtain a sequence
    \begin{equation}\label{eq:sequenceext}
    e_{i_{1, 1}}, \ldots, e_{i_{1, l_1}}, e_{I_1}, e_{i_{2, 1}}, \ldots, e_{i_{2, l_2}}, e_{I_2}, \;\; \ldots \;\;, e_{i_{T, 1}}, \ldots, e_{i_{T, l_T}}, e_{I_T}, e_{i_{T+1,1}}, \ldots, e_{i_{T+1,l_{T+1}}}.
    \end{equation}
    
    We rename this sequence of edges as $e_1, \ldots, e_m$, and verify it satisfies the conditions (i)-(vi) listed at the beginning of this proof.  The conditions (i), (ii), (iii), (iv) directly follow from the construction of the sequence~\eqref{eq:sequenceext}. For condition (v) it suffices to note that the edge $e_{I_1}$ defined in the construction has target in $(\conv(\Gamma) \cap \Lambda) \setminus \bigcup_{i \in I_1} \left(\{b_i\} \times F_i \right)$.
    For (vi) it suffices to note that the last edge $e_{i_{T+1,l_{T+1}}}$ in the sequence has target in $F_{j}$ with $j \in \{i_{t, 1}, \ldots, i_{t, l_t}\}$: in other words the index $j$ has appeared earlier in the sequence.
    
    By renaming the lattices $\Lambda_1, \ldots, \Lambda_s,$ we can suppose without loss of generality that $\sigma(e_i) \in \Lambda_i$ for $i = 1, 2, \ldots, M$, and $\tau(e_m) \in \Lambda_m$. In other words, the indices $i_1, \ldots, i_M$ in conditions (i)-(vi) can respectively be assumed to be exactly $1, \ldots, m$.
    
    Notice that by applying a suitable linear transformation in $\SL(s+n, \Z)$ to $\Gamma$ and $\Z^{s+n}$, we can translate $\Lambda_i$ by any vector in $\{0\}^s \times \Z^n$, without moving $\Lambda_j, j \neq i$.
    This translates the set $C_i$ by the same vector, and therefore also translates $C$ by a vector in $\{0\}^s \times \Z^n$.
    Therefore, the effect of this transformation on $\Gamma^C$ is an \emph{affine} transformation (a linear transformation plus a translation) that stabilizes $\Lambda$, and does not change the properties we are interested in for $\Gamma^C$.
    Therefore we can without loss of generality suppose that $\sigma(e_i) = b_i \times 0^n, i = 1, \ldots, M$, and $b_i \times 0^n \in F_i$ for $i = M+1, \ldots, s$.
    In particular, $0^n \in C_i$ for all $1 \leq i \leq s$.

    For each $i = 2, \ldots, M$, let $z_i \in \Z^n$ be such that $\tau(e_{i-1}) = (0^s, z_i) + \sigma(e_i)$.
    Since $\sigma(e_i) = b_i \times 0^n$ and $\tau(e_{i-1}) \in \{b_i\} \times C_{i}$, we have $z_i \in C_i$.
    Recall that for an edge $e$ we denote by $e^-$ its inverse.
    Denote by $P$ the concatenation of edges
    \begin{multline*}
        e_1, e_2 + (0^s, z_2), e_3 + (0^s, z_2 + z_3), \ldots, e_M + (0^s, z_2 + z_3 + \cdots + z_M), \\
        e_{m-1}^- + (0^s, z_2 + \cdots + z_{m-2} + z_m + \cdots + z_M), \ldots, e_1^- + (0^s, z_m + \cdots + z_M).
    \end{multline*}
    We claim that $P$ is a path in $\Gamma^{C}$, and furthermore it starts in $\{b_1\} \times F$ and ends in $\{b_1\} \times ((C + C_1) \setminus F)$.

    Recall that $\sigma(e_1) \in \{b_1\} \times F_1$ by (i).
    Since $F_1 = F_1 + \{0^n\} + \cdots + \{0^n\} \subseteq F_1 + F_2 + \cdots + F_s + F_1 = F$, we indeed have $\sigma(e_1) \in \{b_1\} \times F$.
    Since $z_i \in C_i$ for each $i = 2, \ldots, m$ and $0^n \in C_1$, we have $z_2 + \cdots + z_{i} \in C$ for all $2 \leq i \leq M$ and $z_2 + \cdots + z_{i} + z_m + \cdots + z_M \in C$ for all $1 \leq i \leq m-1$.
    Therefore, each edge $e_i + (0^s, z_2 + \cdots + z_{i}) \in e_i + \{0^s\} \times C$ and each edge $e_i^{-} + (0^s, z_2 + \cdots + z_{i} + z_m + \cdots + z_M) \in e_i^{-} + \{0^s\} \times C,$ is in $\Gamma^{C} = \sum_{z \in C \cap \Z^n} \Gamma + (0^s, z)$.
    Consequently, every edge in $P$ belongs to $\Gamma^{C}$.
    We now show it ends in $\{b_i\} \times ((C + C_i) \setminus F)$.
    
    It suffices to show that $z_m + \cdots + z_M$ is not orthogonal to $v$.
    Indeed, each edge $e_i, i = 1, \ldots, M$ starts in $F_i$, meaning $v \cdot z_i = (0^s, v) \cdot (\sigma(e_i) - \tau(e_{i-1})) \geq 0$.
    Furthermore, condition (v) shows there exists $j \in \{1, \ldots, M\}$ such that $\tau(e_{j-1}) \in \{b_{i_j}\} \times \left(C_{i_j} \setminus F_{i_j} \right)$, meaning $v \cdot z_j$ is strictly positive.
    We therefore conclude that $P$ is a path in $\Gamma^{C}$ starting in $\{b_1\} \times F$ and ending in $\{b_1\} \times ((C + C_1) \setminus F)$.
\end{proof}

We now without loss of generality replace $\Gamma$ with $\Gamma^{C}$, this does not change the fact that $\Gamma$ is full-image and symmetric. Thus by Lemma~\ref{lem:acctoacc}, the new $\Gamma$ satisfies the following property: for each $i \in \{1, \ldots, s\}$, and each strict face $\{b_i\} \times F$ of the convex hull $\conv(\Gamma \cap \Lambda_i)$, there exists a path in $\Gamma$ starting in $\{b_i\} \times F$ and ending in $\conv(\Gamma \cap \Lambda_i) \setminus F$.
Furthermore, each $\conv(\Gamma \cap \Lambda_i)$ is of dimension $n$.

Now, the convex hull of $V(\Gamma) \cap \Lambda_i$ is a polytope of the form $\{b_i\} \times C_i$, where $C_i$ is a polytope in $\R^n$.
This time, each $C_i$ is of dimension $n$.
Consider the Minkowski sum $C \coloneqq C_1 + \cdots + C_s$.

For each $N \in \N$, define 
\[
\Gamma^{NC} \coloneqq \sum_{z \in NC \cap \Z^n} \Gamma + (0^s, z). 
\]
We will show that there exists $N$ such that $\Gamma^{NC}$ is connected.

Fix $i \in \{1, \ldots, s\}$, consider the (undirected) graph $\Gamma(i)$ over $\Z^n$, defined as follows.
The set of vertices of $\Gamma(i)$ is $\{v \in \Z^n \mid (b_i, v) \in V(\Gamma) \cap \Lambda_i\}$.
Two vertices $v, v'$ in $\Gamma(i)$ are connected by an edge if $v$ and $v'$ are connected by a path in $\Gamma$.
Then by Lemma~\ref{lem:acctoacc}, the graph $\Gamma(i)$ satisfies the following property: for each strict face $F$ of $\conv(\Gamma(i))$ there exists an edge with source in $F$ and target outside $F$. This is exactly the definition of face-accessibility as in~\cite[Section~3.1]{dong2024semigroup}.

Define
\[
\Gamma(i)^{(N-1)C_i} \coloneqq \bigcup_{z \in (N-1) C_i} \Gamma(i) + z,
\]
then $\conv(\Gamma(i)^{(N-1)C_i}) = (N-1)C_i + C_i = NC_i$.
For $c \in \R^n, R \in (0, 1), S \subseteq \R^n$, define 
\[
scale(S, c, R) \coloneqq \{c + R \cdot (x - c) \mid x \in S\}.
\]
Intuitively, $scale(S, c, R)$ is the scaling of the set $S$ with proportion $R$, and $c$ is the invariant point of the scaling.

\begin{lemma}[{\cite[Lemma~4.6]{dong2024semigroup}}]\label{lem:coverPv}
    Let $c_0 \in \Q^n$ be an interior point of $C_i$.
    There exists $N_i \in \N, R \in (0, 1)$, such that if $N > 0$ is divisible by $N_i$, then every vertex in the graph $\Gamma(i)^{(N-1)C_i}$ is connected by a path to some vertex in $scale(NC_i, N c_0, R)$.\footnote{Note that by taking $\Gamma$ in~\cite[Section~4]{dong2024semigroup} to be the graph $\Gamma(i)$ defined here and by taking $C$ to be $C_i$, the graph $\Gamma_N$ defined in~\cite[Section~4]{dong2024semigroup} corresponds to the graph $\Gamma(i)^{(N-1)C_i}$ defined above. This is because the set $S_N \coloneqq \{z \in \Z^n \mid z + C \subset NC\}$ defined in the beginning of~\cite[Section~4]{dong2024semigroup} corresponds exactly to $\Z^n \cap (N-1)C_i$. Indeed, using the notation of~\cite[Section~4]{dong2024semigroup}, we have $(N-1)C + C = NC$, so $(N-1)C \subseteq \{z \in \R^n \mid z + C \subset NC\}$. On the other hand, take $v \notin (N-1)C$, taking a linear transformation we can suppose $v = 0^n$, then $v+C = C \not\subset NC$ because the distance from $NC$ to $0^n$ is strictly larger than the distance from $C$ to $0^n$. Therefore $v \not\in \{z \in \R^n \mid z + C \subset NC\}$. We conclude that $\Z^n \cap (N-1)C = \Z^n \cap \{z \in \R^n \mid z + C \subset NC\} = S_N$.}
\end{lemma}

See Figure~\ref{fig:coverPv} for an illustration of Lemma~\ref{lem:coverPv}.

\begin{figure}[ht!]
    \centering
    \begin{minipage}[t]{.30\textwidth}
        \centering
        \includegraphics[width=\textwidth,height=1.0\textheight,keepaspectratio, trim={8cm 0cm 4.5cm 5cm},clip]{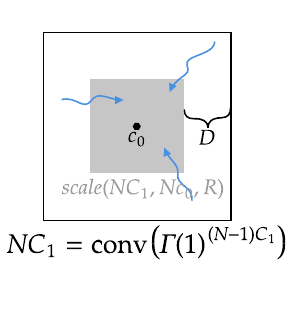}
        \caption{Illustration of Lemma~\ref{lem:coverPv} with $n = 2$, $i = 1$.}
        \label{fig:coverPv}
    \end{minipage}
    \hfill
    \begin{minipage}[t]{0.65\textwidth}
        \centering
        \includegraphics[width=\textwidth,height=1.0\textheight,keepaspectratio, trim={1.9cm 1cm 3cm 0.5cm},clip]{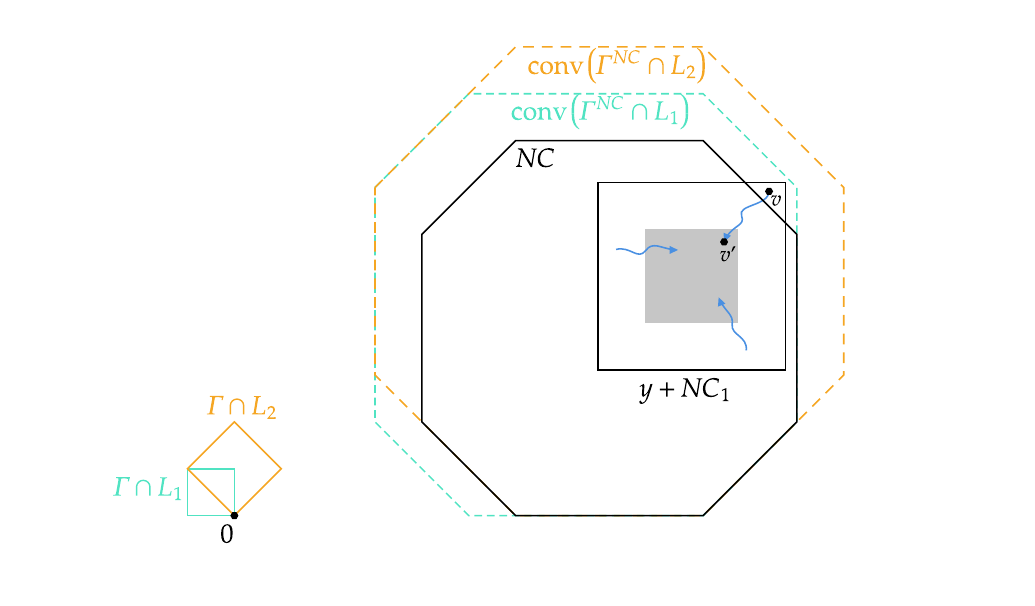}
        \caption{Illustration for Lemma~\ref{lem:getin}, projected on $\Z^n$. Here, $n = 2, s = 2, i = 1$.}
        \label{fig:getin}
    \end{minipage}
\end{figure}

Denote by $d_j \in \Z^n, j = 1, \ldots, n$ the natural basis of $\Z^n$, that is, $d_j$ is the vector with $1$ on the $j$-th coordinate and $0$ elsewhere.
Since $\mA$ is primitive and $\Gamma$ is full-image, for each $d_j, j = 1, \ldots, n,$ there exists a concatenation $Q_{1j}$ of $(\{0\}^s \times \Z^n)$-translations of edges in $\Gamma$ that goes from $(b_1, 0^n)$ to $(b_1, d_j)$.
(Recall that for each $e \in E(\Gamma)$ we have added $e^- \in E(\Gamma)$.)
For $i = 1, \ldots, s$, by additionally appending paths to and from $\Lambda_i$, we obtain a concatenation $Q_{ij}$ of $(\{0\}^s \times \Z^n)$-translations of edges in $\Gamma$ that goes from $(b_i, 0^n)$ to $(b_i, d_j)$.

More precisely, for every $i \in \{1, \ldots, s\}, j \in \{1, \ldots, n\}$, there exist edges $e_{1}, \ldots, e_{m}$ in $\Gamma$, satisfying the following properties.
\begin{enumerate}[nosep, label = (\roman*)]
    \item $\sigma(e_1), \tau(e_m) \in \Lambda_i$.
    \item For $k = 2, \ldots, m$, both $\tau(e_{k-1})$ and $\sigma(e_k)$ are in some same lattice $\Lambda_{l_k}$. In particular, both $\tau(e_{k-1})$ and $\sigma(e_k)$ belong to $\{b_{l_k}\} \times C_{l_k}$. Write $\tau(e_{k-1}) - \sigma(e_k) = (0^s, z_k)$ with $z_k \in \Z^n$.
    \item The concatenation of $e_1, e_2 + (0^s, z_2), e_3 + (0^s, z_2 + z_3), \ldots, e_m + (0^s, z_2 + \cdots + z_m)$ is a path from $\sigma(e_1)$ to $\sigma(e_1) + (0^s, d_j)$. We denote by $Q_{ij}$ this concatenation.
    \item We additionally define the value $D_{ij} \coloneqq \max\{\| z_2 \|, \| z_2 + z_3 \|, \ldots, \| z_2 + \cdots + z_m \|\} \in \R_{\geq 0}$.
\end{enumerate}

For two sets $S, S' \subseteq \R^n$, define their \emph{distance} as $\dist(S, S') \coloneqq \inf_{x \in S, x' \in S'}\|x - x'\|$.
For a set $S \subseteq \R^n$, define its \emph{diameter} as $\diam(S) \coloneqq \sup_{x, x' \in S}\|x - x'\|$.
Denote 
\[
D \coloneqq \max\{D_{ij} \mid 1 \leq i \leq s, 1 \leq j \leq n \} + \sqrt{n} + \max\{\diam(C_i) \mid 1 \leq i \leq s\}.
\]

\begin{lemma}\label{lem:getin}
    There exists an integer $N \in \N$ such that in the graph $\Gamma^{NC}$, for every $i \in \{1, \ldots, s\}$, every vertex $(b_i, v) \in V(\Gamma^{NC} \cap \Lambda_i)$ is connected to some other vertex $(b_i, v') \in V(\Gamma^{NC} \cap \Lambda_i)$ that is of distance at least $D$ from the boundary of $\conv(\Gamma^{NC} \cap \Lambda_i)$.
\end{lemma}
\begin{proof}
    See Figure~\ref{fig:getin} for an illustration of the proof.
    Let $N_i, i = 1, \ldots, s,$ be as defined in Lemma~\ref{lem:coverPv}.
    Let $N$ be a large enough multiple of $N_1 N_2 \cdots N_s$, such that the distance between $scale(NC_i, N c_0, R)$ and the boundary of $C_i$ is at least $D$.
    Let $(b_i, v)$ be a vertex of $\Gamma^{NC} \cap \Lambda_i$, it appears in some translated graph $\Gamma + (0^s, z)$ with $z \in NC$.
    Then, $(b_i, v - z) \in V(\Gamma)$.
    Since $z \in NC$, there exists some translation $y + (N-1)C_i$ of $(N-1)C_i$ such that $z \in y + (N-1)C_i \subseteq NC$.
    By Lemma~\ref{lem:coverPv}, $(b_i, v)$ is connected in $y + \Gamma(i)^{(N-1)C_i}$, and hence also in $\Gamma^{NC}$, to some vertex $(b_i, v')$ in $y + scale(NC_i, N c_0, R)$. Since $v'$ is of distance at least $D$ from the boundary of $(b_i, y) + NC_i$, it is also of distance at least $D$ from the boundary of $\conv(\Gamma^{NC} \cap \Lambda_i) = NC + (b_i, 0^n) + C_i \supseteq (b_i, y) + (N-1)C_i + C_i = (b_i, y) + NC_i$.
\end{proof}

\begin{lemma}[{Generalization of~\cite[Lemma~4.7]{dong2024semigroup}}]\label{lem:mix}
    Two vertices $(b_i, v_1), (b_i, v_2) \in V(\Gamma^{NC} \cap \Lambda_i)$ of distance at least $D$ from the boundary of $\conv(\Gamma^{NC} \cap \Lambda_i)$ are connected in $\Gamma^{NC}$.
\end{lemma}
\begin{proof}
    Upon applying a transformation in $\SL(s+n, \Z)$ to $\Gamma \subseteq \Z^{s+n}$, we now without loss of generality suppose $0^n \in C_i$ for all $i$.

    For two points $x, x' \in \R^n$, define the \emph{segment} $seg(x, x') \coloneqq \{rx + (1 - r)x' \mid r \in [0, 1]\} \subseteq \R^n$.
    Since $(b_i, v_1), (b_i, v_2) \in \Lambda_i,$ are of distance at least $D$ from the boundary of $\conv(\Gamma^{NC} \cap \Lambda_i) \subseteq \{b_i\} \times (NC + C_i)$, they are of distance at least $D - \diam(C_i) \geq \max_{i,j}\{D_{ij}\} + \sqrt{n}$ from the boundary of $\{b_i\} \times NC$.
    Hence, every point in the segment $seg(v_1, v_2)$ is of distance at least $\max_{i,j}\{D_{ij}\} + \sqrt{n}$ from the boundary of $NC$.

    There exists a path $P_{\Z^n}(v_1, v_2)$ in the lattice $\Lambda_i$ from $(b_i, v_1)$ to $(b_i, v_2)$, consisting of translations of the segments $\{b_i\} \times seg(0^n, d_k), k = 1, \ldots, n$, such that each point in $P_{\Z^n}(v_1, v_2)$ is of distance at most $\sqrt{n}$ from the segment $\{b_i\} \times seg(v_1, v_2)$.
    For $k = 1, \ldots, n$, replacing each segment $\{b_i\} \times (z + seg(0^n, d_k))$ in $P_{\Z^n}(v_1, v_2)$ by the translation $Q_{ik} + (0^n, z)$ of the path $Q_{ik}$, we obtain a path $P_{\Gamma}(v_1, v_2)$.
    We now show that each edge of $P_{\Gamma}(v_1, v_2)$ is in $E(\Gamma^{NC})$.

    By definition of the concatenation $Q_{ik}$, Each edge in $Q_{ik} + (0^n, z)$ appears in a translation $\Gamma + (0^n, z + z')$ of the graph $\Gamma$ satisfying $\|z'\| \leq D_{ik}$.
    Since $(b_i, z)$ is of distance at most $\sqrt{n}$ from the segment $\{b_i\} \times seg(v_1, v_2)$, and every point in the segment $\{b_i\} \times seg(v_1, v_2)$ is of distance at least $\max_{i, j}\{D_{ij}\} + \sqrt{n}$ from the boundary of $\{b_i\} \times NC$, we conclude that $(b_i, z)$ is of distance at least $\max_{i, j}\{D_{ij}\}$ from the boundary of $\{b_i\} \times NC$.
    Therefore, $(b_i, z + z')$ is within the boundary of $\{b_i\} \times NC$ because $\|z'\| \leq \max\{D_{ij}\}$.
    In other words, $z + z' \in NC$, the translation $\Gamma + (0^n, z + z')$ appears as a subgraph of $\Gamma^{NC}$.
    This shows that the path $P_{\Gamma}(v_1, v_2)$ connecting $v_1$ and $v_2$ is a subgraph of $\Gamma^{NC}$.
\end{proof}

\begin{proof}[Proof of Theorem~\ref{thm:acctocon}]
    See Figure~\ref{fig:connect} for an illustration of the proof.
    Let $N$ be the integer defined in Lemma~\ref{lem:getin}, we show that $\hG \coloneqq \Gamma^{NC} = \sum_{z \in NC \cap \Z^n} \Gamma + (0^s, z)$ is connected.
    First we show that for any $i \in \{1, \ldots, s\}$, every two vertices $(b_i, v), (b_i, w)$ in $\Gamma^{NC} \cap \Lambda_i$ are connected. Indeed, by Lemma~\ref{lem:getin}, $(b_i, v)$ and $(b_i, w)$ are respectively connected to vertices $(b_i, v'), (b_i, w')$, which are of distance at least $D$ from the boundary of $\conv(\Gamma^{NC} \cap \Lambda_i)$.
    By Lemma~\ref{lem:mix}, $(b_i, v')$ and $(b_i, w')$ are connected in $\Gamma^{NC}$. Therefore, $(b_i, v)$ and $(b_i, w)$ are connected in $\Gamma^{NC}$.

    This shows that for any $i \in \{1, \ldots, s\}$, all vertices in $\Gamma^{NC} \cap \Lambda_i$ lie in the same connected component of $\Gamma^{NC}$.
    Since the $\mA$-graph $\Gamma^{NC}$ is full-image and the automaton $\mA$ is trim, the different lattices $\Gamma^{NC} \cap \Lambda_i$ connected to each other by edges in $\Gamma^{NC}$ form a single connected component.
    Therefore $\Gamma^{NC}$ is connected. Since $\Gamma^{NC}$ is also symmetric, it is Eulerian.
\end{proof}

\section{Proof of Theorem~\ref{thm:dec}}\label{app:proofdec}
In this section of the appendix we prove Theorem~\ref{thm:dec}:

\thmdec*

Our proof strictly follows the proof of~\cite[Theorem~3.9]{dong2024semigroup} provided in~\cite[Sections~5-6]{dong2024semigroup}.
We recall~\cite[Theorem~3.9]{dong2024semigroup} for comparison:

\begin{theorem}[{\cite[Theorem~3.9]{dong2024semigroup}}]\label{thm:olddec}
Denote $\A \coloneqq \R[\oX^{\pm}], \A^+ \coloneqq \Rp[\oX^{\pm}]^*$. Fix $n \in \N$.
Suppose we are given as input a set of generators $\bg_1, \ldots, \bg_m \in \A^K$ with integer coefficients, as well as the vectors $\ta_1, \ldots, \ta_K \in \Z^n$ and two subsets $I, J$ of $\{1, \ldots, K\}$.
Denote by $\mM$ be the $\A$-submodule of $\A^K$ generated by $\bg_1, \ldots, \bg_m$.
It is decidable whether there exists $\bff \in \mM \cap \ApK$ satisfying
\begin{equation}\label{eq:olddeccond}
        \left(O_{v} \cup J\right) \cap M_{v}(I, \bff) \neq \emptyset, \quad \text{ for every } v \in \Rns.
\end{equation}
Here, if $n = 0$ then $\A$ is understood as $\R$, and Property~\eqref{eq:olddeccond} is considered trivially true.
\end{theorem}

We need to add the quantifier ``for all $\xi \in \Xi$'' in appropriate places of the proof and modify a few definitions and lemmas accordingly.
\subsection{Local-global principle}
Given $f \in \A$ and $v \in \Rns$, the \emph{initial polynomial} of $f$ is defined as the sum of all monomials in $f$ having the maximal degree $\deg_v(\cdot)$:
\[
\init_v(f) \coloneqq \sum\nolimits_{\deg_v(\oX^b) = \deg_v(f)} c_b \oX^b, \quad \text{ where } f = \sum c_b \oX^b.
\]
For $\bff = (f_1, \ldots, f_K) \in \A^K$, we naturally denote $\init_v(\bff) \coloneqq (\init_v(f_1), \ldots, \init_v(f_K)) \in \A^K$.

In this subsection we prove the following local-global principle:

\begin{restatable}[{Generalization of~\cite[Theorem~3.8]{dong2024semigroup}}]{theorem}{thmlocglob}\label{thm:locglob}
Let $\mM$ be an $\A$-submodule of $\A^K$ and $I_{\xi}, J_{\xi}, \xi \in \Xi$ be subsets of $\{1, \ldots, K\}$.
There exists $\bff \in \mM \cap \ApK$ satisfying 
\begin{equation}\label{eq:globv}
        \left(O_{v} \cup J_{\xi}\right) \cap M_{v}(I_{\xi}, \bff) \neq \emptyset, \quad \text{ for every } v \in \Rns, \xi \in \Xi
\end{equation}
if and only if the two following conditions are satisfied:
\begin{enumerate}[noitemsep, label = \arabic*.]
    \item \label{item:locr} \emph{\textbf{(LocR):}} For every $r \in \Rpp^n$, there exists $\bff_{r} \in \mM$ such that $\bff_{r}(r) \in \Rpp^K$.
    \item \label{item:locinf} \emph{\textbf{(LocInf):}} For every $v \in \Rns$, there exists $\bff_{v} \in \mM$,
    \begin{enumerate}[noitemsep]
        \item $\init_{v}\left(\bff_{v}\right) \in \ApK$.
        \item Denote $I'_{\xi} \coloneqq M_{v}(I_{\xi}, \bff_{v}), J'_{\xi} \coloneqq O_{v} \cup J_{\xi}$.
        We have 
        \begin{equation}\label{eq:locv}
        \left(O_{w} \cup J'_{\xi}\right) \cap M_{w}(I'_{\xi}, \init_{v}(\bff_{v})) \neq \emptyset \quad \text{ for every $w \in \Rns, \xi \in \Xi$}.
        \end{equation}
    \end{enumerate}
\end{enumerate}
\end{restatable}


Define the quotient 
$
D_n \coloneqq \Rns / \Rpp
$.
That is, elements of $D_n$ are of the form $v \Rpp, v \in \Rns$, where $v \Rpp = v' \Rpp$ if and only if $v = r \cdot v'$ for some $r \in \Rpp$.
The quotient $D_n$ can be identified with the unit sphere of dimension $n$ since every $v \Rpp$ is equal to exactly one $v' \Rpp$ with $||v'|| = 1$.
We equip $D_n$ with the standard topology of the unit sphere.
Note that $\init_v(\cdot), M_v(\cdot)$ and $O_v$ are invariant when scaling $v$ by any positive real number.

\begin{restatable}[{\cite[Theorem~5.1]{dong2024semigroup}}]{lemma}{lemtweak}\label{lem:tweakinit}
    Fix $v \in \Rns$, a set $I_{\xi} \subseteq \{1, \ldots, K\}$ and $\bff \in \A^K$. 
    There exists an open neighbourhood $U \subseteq D_n$ of $v \Rpp$, such that for every $w \in \Rns$ with $(v + w)\Rpp \in U$, we have
    \begin{equation}\label{eq:vw}
    \init_{v + w}\left(\bff\right) = \init_{w}\left(\init_{v}(\bff)\right), \quad M_{v + w}(I_{\xi}, \bff) = M_{w}(M_{v}(I_{\xi}, \bff), \init_{v}(\bff)), \quad \text{and} \quad O_{v + w} = O_{v} \cup O_{w}.
    \end{equation}
\end{restatable}

With Lemma~\ref{lem:tweakinit} we can prove the ``only if'' part of Theorem~\ref{thm:locglob}:

\begin{proof}[Proof of ``only if'' part of Theorem~\ref{thm:locglob}]
    Suppose $\bff \in \mM \cap \ApK$ satisfies Property~\eqref{eq:globv}.
    To show \hyperref[item:locr]{(LocR)}, simply take $\bff_{r} \coloneqq \bff$ for all $r \in \Rpp^n$, then $\bff(r) \in \Rpp^K$.
    As for \hyperref[item:locinf]{(LocInf)}, for every $v \in \Rns$ we show that $\bff_{v} \coloneqq \bff$ satisfies Properties \hyperref[item:locinf]{(LocInf)}(a) and (b). 
    Property \hyperref[item:locinf]{(LocInf)}(a) is satisfied by the definition of $\bff$.
    We now show Property \hyperref[item:locinf]{(LocInf)}(b). Take any $w \in \Rns$ and $\xi \in \Xi$.
    
    When $w \in v \Rpp$, we have $O_{w} \cup J'_{\xi} = O_{w} \cup O_{v} \cup J_{\xi} = O_{v} \cup J_{\xi}$ and 
    \[
    M_{w}(I'_{\xi}, \init_{v}(\bff)) = M_{w}(M_{v}(I_{\xi}, \bff), \init_{v}(\bff)) = M_{v}(M_{v}(I_{\xi}, \bff), \init_{v}(\bff)) = M_{v}(I_{\xi}, \bff),
    \]
    so Property~\hyperref[item:locinf]{(LocInf)}(b) is equivalent to $\left(O_{v} \cup J_{\xi}\right) \cap M_{v}(I_{\xi}, \bff) \neq \emptyset$.
    This is exactly the Property~\eqref{eq:globv} satisfied by $\bff$.
    
    When $w \not\in v \Rpp$, let $U \subseteq D_n$ be the open neighbourhood of $v \Rpp$ defined in Lemma~\ref{lem:tweakinit}.
    Scaling $w$ by a small enough positive real we can suppose $(v + w)\Rpp \in U$.
    We have
    \[
    \init_{v + w}\left(\bff\right) = \init_{w}\left(\init_{v}(\bff)\right), \quad M_{v + w}(I_{\xi}, \bff) = M_{w}(I'_{\xi}, \init_{v}(\bff)), \quad \text{and} \quad O_{v + w} = O_{v} \cup O_{w},
    \]
    where $I'_{\xi} = M_{v}(I_{\xi}, \bff)$.
    Therefore
    $
    \left(O_{v + w} \cup J_{\xi}\right) \cap M_{v + w}(I_{\xi}, \bff) = \left(O_{w} \cup O_{v} \cup J_{\xi}\right) \cap M_{w}(I'_{\xi}, \init_{v}(\bff))
     = \left(O_{w} \cup J'_{\xi}\right) \cap M_{w}(I'_{\xi}, \init_{v}(\bff)).
    $
    Since $\bff$ satisfies Property~\eqref{eq:globv}, we have $\left(O_{v + w} \cup J_{\xi} \right) \cap M_{v + w}(I_{\xi}, \bff) \neq \emptyset$.
    Therefore we also have $\left(O_{w} \cup J'_{\xi}\right) \cap M_{w}(I'_{\xi}, \init_{v}(\bff)) \neq \emptyset$ for all $w \not\in v \Rpp$.
\end{proof}

We now start working towards proving the ``if'' part of Theorem~\ref{thm:locglob}.
The main idea is a ``gluing'' procedure.
The following lemma is the foundation of this gluing argument.
It shows the ``continuity'' of the property~\hyperref[item:locinf]{(LocInf)} when changing the direction $v$ by a small amount.

\begin{lemma}[{Generalization of~\cite[Lemma~5.2]{dong2024semigroup}}]\label{lem:Uv}
    Suppose $v \in \Rns$ and $\bff_{v} \in \mM$ satisfy properties \hyperref[item:locinf]{(LocInf)}(a) and (b) of Theorem~\ref{thm:locglob}.
    Then there exists an open neighbourhood $U_{v} \subseteq D_n$ of $v \Rpp$ such that for every $v'\Rpp \in U_v, \xi \in \Xi$,
    \begin{enumerate}[noitemsep, label = (\roman*)]
        \item $\init_{v'}\left(\bff_{v}\right) \in \ApK$.
        \item $(O_{v'} \cup J_{\xi}) \cap M_{v'}(I_{\xi}, \bff_{v}) \neq \emptyset$.
    \end{enumerate}
\end{lemma}
\begin{proof}
    We use Lemma~\ref{lem:tweakinit} on $v, I_{\xi}$ and $\bff_{v}$ to obtain an open neighbourhood $U_{v, \xi} \subseteq D_n$ of $v \Rpp$,
    where for all $(v + w)\Rpp \in U_{v, \xi}$ we have
    \[
    \init_{v + w}\left(\bff_{v}\right) = \init_{w}\left(\init_{v}(\bff_{v})\right), \; M_{v + w}(I_{\xi}, \bff_{v}) = M_{w}(M_{v}(I_{\xi}, \bff_{v}), \init_{v}(\bff_{v})), \; \text{and} \; O_{v + w} = O_{v} \cup O_{w}.
    \] 
    Note that $\init_{v}(\bff_{v}) \in \ApK$ by Property~\hyperref[item:locinf]{(LocInf)}(a) of $\bff_{v}$.
    Take $U_{v} \coloneqq \bigcap_{\xi \in \Xi} U_{v, \xi}$.
    Since taking the initial polynomial of any polynomial in $\A^+$ yields an element of $\A^+$, we have $\init_{v + w}\left(\bff_{v}\right) = \init_{w}\left(\init_{v}(\bff_{v})\right) \in \ApK$.
     Furthermore, $(O_{v + w} \cup J_{\xi}) \cap M_{v + w}(I_{\xi}, \bff_{v}) = \left(O_{w} \cup J'_{\xi}\right) \cap M_{w}(I'_{\xi}, \init_{v}(\bff_{v}))$, which is non-empty by Property~\hyperref[item:locinf]{(LocInf)}(b) of $\bff_{v}$.
    Therefore, both (i) and (ii) are satisfied for $v'\Rpp \in U_{v}$.
\end{proof}

The following lemma shows that one can ``glue'' all different $\bff_v, v \in \Rns$ together to obtain a single $\bff$ that has positive initial polynomial at every direction $v \in \Rns$.
\begin{lemma}[{Generalization of~\cite[Lemma~5.3]{dong2024semigroup}}]\label{lem:finf}
    Suppose Condition \hyperref[item:locinf]{(LocInf)} of Theorem~\ref{thm:locglob} is satisfied.
    Then there exists $\bff \in \mM$ that satisfies 
    \begin{enumerate}[noitemsep, label = (\roman*)]
        \item $\init_{v}\left(\bff\right) \in \ApK$ for all $v \in \Rns$.
        \item $(O_{v} \cup J_{\xi}) \cap M_{v}(I_{\xi}, \bff) \neq \emptyset$ for all $v \in \Rns, \xi \in \Xi$.
    \end{enumerate}
\end{lemma}
\begin{proof}
    The exact same proof as~\cite[Lemma~5.3]{dong2024semigroup} works, we only need to replace the usage of~\cite[Lemma~5.2]{dong2024semigroup} by the usage of Lemma~\ref{lem:Uv}.
\end{proof}

Denote by $\bff_{\infty}$ the element $\bff \in \mM$ obtained in Lemma~\ref{lem:finf}.
Since $\init_{v}\left(\bff_{\infty}\right) \in \ApK$ for all $v \in \Rns$, there exists $c > 1$ such that $\bff_{\infty}(x) \in \Rpp^K$ for all $x \in \Rpp^n \setminus [1/c, c]^n$.
Define the compact set 
\[
C \coloneqq [1/(4nc), 4nc]^n \supseteq [1/c, c]^n.
\]

\begin{lemma}[{\cite[Lemma~5.4]{dong2024semigroup}}]\label{lem:fC}
    Let $\mM$ be an $\A$-submodule of $\A^K$ and $C \subset \Rpp^n$ be a compact set.
    Suppose for all $r \in C$ there exists $\bff_{r} \in \mM$ with $\bff_{r}(r) \in \Rpp^K$.
    Then there exists $\bff \in \mM$ such that 
    $
        \bff(x) \in \Rpp^{n}
    $
    for all $x \in C$.
\end{lemma}

Denote by $\bff_{C}$ the element $\bff \in \mM$ obtained in Lemma~\ref{lem:fC}.

\begin{lemma}[{\cite[Corollary~5.6]{dong2024semigroup}}]\label{cor:Handelman}
    Let $\bff \in \A^{K}$.
    There exists $g \in \A^+$ such that $g \bff \in \left(\A^+\right)^{K}$ if and only if the two following conditions are satisfied:
    \begin{enumerate}[noitemsep, label = (\roman*)]
        \item For all $r \in \Rpp^n$, we have $\bff(r) \in \Rpp^{K}$.
        \item For all $v \in \Rns$ and $r \in \Rpp^n$, we have $\init_{v}(\bff)(r) \in \Rpp^{K}$.
    \end{enumerate}
\end{lemma}

We are now ready to prove the ``if'' part of Theorem~\ref{thm:locglob} by ``gluing'' together the elements $\bff_{\infty}, \bff_C \in \mM$ obtained respectively in Lemma~\ref{lem:finf} and \ref{lem:fC}.
\begin{proof}[Proof of ``if'' part of Theorem~\ref{thm:locglob}]
    Let $\bff_{\infty}, \bff_C \in \mM$ be the elements obtained respectively in Lemma~\ref{lem:finf} and \ref{lem:fC}.
    Define the  polynomial
    \[
    q \coloneqq \frac{1}{2nc}\sum_{i = 1}^n(X_i + X_i^{-1}) \in \A^+.
    \]
    It is easy to see that we have $\deg_w(q) > 0$ for all $w \in \Rns$.
    By the compactness of the unit sphere, the value $\min_{||w|| = 1} \deg_w(q)$ exists and is a positive number.
    
    Let $\epsilon > 0$ be such that 
    \begin{equation}\label{eq:ineqepapp}
        \epsilon \cdot \bff_{\infty}(x) + \bff_{C}(x) \in \Rpp^{n}
    \end{equation}
    for all $x \in C$.
    Such an $\epsilon$ exists by the compactness of $C$.
    We claim that there exists $N \in \N$ such that the vector $\bff \coloneqq \epsilon q^N \cdot \bff_{\infty} + \bff_{C}$ satisfies Conditions~(i) and (ii) in Corollary~\ref{cor:Handelman} simultaneously.

    Let $M \in \N$ be such that $\deg_{v}(f_{\infty, i}) + M \cdot \min_{||w|| = 1} \deg_w(q) > \deg_{v}(f_{C, i})$ for all $v \in \Rns, ||v|| = 1$ and $i = 1, \ldots, K$.
    Such an $M$ exists by the compactness of the unit sphere and because $\min_{||w|| = 1} \deg_w(q) > 0$.
    Let $\bg \coloneqq \epsilon q^M \cdot \bff_{\infty} + \bff_{C}$.
    Then for all $v \in \Rns, i = 1, \ldots, K,$ we have $\deg_v(\epsilon q^M \cdot f_{\infty, i}) = M \cdot \deg_v(q) + \deg_v(f_{\infty, i}) > \deg_v(f_{C, i})$.
    Therefore $\init_{v}(\bg) = \init_{v}(\epsilon q^M \cdot \bff_{\infty}) \in \ApK$ for all $v \in \Rns$.
    Therefore, there exists another compact set $[1/d, d]^n \supset C$ such that $\bg(x) \in \Rpp^{K}$ for all $x \in \Rpp^n \setminus [1/d, d]^n$.
    Since $[1/d, d]^n \supset C = [1/(4nc), 4nc]^n$, we have $d \geq 4nc$.
    Since the set $[1/d, d]^n \setminus (1/(4nc), 4nc)^n$ is compact and $\bff_{\infty}(x) \in \Rpp^K$ for all $x \in [1/d, d]^n \setminus (1/(4nc), 4nc)^n$, there exists $N > M$ such that
    \begin{equation}\label{eq:ineq2Napp}
    \epsilon f_{\infty, i}(x) \cdot 2^{N} + f_{C, i}(x) > 0
    \end{equation}
    for all $x \in [1/d, d]^n \setminus (1/(4nc), 4nc)^n$ and all $i = 1, \ldots, K$.
    We prove that for this $N$, the element $\bff \coloneqq \epsilon q^N \cdot \bff_{\infty} + \bff_{C}$ satisfies Conditions~(i) and (ii) in Corollary~\ref{cor:Handelman}  simultaneously.

    Fix any $i \in \{1, \ldots, K\}$.
    For every $x \in \Rpp^n \setminus [1/d, d]^n$, we have $q(x) > \frac{d}{2nc} \geq 1$ and $f_{\infty, i}(x) > 0$, so
    \[
    f_{i}(x) = \epsilon q(x)^N \cdot f_{\infty, i}(x) + f_{C, i}(x) \geq \epsilon q(x)^M \cdot f_{\infty, i}(x) + f_{C, i}(x) = g_{i}(x) > 0.
    \]
    
    For every $x \in [1/d, d]^n \setminus C = [1/d, d]^n \setminus [1/(4nc), 4nc]^n$, we have $x_{i'} \geq 4nc$ for at least one $i' \in \{1, \ldots, K\}$.
    Since $f_{\infty, i}(x) > 0$ by the definition of $C$, we have
    \[
    f_{i}(x) = \epsilon f_{\infty, i}(x) \cdot \left(\sum_{j = 1}^n\frac{x_{j} + x_{j}^{-1}}{2nc}\right)^N + f_{C, i}(x) \geq \epsilon f_{\infty, i}(x) \cdot 2^N + f_{C, i}(x) > 0
    \]
    by $\sum_{j = 1}^n(x_j + x_j^{-1}) > x_{i'} \geq 4nc$ and Inequality~\eqref{eq:ineq2Napp}.
    
    For every $x \in C \setminus [1/c, c]^n$,
    \[
    f_{i}(x) = \epsilon q(x)^N \cdot f_{\infty, i}(x) + f_{C, i}(x) > 0
    \]
    since $f_{\infty, i}(x) > 0$ for all $x \not\in [1/c, c]^n$ and $f_{C, i}(x) > 0$ for all $x \in C$.
    
    For every $x \in [1/c, c]^n$,
    \[
    f_{i}(x) = \epsilon f_{\infty, i}(x) \cdot \left(\sum_{j = 1}^n\frac{x_j + x_j^{-1}}{2nc}\right)^N + f_{C, i}(x) \geq \min\{\epsilon f_{\infty, i}(x), 0\} + f_{C, i}(x) > 0.
    \]
    The last inequality is due to $f_{C, i}(x) > 0$ and Inequality~\eqref{eq:ineqepapp}.
    The second to last inequality can be justified as follows.
    If $f_{\infty, i}(x) \geq 0$ then $f_{\infty, i}(x) \cdot \left(\sum_{i = 1}^n\frac{x_i + x_i^{-1}}{2nc}\right)^N \geq 0$, otherwise $\left(\sum_{i = 1}^n\frac{x_i + x_i^{-1}}{2nc}\right)^N \leq \left(\sum_{i = 1}^n\frac{2c}{2nc}\right)^N = 1$ so $f_{\infty, i}(x) \cdot \left(\sum_{i = 1}^n\frac{x_i + x_i^{-1}}{2nc}\right)^N \geq f_{\infty, i}(x)$.
    
    Therefore, for every $x \in \Rpp^n$, we have $f_i(x) > 0$.
    In other words, $\bff$ satisfies Conditions~(i) in Corollary~\ref{cor:Handelman}.
    Furthermore, since $N > M$ we have $\deg_{v}(q^N \cdot f_{\infty, i}) > \deg_{v}(f_{C, i})$ for $i = 1, \ldots, K, v \in \Rns$.
    Hence $\init_{v}(\bff) = \init_{v}(\epsilon q^N \cdot \bff_{\infty}) \in \ApK$ and $M_{v}(I_{\xi}, \bff) = M_{v}(I_{\xi}, \bff_{\infty})$ for all $v \in \Rns, \xi \in \Xi$.
    Therefore, $\bff$ satisfies Conditions~(ii) in Corollary~\ref{cor:Handelman}.
    
    Therefore, by Lemma~\ref{cor:Handelman}, we can find $g \in \A^+$ such that $g \bff \in \ApK$.
    We have at the same time $g \bff \in \mM$ as well as $\left(O_{v} \cup J_{\xi}\right) \cap M_{v}(I_{\xi}, g \bff) = \left(O_{v} \cup J_{\xi}\right) \cap M_{v}(I_{\xi}, \bff) = \left(O_{v} \cup J_{\xi}\right) \cap M_{v}(I_{\xi}, \bff_{\infty}) \neq \emptyset$ for all $v \in \Rns, \xi \in \Xi$.
    We have thus found the required element $g \bff \in \mM \cap \ApK$ satisfying Property~\eqref{eq:globv}.
\end{proof}

\subsection{Decidability of local conditions}\label{subapp:dec}
This subsection is dedicated to the proof of Theorem~\ref{thm:dec}.
By the local-global principle (Theorem~\ref{thm:locglob}), this amounts to showing decidability of the two ``local'' Conditions~\hyperref[item:locr]{(LocR)} and \hyperref[item:locinf]{(LocInf)}.

Decidability of the Condition~\hyperref[item:locr]{(LocR)} is the same as in~\cite{dong2024semigroup}, which uses Tarski's theorem~\cite{Tarski1949}.

\begin{proposition}[{\cite[Proposition~6.2]{dong2024semigroup}}]\label{prop:decr}
    Given the generators $\bg_1, \ldots, \bg_m$ for $\mM$, it is decidable whether Condition~\hyperref[item:locr]{(LocR)} of Theorem~\ref{thm:locglob} is satisfied.
\end{proposition}

We now focus on deciding the Condition~\hyperref[item:locinf]{(LocInf)}.

\paragraph*{From \hyperref[item:locinf]{(LocInf)} to shifted initials \hyperref[item:locshift]{(LocInfShift)}.}
First, we introduce the \emph{shifted initials}, in order to replace Condition~\hyperref[item:locinf]{(LocInf)} of Theorem~\ref{thm:locglob} with a new Condition~\hyperref[item:locshift]{(LocInfShift)}.

Suppose we are given $\bff \in \A^K$, $v \in \Rns$ and $\balpha = (\alpha_1, \ldots, \alpha_{K}) \in \R^{K}$.
Then the \emph{shifted initials}
$\init_{v, \alpha}(\bff) = (\init_{v, \alpha}(\bff)_1, \ldots, \init_{v, \alpha}(\bff)_K)$ is defined as
\begin{align*}
\init_{v, \alpha}(\bff)_i \coloneqq 
\begin{cases}
\init_{v}(f_i) & \text{ if } \deg_v(f_i) + \alpha_i = \max_{1 \leq i' \leq K} \{\deg_v(f_{i'}) + \alpha_{i'}\},\\
0 & \text{ if } \deg_v(f_i) + \alpha_i < \max_{1 \leq i' \leq K} \{\deg_v(f_{i'}) + \alpha_{i'}\}. \\
\end{cases}
\end{align*}

\begin{lemma}[{\cite[Lemma~6.3]{dong2024semigroup}}]\label{lem:initpos}
    Let $\bff \in \A^K$ and $v \in \Rns$.
    We have $\init_v(\bff) \in \ApK$ if and only if there exists $\balpha \in \R^{K}$ such that $\init_{v, \balpha}(\bff) \in \ApK$.
    Furthermore, in this case, $\init_v(\bff) = \init_{v, \balpha}(\bff)$ and $\alpha_1 + \deg_v(f_1) = \cdots = \alpha_K + \deg_v(f_K)$.
\end{lemma}

Given $v = (v_1, \ldots, v_n) \in \Rns$, define the following set of real numbers:
\[
\sum_{k = 1}^n \Z v_k \coloneqq \left\{\sum_{k = 1}^n z_k v_k \;\middle|\; z_1, \ldots, z_n \in \Z \right\}.
\]
Then for every $f \in \A$, we have $\deg_{v}(f) \in \sum_{k = 1}^n \Z v_k$.

\begin{proposition}[{Generalization fof~\cite[Proposition~6.4]{dong2024semigroup}}]\label{prop:inftoshift}
Condition~\emph{\hyperref[item:locinf]{(LocInf)}} of Theorem~\ref{thm:locglob} is equivalent to the following:

\begin{enumerate}[noitemsep]
    \item[2.] \label{item:locshift} \emph{\textbf{(LocInfShift):}}  For every $v \in \Rns$, there exists $\bff \in \mM$ as well as $\balpha \in \left(\sum_{k = 1}^n \Z v_k\right)^K$ satisfying the following properties: 
    \begin{enumerate}[noitemsep]
        \item $\init_{v, \balpha}\left(\bff\right) \in \ApK$.
        \item For every $\xi \in \Xi$, denote $I'_{\xi} \coloneqq \{i \in I_{\xi} \mid \alpha_{i} = \min_{i' \in I_{\xi}} \alpha_{i'}\}, J'_{\xi} \coloneqq O_{v} \cup J_{\xi}$.
        We have 
        \[
        \left(O_{w} \cup J'_{\xi}\right) \cap M_{w}(I'_{\xi}, \init_{v, \balpha}(\bff)) \neq \emptyset \quad \text{ for every $w \in \Rns, \xi \in \Xi$}.
        \]
    \end{enumerate}
\end{enumerate}
\end{proposition}
\begin{proof}
    \textbf{\hyperref[item:locinf]{(LocInf)}}$\implies$\textbf{\hyperref[item:locshift]{(LocInfShift)}}.
    Suppose Condition~\hyperref[item:locinf]{(LocInf)} of Theorem~\ref{thm:locglob} is true.
    Fix a vector $v \in \Rns$.
    Then there exists $\bff \in \mM$, such that $\init_{v}\left(\bff\right) \in \ApK$ satisfies Property~\hyperref[item:locinf]{(LocInf)}(b).
    As in Lemma~\ref{lem:initpos}, we can let $\alpha_i \coloneqq - \deg_v(f_{i})$ for $i = 1, \ldots, K$, then $\init_{v, \balpha}\left(\bff\right) = \init_{v}\left(\bff\right) \in \ApK$, satisfying \hyperref[item:locshift]{(LocInfShift)}(a).
    Furthermore, we have $\balpha \in \left(\sum_{k = 1}^n \Z v_k\right)^K$ by the definition of $\alpha_i = - \deg_v(f_{i})$.
    Finally, for every $\xi \in \Xi$, we have $\{i \in I_{\xi} \mid \alpha_{i} = \min_{i' \in I_{\xi}} \alpha_{i'}\} = \{i \in I_{\xi} \mid \deg_v(f_{i}) = \max_{i' \in I_{\xi}} \deg_v(f_{i'})\} = M_v(I_{\xi}, \bff)$, so \hyperref[item:locinf]{(LocInf)}(b) implies \hyperref[item:locshift]{(LocInfShift)}(b).

    \textbf{\hyperref[item:locshift]{(LocInfShift)}}$\implies$\textbf{\hyperref[item:locinf]{(LocInf)}}.
    Suppose Condition~\hyperref[item:locshift]{(LocInfShift)} is true.
    Fix a vector $v \in \Rns$.
    Then there exists $\bff \in \mM$ as well as $\balpha \in \Rns$, such that $\init_{v, \balpha}\left(\bff\right) \in \ApK$ satisfies Property~\hyperref[item:locshift]{(LocInfShift)}(b).
    By Lemma~\ref{lem:initpos}, we have $\init_{v}\left(\bff\right) = \init_{v, \balpha}\left(\bff\right) \in \ApK$, and $\alpha_1 + \deg_v(f_{1}) = \cdots = \alpha_K + \deg_v(f_{K})$.
    Therefore for every $\xi \in \Xi$, we have $\{i \in I_{\xi} \mid \alpha_{i} = \min_{i' \in I_{\xi}} \alpha_{i'}\} = \{i \in I_{\xi} \mid \deg_v(f_{i}) = \min_{i' \in I_{\xi}} \deg_v(f_{i'})\} = M_v(I_{\xi}, \bff)$, so \hyperref[item:locshift]{(LocInfShift)}(b) implies \hyperref[item:locinf]{(LocInf)}(b).
\end{proof}

\paragraph*{Dimension reduction: a special case.}
We will further reduce Condition~\hyperref[item:locshift]{(LocInfShift)} to a Condition~\hyperref[item:locd]{(LocInfD)} (which will be defined in Proposition~\ref{prop:shifttod}).
We first consider the special case where the vector $v \in \Rns$ in Condition~\hyperref[item:locshift]{(LocInfShift)} is of the form $(0, \ldots, 0, v_{d+1}, \ldots, v_n)$, where $v_{d+1}, \ldots, v_n \in \R$ are $\Q$-linearly independent.

As in~\cite[Definition~6.5]{dong2024semigroup}, we now define the \emph{super Gr\"{o}bner basis} of an $\A$-module $\mM$.
Let $v \in \Rns, \balpha \in \R^K$.
Define $\init_{v, \alpha}(\mM)$ to be the $\A$-module generated by the elements $\init_{v, \alpha}(\bff), \bff \in \mM$:
\[
\init_{v, \alpha}(\mM) \coloneqq \sum_{\bff \in \mM} \A \cdot \init_{v, \alpha}(\bff) = \left\{\sum_{j = 1}^q p_j \cdot \init_{v, \alpha}(\bff_j) \;\middle|\; q \in \N, p_1, \ldots p_q \in \A, \bff_1, \ldots, \bff_q \in \mM\right\}.
\]

\begin{definition}[{Super Gr\"{o}bner basis~\cite[Definition~6.5]{dong2024semigroup}}]\label{def:grob}
    A set of generators $\bg_1, \ldots, \bg_m$ for $\mM$ is called a \emph{super Gr\"{o}bner basis} if for all $v \in \Rns, \balpha \in \R^K$, the set $\{\init_{v, \alpha}(\bg_1), \ldots, \init_{v, \alpha}(\bg_m)\}$ generates $\init_{v, \alpha}(\mM)$ as an $\A$-module.
\end{definition}

\begin{lemma}[{\cite[Lemma~6.6]{dong2024semigroup}}]\label{lem:grob}
    Suppose we are given an arbitrary set of generators for a module $\mM$. Then a super Gr\"{o}bner basis of $\mM$ is effectively computable.
\end{lemma}

Furthermore, if the given generators for $\mM$ contain only polynomials with integer coefficients, then Lemma~\ref{lem:grob} computes a super Gr\"{o}bner basis containing only polynomials with integer coefficients.

Let $0 \leq d \leq n-1$ be an integer.
From now on we denote 
\[
\A_d \coloneqq \R[X_1^{\pm}, \ldots, X_d^{\pm}], \quad \A_d^+ \coloneqq \Rp[X_1^{\pm}, \ldots, X_d^{\pm}]^*.
\]
In particular, $\A_0 = \R, \A_0^+ = \Rpp$.

We now consider the vectors $v \in \Rns$ with the special form $(0, \ldots, 0, v_{d+1}, \ldots, v_n)$ where $v_{d+1}, \ldots, v_n$ are $\Q$-linearly independent.

\begin{lemma}[{\cite[Lemma~6.7]{dong2024semigroup}}]\label{lem:ing}
    Let $\bg_1, \ldots, \bg_m$ be a super Gr\"{o}bner basis of $\mM$.
    
    Let $v = (0, \ldots, 0, v_{d+1}, \ldots, v_n) \in \Rns$ be such that $0 \leq d \leq n-1$ and $v_{d+1}, \ldots, v_n$ are $\Q$-linearly independent.
    Let $\balpha \in \R^K$.
    Then there exists $b_i \in \{0\}^d \times \Z^{n-d}$ and $c_j \in \{0\}^d \times \Z^{n-d}$ such that $\oX^{b_i} \oX^{c_j} \init_{v, \balpha}(\bg_j)_i \in \A_d$ for $i = 1, \ldots, K$ and $j = 1, \ldots, m$.
    See~\cite[Figure~27]{dong2024semigroup} for an illustration.
\end{lemma}

Suppose $v \in \Rns$ is such that $v = (0, \ldots, 0, v_{d+1}, \ldots, v_n)$ with $v_{d+1}, \ldots, v_n$ being $\Q$-linearly independent.
Let $\balpha \in \R^K$.
For each $j = 1, \ldots, m$, define $\init_{v, \balpha}^{d}(\bg_j) = (\init_{v, \balpha}^{d}(\bg_j)_1, \ldots, \init_{v, \balpha}^{d}(\bg_j)_K)$ where
\[
\init_{v, \balpha}^{d}(\bg_j)_i \coloneqq \oX^{b_i} \oX^{c_j} \init_{v, \balpha}(\bg_j)_i \in \A_d, \quad i = 1, \ldots, K.
\]
Here, $b_i$ and $c_j$ are defined as in Lemma~\ref{lem:ing}.
Note that the vectors $b_i, c_j \in \{0\}^d \times \Z^{n-d}$ are not necessarily uniquely determined. However, when $d, v, \balpha$ are fixed, the polynomials $\init_{v, \balpha}^{d}(\bg_j)_i$ are uniquely determined by $\bg_1, \ldots, \bg_m$.
In fact, by Lemma~\ref{lem:ing} each $\init_{v, \balpha}(\bg_j)_i$ can be uniquely written as $\oX^{s} \cdot p$ for some $\oX^{s} \in \R[X_{d+1}^{\pm}, \ldots, X_{n}^{\pm}]$ and $p \in \R[X_{1}^{\pm}, \ldots, X_{d}^{\pm}]$.
Therefore $\init_{v, \balpha}^{d}(\bg_j)_i$ is uniquely determined as the polynomial $p$ in the decomposition.



Define by $\init_{v, \balpha}^{d}(\mM)$ the $\A_d$-module generated by $\init_{v, \balpha}^{d}(\bg_1), \ldots, \init_{v, \balpha}^{d}(\bg_m) \in \A_d^K$:
\[
\init_{v, \balpha}^{d}(\mM) \coloneqq \sum_{j = 1}^m \A_d \cdot \init_{v, \balpha}^{d}(\bg_j) =\left\{\sum_{j = 1}^m p_j \cdot \init_{v, \balpha}^{d}(\bg_j) \;\middle|\; p_1, \ldots p_m \in \A_d \right\}.
\]

\begin{lemma}[{\cite[Lemma~6.9]{dong2024semigroup}}]\label{lem:dimred}
    Let $\bg_1, \ldots, \bg_m$ be a super Gr\"{o}bner basis of $\mM$.
    
    Let  $v = (0, \ldots, 0, v_{d+1}, \ldots, v_n)$ be such that $0 \leq d \leq n-1$ and $v_{d+1}, \ldots, v_n$ are $\Q$-linearly independent.
    Let $\balpha \in \left(\sum_{k = d+1}^n \Z v_k\right)^K$.
    Fix $\xi \in \Xi$.
    Denote $I'_{\xi} \coloneqq \{i \in I_{\xi} \mid \alpha_{i} = \min_{i' \in I_{\xi}} \alpha_{i'}\}, J'_{\xi} \coloneqq O_{v} \cup J$.
    Denote by $\pi_d \colon \Z^n \rightarrow \Z^d$ the projection onto the first $d$ coordinates.
    For every $u \in (\R^{d})^*$, define $O'_u \coloneqq \{i \in \{1, \ldots, K\} \mid \pi_d(a_i) \not\perp u\}$.
    Then the two following conditions are equivalent:
        \begin{enumerate}[label = (\roman*)]
            \item (Condition in \emph{\hyperref[item:locshift]{(LocInfShift)}}): There exists $\bff \in \mM$ such that $\init_{v, \balpha}\left(\bff\right) \in \ApK$ and
            \begin{equation}\label{eq:condirrshift}
            \left(O_{w} \cup J'_{\xi}\right) \cap M_{w}(I'_{\xi}, \init_{v, \balpha}(\bff)) \neq \emptyset \quad \text{ for every $w \in \Rns$}.
            \end{equation}
            \item We have $J'_{\xi} \cap I'_{\xi} \neq \emptyset$, and there exists $\bff^d \in \init_{v, \balpha}^{d}(\mM) \cap \left(\A_d^+\right)^K$, such that
            \begin{equation}\label{eq:condirr2}
            \left(O'_{u} \cup J'_{\xi}\right) \cap M_{u}(I'_{\xi}, \bff^d) \neq \emptyset \quad \text{ for every $u \in (\R^{d})^*$}.
            \end{equation}
        \end{enumerate}
    When $d = 0$, Property~\eqref{eq:condirr2} is considered trivially true.
\end{lemma}

\paragraph*{Dimension reduction: the general case.}
Having considered the special case where the vector $v \in \Rns$ in Condition~\hyperref[item:locshift]{(LocInfShift)} is of the form $(0, \ldots, 0, v_{d+1}, \ldots, v_n)$, we now consider the general case.
The key idea when dealing with the general case of $v \in \Rns$ is the following \emph{coordinate change}.

Given a matrix $A = (a_{ij})_{1 \leq i, j \leq n} \in \GL(n, \Z)$, define the new variables $X'_1, \ldots, X'_n$ where $X'_i \coloneqq X_1^{a_{i1}} X_2^{a_{i2}} \cdots X_n^{a_{in}}$.
Then 
\[
\R[X_1, \ldots, X_n] = \R[X'_1, \ldots, X'_n].
\]
In other words, we can define the ring automorphism
\[
\varphi_A \colon \A \rightarrow \A, \quad X_i \mapsto X_1^{a_{i1}} X_2^{a_{i2}} \cdots X_n^{a_{in}},
\]
such that $\varphi_A(\oX^{b}) = \oX^{b A}$.
The automorphism $\varphi_A$ extends entry-wise to $\A^K \rightarrow \A^K$.

For each $A \in \GL(n, \Z)$, denote by $A^{- \top}$ the inverse of its transpose.
Then $(v A^{-\top}) \cdot (b A) = v \cdot b$ for all $v \in \Rns, b \in \Z^n$.
Hence, for any $f \in \A,$ we have $\init_{v A^{-\top}}(\varphi_A(f)) = \varphi_A(\init_{v}(f))$, and for any $\bff \in \A^K, \xi \in \Xi,$ we have $M_v(I_{\xi}, \bff) = M_{v A^{-\top}}(I_{\xi}, \varphi_A(\bff))$.
Furthermore, if we replace the vectors $\ta_1, \ldots, \ta_K \in \Z^n$ by the vectors $\ta_1 A, \ldots, \ta_K A \in \Z^n$, then the set $O_v$ becomes $O_{v A^{-\top}}$.
It is easy to verify that if $\bg_1, \ldots, \bg_m$ is a super Gr\"{o}bner basis for $\mM$, then $\varphi_A(\bg_1), \ldots, \varphi_A(\bg_m)$ is still a super Gr\"{o}bner basis for $\varphi_A(\mM) \coloneqq \{\varphi_A(\bff) \mid \bff \in \mM\}$. 

Let $v \in \Rns$ and let $A \in \GL(n, \Z)$ be such that $v A^{-\top} = (0, \ldots, 0, v_{d+1}, \ldots, v_n)$ where $v_{d+1}, \ldots, v_n$ are $\Q$-linearly independent.
Then as in the previous section we define the module $\init_{v A^{-\top}, \balpha}^{d}(\varphi_A(\mM))$ to be the module generated by $\init_{v A^{-\top}, \balpha}^{d}(\varphi_A(\bg_1)), \ldots, \init_{v A^{-\top}, \balpha}^{d}(\varphi_A(\bg_m))$.

The above observation shows the following.
Fix $v \in \Rns$ in \hyperref[item:locshift]{(LocInfShift)} of Theorem~\ref{thm:locglob}.
Given any change of coordinates $A \in \GL(n, \Z)$, we can simultaneously (right-)multiply $A^{-\top}$ to $v$ and $A$ to all $a_1, \ldots, a_K$, while applying $\varphi_A$ to the super Gr\"{o}bner basis $\bg_1, \ldots, \bg_m$ of $\mM$.
Then the original properties \hyperref[item:locshift]{(LocInfShift)}(a)(b) are satisfied by $\bff$ if and only if they are satisfied by $\varphi_A(\bff)$ after the change of coordinates.
We will use this observation to reduce the general case for $v$ to the special case considered in the previous subsection.

\begin{fact}[{\cite[Fact~6.10]{dong2024semigroup}}]\label{fct:A}
    For every $v \in \Rns$, there exists $A \in \GL(n, \Z)$ such that $v A^{-\top} = (0, \ldots, 0, v_{d+1}, \ldots, v_n)$ with $v_{d+1}, \ldots, v_n$ being $\Q$-linearly independent.
\end{fact}

\begin{proposition}[{Generalization of~\cite[Proposition~6.11]{dong2024semigroup}}]\label{prop:shifttod}
    Condition~\emph{\hyperref[item:locshift]{(LocInfShift)}} of Proposition~\ref{prop:inftoshift} is equivalent to the following:
    \begin{enumerate}
        \item[2.] \label{item:locd} \emph{\textbf{(LocInfD):}}
        For every $v \in \Rns, A \in \GL(n, \Z),$ such that $v A^{-\top} = (0, \ldots, 0, v_{d+1}, \ldots, v_n)$, $0 \leq d \leq n-1$ with $v_{d+1}, \ldots, v_n$ being $\Q$-linearly independent,
        there exist $\balpha \in \left(\sum_{k = d+1}^n \Z v_k\right)^K$ and $\bff^d \in \init_{v A^{-\top}, \balpha}^{d}(\varphi_A(\mM))$ satisfying the following properties:
        \begin{enumerate}
            \item[(a)] $\bff^d \in \left(\A_{d}^+\right)^K$.
            \item[(b1)] For each $\xi \in \Xi$, denote $I'_{\xi} \coloneqq \{i \in I_{\xi} \mid \alpha_{i} = \min_{i' \in I_{\xi}} \alpha_{i'}\}, J'_{\xi} \coloneqq \left(O_{v} \cup J_{\xi}\right)$. We have
            \begin{equation}\label{eq:condirr}
                J'_{\xi} \cap I'_{\xi} \neq \emptyset \quad \text{ for all } \xi \in \Xi.
            \end{equation}
            (Note that this property depends only on $v$ and $\balpha$, but not on $\bff^d$.)
            \item[(b2)] Denote by $\pi_d \coloneqq \Z^n \rightarrow \Z^d$ the projection onto the first $d$ coordinates.
            For $u \in (\R^{d})^*$, define $O'_u \coloneqq \{i \in \{1, \ldots, K\} \mid \pi_d(\ta_i A) \not\perp u\}$, we have
            \begin{equation}\label{eq:condirrD}
            \left(O'_{u} \cup J'_{\xi}\right) \cap M_{u}(I'_{\xi}, \bff^d) \neq \emptyset \quad \text{ for every $u \in (\R^{d})^*, \xi \in \Xi$}.
            \end{equation}
        \end{enumerate}
    \end{enumerate}
    As in Lemma~\ref{lem:dimred}, Property~\eqref{eq:condirrD} is considered trivially true when $d = 0$.
\end{proposition}
\begin{proof}
    Fix a $v = (v_1, \ldots, v_n) \in \Rns$.
    Take any $A \in \GL(n, \Z)$ with $v A^{-\top} = (0, \ldots, 0, v'_{d+1}, \ldots, v'_n)$ such that $v'_{d+1}, \ldots, v'_n$ are $\Q$-linearly independent.
    Note that $\sum_{k = 1}^n \Z v_k = \sum_{k = d+1}^n \Z v'_k$ because $A \in \GL(n, \Z)$.
    Therefore, we can apply Lemma~\ref{lem:dimred} to the super Gr\"{o}bner basis $\varphi_A(\bg_1), \ldots, \varphi_A(\bg_m)$, the vector $v A^{-\top} = (0, \ldots, 0, v'_{d+1}, \ldots, v'_n)$ and the vectors $\ta_1 A, \ldots, \ta_K A \in \Z^n$.
    Lemma~\ref{lem:dimred} shows that there exist $\balpha \in \left(\sum_{k = d+1}^n \Z v'_k\right)^K$ and $\bff^d \in \init_{v A^{-\top}, \balpha}^{d}(\varphi_A(\mM))$ satisfying \hyperref[item:locd]{(LocInfD)}(a)(b1)(b2) if and only if there exists $\balpha \in \left(\sum_{k = 1}^n \Z v_k\right)^K$ and $\bff \in \mM$ satisfying \hyperref[item:locshift]{(LocInfShift)}(a)(b).
\end{proof}

\paragraph*{Computing cells \hyperref[item:loccell]{(LocInfCell)}.}
We further reduce the Condition~\hyperref[item:locd]{(LocInfD)} to a Condition~\hyperref[item:loccell]{(LocInfCell)} which consists of verifying a \emph{finite} number of $v \in \Rns$ for each coordinate-change matrix $A \in \GL(n, \Z)$.

Let $v \in \Rns, \balpha \in \R^K$.
Denote by $e_1, \ldots, e_K$ the canonical basis of the $\A$-module $\A^K$.
We introduce the new variables $T_1, \ldots, T_K$ and define an $\A$-module homomorphism
\[
\phi: \A^K \rightarrow \R[X_1^{\pm}, \ldots, X_n^{\pm}, T_1^{\pm}, \ldots, T_K^{\pm}], \quad \oX^{u} e_i \mapsto \oX^{u} T_i.
\]
We have $\phi(\init_{v, \balpha}(\bff)) = \init_{(v, \balpha)}(\phi(\bff))$ for every $\bff \in \A^K$.

As in the previous subsections let $\bg_1, \ldots, \bg_m$ be a super Gr\"{o}bner basis of $\mM$.
Since $\phi(\bg_i)$ is a polynomial in $\R[X_1^{\pm}, \ldots, X_n^{\pm}, T_1^{\pm}, \ldots, T_K^{\pm}]$, there exists a partition of $\Rns \times \R^K$ such that for any two directions in the same partition element the initial parts of $\phi(\bg_i)$ are the same.
Let $\mL_{\mM}$ be the common refinement of the partitions associated to the polynomials $\phi(\bg_1), \ldots, \phi(\bg_m)$.

From now on we use the term ``\emph{cell}'' to call elements of a given partition.
Fix $I_{\xi} \subseteq \{1, \ldots, K\}$.
There exists a partition $\mL_{I_{\xi}}$ of $\R^K$ such that for any two vectors $(\alpha_1, \ldots, \alpha_K)$, $(\alpha'_1, \ldots, \alpha'_K)$ in the same cell, we have $\alpha_i > \alpha_j \iff \alpha'_i > \alpha'_j$ and $\alpha_i < \alpha_j \iff \alpha'_i < \alpha'_j$ for all $i, j \in I_{\xi}$.
Define the partition $\mL'_{I_{\xi}} \coloneqq \Rns \times \mL_{I_{\xi}}$ of $\Rns \times \R^K$ where each cell is of the form $\Rns \times P, P \in \mL_{I_{\xi}}$.

There exists a partition $\mL_O$ of $\Rns$ such that any two vectors $v, v'$ in the same cell satisfy $v \perp \ta_i \iff v' \perp \ta_i$ for all $i \in \{1, \ldots, K\}$.
By subdividing $\mL_O$ we can suppose that each cell is a convex polyhedron.
Similar to the definition of $\mL'_{I_{\xi}}$, we define the partition $\mL'_O \coloneqq \mL_O \times \R^K$ of $\Rns \times \R^K$.

For any two partition $\mB, \mB'$ of the same set $S$, define $\mB \vee \mB'$ to be the partition of $S$ whose elements are of the form $B \cap B', B \in \mB, B' \in \mB'$.
Consider the partition $\mL$ of $\Rns \times \R^K$ defined by
\[
\mL \coloneqq \mL_{\mM} \vee \mL'_O \vee \left( \bigvee_{\xi \in \Xi} \mL'_{I_{\xi}}\right).
\]
We point out that the cells of $\mL$ are invariant under scaling by a positive real, meaning $x \in Q \iff r \cdot x \in Q$ for all cells $Q \in \mL$ and $r \in \Rpp$.

Let $\pi \colon \Rns \times \R^K \rightarrow \Rns, (v, \balpha) \mapsto v$ be the canonical projection.
For each $Q \in \mL$, define the two-element partition $\{\pi(Q), \Rns \setminus \pi(Q)\}$ of $\Rns$, and define
\[
\mP \coloneqq \bigvee_{Q \in \mL} \{\pi(Q), \Rns \setminus \pi(Q)\}.
\]
By this definition, take any $P \in \mP$ and $Q \in \mL$ with $\pi^{-1}(P) \cap Q \neq \emptyset$; then for $v, v' \in P$, there exists $\balpha \in \R^K$ with $(v, \balpha) \in Q$ if and only if there exists $\balpha' \in \R^K$ with $(v', \balpha') \in Q$.
See~\cite[Figure~28]{dong2024semigroup} for an illustration.

It is important to note that the partitions $\mL_{\mM}, \mL_{O}, \mL_{I_{\xi}}, \xi \in \Xi,$ are all defined using equalities and inequalities with \emph{rational} coefficients.
Also, each inequality is strict, so every cell $Q \in \mL$ and $P \in \mP$ is relatively open (a polyhedron is called relative open if it is open in the smallest linear space containing it).
In other words, each cell is defined by a combination of equalities and \emph{strict} inequalities.
We also point out that, like the cells of $\mL$, the cell of $\mP$ are invariant under scaling by a positive real, meaning $x \in P \iff r \cdot x \in P$ for all cells $P \in \mP$ and $r \in \Rpp$.
By the definition of $\mL$, we immediately obtain the following.

\begin{lemma}[{\cite[Lemma~6.12]{dong2024semigroup}}]\label{lem:compop1}
    Fix a change of coordinates $A \in \GL(n, \Z)$, two sets $I_{\xi}, J_{\xi} \subseteq \{1, \ldots, K\}$, and a cell $Q \in \mL A^{- \top}$. 
    Then the sets $I'_{\xi}, J'_{\xi}$ defined in \emph{\hyperref[item:locd]{(LocInfD)}}(b1) are effectively computable and do not depend on $v, \balpha$ as long as $(v A^{-\top}, \balpha) \in Q$.
    In particular, the Property \emph{\hyperref[item:locd]{(LocInfD)}}(b1) is either always true or always false for $v, \balpha, (v A^{-\top}, \balpha) \in Q$.
\end{lemma}

Let $Q \in \mL$. For $(v, \balpha), (v', \balpha') \in Q$, we have
\[
    \init_{v, \balpha}(\bg_j) = \init_{v', \balpha'}(\bg_j)
\]
for all $j = 1, \ldots, m$.
Thus, if $v = (0, \ldots, 0, v_{d+1}, \ldots, v_n)$ is such that $v_{d+1}, \ldots, v_n$ are $\Q$-linearly independent, then $\init^{d}_{v, \balpha}(\mM)$ depends only on the cell $Q \in \mL$ containing $(v, \balpha)$.
Hence, we can denote
\[
\init^{d}_{Q}(\bg_j) \coloneqq \init^{d}_{v, \balpha}(\bg_j), \quad j = 1, \ldots, m, \quad \init^{d}_{Q}(\mM) \coloneqq \init^{d}_{v, \balpha}(\mM), \quad \text{ where } (v, \balpha) \in Q.
\]

For any coordinate change $A \in \GL(n, \Z)$, we similarly define the partitions $\mL A^{-\top}$ and $\mP A^{-\top}$ based on the super Gr\"{o}bner basis $\varphi_A(\bg_1), \ldots, \varphi_A(\bg_m)$ and the vectors $\ta_1 A, \ldots, \ta_K A$.
In particular, each cell of $\mL A^{-\top}$ is of the form $Q \cdot diag(A^{-\top}, I_K), Q \in \mL$, and each cell of $\mP A^{-\top}$ is of the form $P \cdot A^{-\top}, P \in \mP$.
If $v A^{- \top} = (0, \ldots, 0, v_{d+1}, \ldots, v_n)$ is such that $v_{d+1}, \ldots, v_n$ are $\Q$-linearly independent, then $\init^{d}_{v A^{- \top}, \balpha}(\varphi_A(\mM))$ depends only on the cell $Q \in \mL A^{- \top}$ containing $(v A^{- \top}, \balpha)$.
Similarly, for $j = 1, \ldots, m$, we can denote
\[
\init^{d}_{Q}(\varphi_A(\bg_j)) \coloneqq \init^{d}_{v A^{- \top}, \balpha}(\varphi_A(\bg_j)), \; \init^{d}_{Q}(\varphi_A(\mM)) \coloneqq \init^{d}_{v A^{- \top}, \balpha}(\varphi_A(\mM)), \; \text{ where } (v A^{- \top}, \balpha) \in Q.
\]

The inputs in Theorem~\ref{thm:dec} are generators for modules $\mM$ over $\A = \R[X_1^{\pm}, \ldots, X_n^{\pm}]$, vectors $\ta_1, \ldots, \ta_K$ in $\Z^n$ and pairs of sets $I_{\xi}, J_{\xi}, \xi \in \Xi$.
Our strategy is to use induction on $n$ to prove Theorem~\ref{thm:dec}.
The base case $n = 0$ reduces to linear programming.
Indeed, when $n = 0$, $\A = \R, \A^+ = \Rpp$, the Property~\eqref{eq:deccond} is trivially true; and the problem becomes the following: given an $\R$-submodule $\mM$ of $\R^K$, decide whether $\mM \cap \Rpp^K$ contains an element.
Since the given generators of $\mM$ all have integer coefficients, this is decidable using linear programming.

The following lemma shows that a decision procedure for Theorem~\ref{thm:dec} with smaller $n$ can help us decide for a given cell $Q \in \mL$ if the module $\init_{Q}^{d}(\varphi_A(\mM))$ contains an $\bff^d$ satisfying the Properties \hyperref[item:locd]{(LocInfD)}(a) and (b2).

\begin{lemma}[{Generalization of~\cite[Lemma~6.13]{dong2024semigroup}}]\label{lem:compop}
    Fix a change of coordinates $A \in \GL(n, \Z)$, sets $I_{\xi}, J_{\xi} \subseteq \{1, \ldots, K\}$ for each $\xi \in \Xi$, and a number $0 \leq d \leq n-1$.  
    Suppose Theorem~\ref{thm:dec} is true for all $n_0$, $0 \leq n_0 \leq n - 1$.
    Fix a cell $Q \in \mL A^{- \top}$, let $I'_{\xi}, J'_{\xi}, \xi \in \Xi,$ be the sets defined in \emph{\hyperref[item:locd]{(LocInfD)}}(b1).
    We can decide whether the module $\init_{Q}^{d}(\varphi_A(\mM))$ contains an element $\bff^d$ satisfying the Properties \emph{\hyperref[item:locd]{(LocInfD)}}(a) and (b2).
\end{lemma}
\begin{proof}
    Suppose Theorem~\ref{thm:dec} is true for all $0 \leq n_0 \leq n - 1$. In particular it is true for $d \leq n - 1$.
    Fix a cell $Q \in \mL A^{- \top}$.
    

    In \hyperref[item:locd]{(LocInfD)}, the $\A_d$-submodule $\init_{v A^{- \top}, \balpha}^{d}(\varphi_A(\mM)) = \init_{Q}^{d}(\varphi_A(\mM))$ of $\A_d^K$ is generated by the elements $\init^{d}_{Q}(\varphi_A(\bg_j)), j = 1, \ldots, m$.
    Recall that $\pi_d \coloneqq \Z^n \rightarrow \Z^d$ denotes the projection onto the first $d$ coordinates.
    
    We then apply Theorem~\ref{thm:dec} the following way: replace $n$ by $d$; replace the elements $\bg_1, \ldots, \bg_m \in \A^K$ by the elements $\init^{d}_{Q}(\varphi_A(\bg_1)), \ldots, \init^{d}_{Q}(\varphi_A(\bg_m)) \in \A_d^K$; replace the vectors $\ta_1, \ldots, \ta_K \in \Z^n$ by the vectors $\pi_d(\ta_1 A), \ldots, \pi_d(\ta_K A) \in \Z^d$; and replace the sets $I_{\xi}, J_{\xi}, \xi \in \Xi,$ by the sets $I'_{\xi}, J'_{\xi}, \xi \in \Xi$. 
    Then Theorem~\ref{thm:dec} shows we can decide whether $\init_{Q}^{d}(\varphi_A(\mM))$ contains an element $\bff^d$ satisfying $\bff^d \in \left(\A_d^+\right)^K$ and 
    \begin{equation*}
            \left(O'_{u} \cup J'_{\xi}\right) \cap M_{u}(I'_{\xi}, \bff^d) \neq \emptyset, \quad \text{ for every } u \in \left(\R^d\right)^*, \xi \in \Xi.
    \end{equation*}
    These are exactly the Properties~\hyperref[item:locd]{(LocInfD)}(a) and (b2).
\end{proof}

Denote by $Op(A, d)$ the union of all cells $Q \in \mL A^{- \top}$ such that 
the Property \emph{\hyperref[item:locd]{(LocInfD)}}(b1) is true for $(v A^{-\top}, \balpha) \in Q$, and such that $\init_{Q}^{d}(\varphi_A(\mM))$ contains an element $\bff^d$ satisfying the Properties~\hyperref[item:locd]{(LocInfD)}(a)(b2).
By Lemma~\ref{lem:compop1} and \ref{lem:compop}, the set $Op(A, d)$ is effectively computable as a finite union of polyhedra defined over rational coefficients (supposing Theorem~\ref{thm:dec} is true for all $0 \leq n_0 \leq n - 1$).
See~\cite[Figure~29]{dong2024semigroup} for an illustration of $Op(A, d)$.

\begin{proposition}[{Generalization of~\cite[Proposition~6.14]{dong2024semigroup}}]\label{prop:dtocell}
    Condition~\emph{\hyperref[item:locd]{(LocInfD)}} of Proposition~\ref{prop:shifttod} is equivalent to the following:
    \begin{enumerate}
        \item[2.] \label{item:loccell} \emph{\textbf{(LocInfCell):}}
        For every $A \in \GL(n, \Z)$ and every number $0 \leq d \leq n-1$, the following is true:
        \begin{enumerate}
            \item[(a)] For every $v = (0, \ldots, 0, v_{d+1}, \ldots, v_n) \in \{0\}^d \times (\R^{n-d})^*$ with $v_{d+1}, \ldots, v_n$ being $\Q$-linearly independent, there exists $\balpha \in \left(\sum_{k = d+1}^n \Z v_k\right)^K$ with $(v, \balpha) \in Op(A, d)$.
        \end{enumerate}
    \end{enumerate}
\end{proposition}
\begin{proof}
This follows directly from the definition of $Op(A, d)$. 
\end{proof}

\begin{lemma}[{\cite[Lemma~6.15]{dong2024semigroup}}]\label{lem:checkcell}
    Given $A \in \GL(n, \Z), d \in \N$ and given $Op(A, d)$ as a finite union of polyhedra defined over rational coefficients, it is decidable whether the statement \emph{\hyperref[item:loccell]{(LocInfCell)}(a)} is true.
\end{lemma}

\paragraph*{Proving Theorem~\ref{thm:dec}: induction and a double procedure.}
We now give the full proof of Theorem~\ref{thm:dec}.
As in~\cite{dong2024semigroup}, the overall strategy is to use induction on $n$, while deciding the Conditions \hyperref[item:locr]{(LocR)} and \hyperref[item:locinf]{(LocInf)} from the local-global principle (Theorem~\ref{thm:locglob}).

\thmdec*
\begin{proof}
    We use induction on $n$.
    As remarked earlier, the base case $n = 0$ degenerates into linear programming (given an $\R$-submodule $\mM$ of $\R^K$, decide whether $\mM \cap \Rpp^K$ contains an element).
    Suppose we have a decision procedure for all $n_0 < n$, we now construct a procedure for $n$.
    
    By Theorem~\ref{thm:locglob} it suffices to decide whether the two conditions \hyperref[item:locr]{(LocR)} and \hyperref[item:locinf]{(LocInf)} are both satisfied.
    First we check if \hyperref[item:locr]{(LocR)} is true using Proposition~\ref{prop:decr}.
    If \hyperref[item:locr]{(LocR)} is false then we return False and conclude there is no $\bff \in \mM \cap \ApK$ satisfying \eqref{eq:deccond}.
    If \hyperref[item:locr]{(LocR)} is true then we proceed.

    We now run the two following procedures \emph{in parallel}:

    \begin{enumerate}
        \item \textbf{Procedure A:} We recursively enumerate all elements of the $\Z[X_1^{\pm}, \ldots, X_n^{\pm}]$-module:
        \[
        \tMZ \coloneqq \left\{\sum_{j = 1}^m h_j \cdot \bg_j \;\middle|\; h_1, \ldots, h_m \in \Z[X_1^{\pm}, \ldots, X_n^{\pm}]\right\}.
        \]
        For each element $\bff \in \tMZ$, check if $\bff$ is in $\ApK$ and if it satisfies Property~\eqref{eq:deccond}.
        This can be done in the following way:
        since the entries of $\bff$ contain finitely many monomials, it suffices to check Property~\eqref{eq:deccond} for a finite number of $v$.
        Indeed, since each of $f_1, \ldots, f_K$ has only finitely many monomials, there exists a partition $\mL_{\bff}$ of $\Rns$ such that for each cell $L \in \mL_{\bff}$ and for each $\xi \in \Xi$, the set $M_v(I_{\xi}, \bff)$ is the same for all $v \in L$.
        Furthermore, for each cell $L \in \mL_O$, the set $O_v$ is the same for all $v \in L$.
        Therefore, it suffices to check Property~\eqref{eq:deccond} for one vector $v$ in each cell of the partition $\mL_{\bff} \vee \mL_O$.
        This can be done in finite time for any given $\bff$.
        If some element $\bff \in \tMZ$ is in $\ApK$ and satisfies Property~\eqref{eq:deccond}, we stop the procedure and return True.
        
        \item \textbf{Procedure B:} We recursively enumerate all $A \in \GL(n, \Z)$ and $d \in \{0, 1, \ldots, n-1\}$.
        For each $A$ and $d$, compute $Op(A, d)$ using Lemma~\ref{lem:compop} and the induction hypothesis on $n$.
        Using Lemma~\ref{lem:checkcell}, we check if the statement~\hyperref[item:loccell]{(LocInfCell)}(a) from Proposition~\ref{prop:dtocell} is false.
        If for some $A, d$, the statement~\hyperref[item:loccell]{(LocInfCell)}(a) is false, then we stop the procedure and return False.
    \end{enumerate}
    We claim that one of the two above procedures must stop.
    
    Indeed, if $\mM$ contains an element of $\ApK$ satisfying Property~\eqref{eq:deccond}, then there exists an element $\bff \in \tMZ \cap \ApK$ that satisfies Property~\eqref{eq:deccond} (see Lemma~\ref{lem:M}).
    In this case, Procedure A terminates by finding an element $\bff$ of $\tMZ \cap \ApK$ that satisfies Property~\eqref{eq:deccond}.
    
    If $\mM$ does not contain an element of $\ApK$ satisfying Property~\eqref{eq:deccond}, then by Theorem~\ref{thm:locglob}, Condition~\hyperref[item:locinf]{(LocInf)} must be false (since we have already checked \hyperref[item:locr]{(LocR)} to be true).
    By the chain of Propositions~\ref{prop:inftoshift}, \ref{prop:shifttod} and \ref{prop:dtocell}, the statement~\hyperref[item:loccell]{(LocInfCell)}(a) must be false for some $A \in \GL(n, \Z)$ and $d \in \{0, 1, \ldots, n-1\}$.
    In this case, Procedure B terminates by finding $A \in \GL(n, \Z)$ and $d \in \{0, 1, \ldots, n-1\}$ where the statement~\hyperref[item:loccell]{(LocInfCell)}(a) is false.
    
    Therefore, by running Procedure A and Procedure B in parallel, we obtain an algorithm that always terminates for $n$.
\end{proof}


\end{document}